\documentclass{article}

\usepackage{microtype}
\usepackage{graphicx}
\usepackage{subfigure}
\usepackage{booktabs}

\usepackage[utf8]{inputenc} 
\usepackage[T1]{fontenc}    
\usepackage{hyperref}       
\usepackage{url}            
\usepackage{booktabs}       
\usepackage{amsfonts}       
\usepackage{nicefrac}       
\usepackage{microtype}      

\usepackage{mathtools}  
\usepackage{bbm}  
\usepackage{fancyhdr}

\usepackage{./shortcuts_mm}
\usepackage{amsmath}
\newcommand\numberthis{\addtocounter{equation}{1}\tag{\theequation}}
\usepackage[round]{natbib}
\usepackage{enumitem}
\usepackage{cleveref}

\hypersetup{
    colorlinks,
    linkcolor={red!50!black},
    citecolor={blue!50!black},
    urlcolor={blue!80!black}
}

\newtheorem{theorem}{Theorem}
\newtheorem{lemma}[theorem]{Lemma}
\newtheorem{proposition}[theorem]{Proposition}
\newtheorem{remark}[theorem]{Remark}
\newtheorem{corollary}[theorem]{Corollary}
\newtheorem{definition}[theorem]{Definition}
\newtheorem{assumption}[theorem]{Assumption}
\Crefname{assumption}{Assumption}{Assumptions}
\Crefname{prop}{Proposition}{Propositions}
\Crefname{equation}{Eq.}{}

\usepackage{thmtools,thm-restate}
\usepackage{stmaryrd}

\usepackage{bbm}
\usepackage{mathtools}  

\usepackage{algorithm}
\usepackage{algorithmic}

\renewcommand{\ts}[1]{}
\renewcommand{\mm}[1]{}
\renewcommand{\bm}[1]{}

\allowdisplaybreaks

\usepackage{hyperref}

\usepackage[accepted]{icml2020}

\icmltitlerunning{Dimension-free convergence rates for gradient Langevin dynamics in RKHS}

\begin{document}

\twocolumn[
\icmltitle{Dimension-free convergence rates for gradient Langevin dynamics in RKHS}



\icmlsetsymbol{equal}{*}

\begin{icmlauthorlist}
\icmlauthor{Boris Muzellec}{crest}
\icmlauthor{Kanji Sato}{uot}
\icmlauthor{Mathurin Massias}{inria}
\icmlauthor{Taiji Suzuki}{uot,riken}
\end{icmlauthorlist}

\icmlaffiliation{crest}{CREST, ENSAE, IP Paris, France}
\icmlaffiliation{uot}{The University of Tokyo}
\icmlaffiliation{riken}{Riken AIP}
\icmlaffiliation{inria}{INRIA, Université Paris-Saclay}



\vskip 0.3in
]



\printAffiliationsAndNotice{Work done while BM and MM were interning at RIKEN AIP.  TS was partially supported by JSPS Kakenhi (26280009, 15H05707 and
18H03201), Japan Digital Design and JST-CREST.}  

\begin{abstract}

Gradient Langevin dynamics (GLD) and stochastic GLD (SGLD) have attracted considerable attention lately, as a way to provide convergence guarantees in a non-convex setting. However, the known rates grow exponentially with the dimension of the space. In this work, we provide a convergence analysis of GLD and SGLD when the optimization space is an infinite dimensional Hilbert space.
More precisely, we derive non-asymptotic, dimension-free convergence rates for GLD/SGLD when performing regularized non-convex optimization in a reproducing kernel Hilbert space.
Amongst others, the convergence analysis relies on the properties of a stochastic differential equation, its discrete time Galerkin approximation and the geometric ergodicity of the associated Markov chains.

\end{abstract}


\section*{Introduction}
\label{sec:intro}

Convex, finite-dimensional optimization problems have been studied at length, and there exists a variety of well-understood algorithms to solve them efficiently \citep{Nesterov83,Nesterov04,Hiriart-Urruty_Lemarechal93,Boyd_Vandenberghe04,Nocedal_Wright06}.
In the non-convex case however, these methods are only guaranteed to converge to stationary points of the objective function.
This is to be contrasted with the ubiquity of the non-convex case, largely due to the successes of deep learning methods, for which optimization methods with good empirical behavior are widely used \citep{Robbins_Monro51,Duchi_Hazan_Singer12,Zeiler12,Kingma_Ba14}.
In a different perspective, stochastic gradient Langevin dynamics (SGLD), which can be seen as stochastic gradient descent methods with additive Gaussian noise injection at each iteration, was introduced by \citet{Welling_Teh11}. In the case of a strongly convex objective function $\cL$, recent studies \citep{Dalalyan17b} highlighted the connections between sampling from log-concave densities $f(x) \propto \exp(-\beta\cL(x))$ concentrated around the minimum of $\cL$, and minimizing $\cL$.
Such distributions can be obtained as the stationary distributions of first order Langevin dynamics
\begin{equation}\label{eq:first_order_langevin}
    \dX(t) = -\nabla \cL(X(t))\dt + \sqrt{2\beta^{-1}} \dd B(t) ,
\end{equation}
where $\{B(t)\}_{t \geq 0}$ is the standard Brownian motion in $\bbR^d$ and $\beta > 0$ is the inverse temperature.
\citet{Chiang_Hwang_Sheu87,Gelfand_Mitter91,Roberts_Tweedie96} studied the convergence of $X(t)$ to the stationary Gibbs distribution $\pi(\dx) \propto \exp(-\beta \cL (x))$, and the concentration of the samples around the global minimum, while more recently \citet{Dalalyan17a,Durmus_Moulines16,Durmus_Moulines17} analyzed the convergence rates of discrete-time Langevin updates for sampling from log-concave densities.

Recent studies have shown that Langevin dynamics based algorithms converge near a global minimum of $\cL$, even when $\cL$ is not convex, provided $\cL$ is \emph{dissipative} and has Lipschitz gradient.
The analysis relies on the connection between the iterates of Langevin dynamics based algorithms and the Markov chain solution of the continuous time Langevin equation, which admits the Gibbs measure as invariant distribution.
\citet{Raginsky_Rakhlin_Telgarsky2017} provided a non asymptotic convergence rate in expectation to an \textit{almost minimizer}\footnote{a point within distance $\cO(d \log(1 + \beta)/\beta)$ to the true minimizer, in finite dimension $d$} for stochastic gradient Langevin dynamics (SGLD), which \citet{Xu_Chen_Zou_Gu18} improved while also providing an extension to variance-reduced algorithms.
In an alternative approach, \citet{Zhang_Liang_Charikar17} provided bounds on the \emph{hitting time} of SGLD to neighborhoods of  local minima.

However, these results only apply to finite-dimensional optimization, with rates growing polynomially or even exponentially with the dimension. In this paper, we study the rate of convergence when applying Langevin dynamics algorithms in {\it infinite dimension}. More precisely, we bound the probability of GLD of reaching prescribed level sets of the objective functional $\cL$ at iteration $m$. To our knowledge, this is the first application of GLD for infinite-dimensional non-convex optimization. 

Our results rely on assumptions which are classical in the GLD/SGLD literature, and in the literature of approximation of invariant laws of stochastic partial differential equations (SPDE) in infinite dimension. 
In particular, we leverage the weak approximation error of the discrete time scheme of SPDEs analyzed by  \citet{Brehier14,Brehier16} for general inverse parameter $\beta > 1$, where \citet{debussche2011weak,WANG2013151,Andersson2015} gave discretization error non-uniformly over the time horizon, and utilize the geometric ergodicity of continuous time dynamics \citep{Jacquot+Gilles:1995,Goldys_Maslowski06}.
Compared with \Cref{eq:first_order_langevin}, results in the infinite-dimensional setting usually
involve a linear operator acting as a regularizer and whose spectrum ``replaces'' dimension in the convergence rates. More specifically, our contribution can be summarized as follows:
{   
\setlength{\leftmargini}{0.4cm} 
\vspace{-0.2cm}
\begin{itemize}[itemsep=-0cm]
    \item We give a non-asymptotic error bound of the infinite dimensional GLD/SGLD implemented with a spectral Galerkin method, which has an explicit dependency on $\beta$ and is uniform over all time horizons.
    \item For that purpose, the geometric ergodicity of the time discretized dynamics is proven, which is known to be non-trivial.
    \item We give an upper bound of the distance between the expected objective value under the invariant measure and the global optimal solution in the infinite dimensional setting.
\end{itemize}
}



\section{Notation and Framework}

\subsection{Notation and background on RKHS}

Let $(\cH, \scal{\cdot}{\cdot})$ be a Hilbert space. 
We will also use the notation $\|\cdot\|_{\cH}$ to explicitly indicate the norm $\|\cdot\|$ is of $\cH$.
If $\phi$ and $V$ are two functions from $\cH$ to $\bbR$ such that for all $x \in \cH$, $|\phi(x)| \leq |V(x)|$, we write $\norm{\phi}_V \leq 1$. $C_b^2$ is the set of bounded, twice continuously Fr\'echet differentiable functions with bounded first and second derivatives. 
We denote by $\mathcal{B}(\cH)$ the set of bounded linear operators from $\cH$ to $\cH$ and  $\|\cdot\|_{\mathcal{B}(\cH)}$ denotes the operator norm.
For a discrete or continuous Markov chain $\{X_t\}$, note $\bbE_x[\cdot] \eqdef \bbE[\cdot \mid X_0 = x]$.


Let $\cH_K \subset \cH$ be a Reproducing Kernel Hilbert Space (RKHS), with reproducing kernel $K$.
By Mercer's theorem 
\citep{Mercer1909, Steinwart2012}, $\cHK$ can be described as:
\begin{equation}\label{eq:mercer}
  \cHK = \condset{\sum_{k=0}^{\infty} \alpha_k f_k}{\sum_{k=0}^{\infty} \frac{\alpha_k^2}{\mu_k} < \infty}   ,
\end{equation}
where the $\mu_k$'s and $f_k$'s are the eigenvalues (in decreasing order) and corresponding eigenfunctions of $T_K$, the integral operator with kernel $K$ for a measure $\rho$:
%
\begin{equation}\label{eq:integral_operator}
  T_K f_k (x) \eqdef \int K(x, y) f_k(y) \dd\rho(y)  = \mu_k f_k (x)  ,
\end{equation}
 and the $f_k$'s form an orthonormal system in $\cH$. 
 Therefore, in general $\cHK \subset \cH = \condset{\sum_{k=0}^{\infty} \alpha_k f_k}{\sum_{k = 0}^{\infty} \alpha_k^2 < \infty}$.
 Hence, we are working in two different geometries: if $f = \sum_{k \geq 0} \alpha_k f_k$, and $g = \sum_{k \geq 0} \beta_k f_k$, then $\cH$ is equipped with the inner product $\scal{f}{g} = {\sum_{k=0}^{\infty} \alpha_k\beta_k}$, and $\cHK$ is equipped with the inner product $\scal{f}{g}_\cHK = \sum_{k=0}^{\infty} \frac{\alpha_k \beta_k}{\mu_k}$. 
The norm in $\cHK$ induced by the inner  product $\scal{\cdot}{\cdot}_{\cHK}$ is denoted by $\|\cdot\|_{\cHK}$.
Unless denoted by the $\cHK$ subscript, we will work in the geometry of $\cH$.

In the following, for $L: \cH \rightarrow \bbR$, the gradient $\nabla L (x)$ is defined as the Riesz representor of the Fr\'echet derivative of $L$, $DL(x)$ (i.e., the unique vector satisfying $\forall h, L(x + h) = L(x) + \scal{ \nabla L (x)}{h} + O(\norm{h}^2)$). We will identify $n$-order derivatives with $n$th-linear forms, and with vectors when there is no ambiguity (e.g., we write $D^3 L(x) \cdot (h, k)$ for the Riesz representor of $l\in \cH \mapsto D^3 L(x) \cdot (h, k, l)$).


\subsection{Algorithm: gradient Langevin dynamics} 
We consider the following optimization problem:
\begin{problem}\label{pb:rkhs}
 \min_{x \in \cH} \cL(x)
  \eqdef L(x) + \frac{\lambda}{2} \normin{x}_\cHK^2  ,
\end{problem}
where $\lambda > 0$ and $L$ is potentially non convex.
Assuming $L$ admits at least one global minimizer, we note
\begin{align}
  &x^* \eqdef \argmin_{x \in \cH} L(x)  ,\\
  &\tilde{x} \eqdef \argmin_{x \in \cH} L(x) + \frac{\lambda}{2} \normin{x}_\cHK^2 \label{eq:x_tilde}.
\end{align}
In this study, we treat $\lambda > 0$ as a constant and assume that $L(x^*)$ and $L(\tilde{x})$ are sufficiently close.
The difference between these two quantities is extensively studied, for example, in the least squares estimation problem in RKHS \cite{Caponnetto_DeVito07}. 

We study the gradient Langevin dynamics (GLD) iterations to solve \Cref{pb:rkhs}.
To define GLD, we need to make a heavy use of the infinite dimensional Brownian motion.
\begin{definition}[Cylindrical Brownian motion/Wiener process]\label{def:cylindrical}
  Given
  \begin{itemize}[parsep=0pt,partopsep=0pt,topsep=0pt,itemindent=-0.3cm]
    \item a complete orthonormal system of $\cH$, $(e_i)_{i \in I}$, where $I \subset \bbN$,
    \item a family $(\{W^i(t)\}_{t \geq 0})_{i \in I}$ of independent real Brownian motions,
  \end{itemize}
  then $\{W(t)\}_{t \geq 0} \eqdef \{\sum_{i \in I} W^i(t) e_i\}_{t \geq 0}$ is called a \emph{cylindrical Brownian motion}. 
\end{definition}
Then, GLD updates are defined as follows:
\begin{align}\label{eq:sgld}
  \begin{cases}
      X_0 = x^0 \in \cH  , \\
      X_{n + 1} = S_\eta X_n - \eta S_\eta \nabla L(X_n) + \sqrt{2 \tfrac{\eta}{\beta}}S_\eta \varepsilon_n,
  \end{cases}
\end{align}
where $\eta > 0$ is the stepsize, $\beta \geq \eta$ 
is the inverse temperature parameter, the variables $\varepsilon_k$ are i.i.d. cylindrical standard Gaussian and $S_\eta \eqdef (\Id + \eta\frac{\lambda}{2}\nabla\norm{\cdot}_\cHK^2)^{-1}$. 
A crucial analysis tool is to see \Cref{eq:sgld} as a time discretization of the following SPDE \cite{da_prato_zabczyk_1996}:
\begin{align}\label{eq:langevin_sde}
  \begin{cases}
      X(0) = x_0   , \\
      \dX(t) = -\nabla \cL (X(t)) + \sqrt{\tfrac{2}{\beta}} \dW(t)  ,
  \end{cases}
\end{align}
where $\{W(t)\}_{t \geq 0}$ is a cylindrical Brownian motion (\Cref{def:cylindrical}).
We refer to \citet{da_prato_zabczyk_1996} for the existence of solutions, its regularity conditions and related mathematical details.
Note that the scheme \Cref{eq:sgld} is semi-implicit: applying $(S_\eta)^{-1}$ reads
\begin{equation*}
    X_{n + 1} =  X_n - \eta( \nabla L(X_n) + \tfrac{\lambda}{2}\nabla\norm{X_{n + 1}}_\cHK^2 )+ \sqrt{2 \tfrac{\eta}{\beta}} \varepsilon_n  .
\end{equation*}
\paragraph{Approximated computation}
Strictly speaking, the infinite dimensional GLD scheme presented above is computationally intractable.
The {\it Galerkin approximation method} projects the dynamics to a finite dimensional subspace to make them computationally feasible. 
Let $\cH_N$ be an $N+1$-dimensional subspace of $\cH$ that is spanned by $(f_k)_{k=0}^{N-1}$: $\cH_N \triangleq \mathrm{Span}\{f_k \mid k=0,\dots,N\}$.
Let $P_N: \cH \to \cH_N$ be the orthogonal projection operator onto $\cH_N$: $P_N (\sum_{k=0}^\infty \alpha_k f_k) = \sum_{k=0}^N \alpha_k f_k$.
Then, the GLD with Galerkin approximation can be formulated as 
\begin{align}\label{eq:GalerkinTimeDiscrete}
    X_{n + 1}^N = S_\eta \left(X_n^N - \eta \nabla L_N(X_n^N) + \sqrt{2 \tfrac{\eta}{\beta}} P_N \varepsilon_n\right),
\end{align}
where $X_0^N = P_N x_0 \in \cH_N$ and $\nabla L_N(x) \eqdef P_N (\nabla L(P_N x))$.
Since this scheme is essentially finite dimensional, it can be implemented in practice. 

Next, we consider a stochastic gradient variant of GLD (stochastic GLD; SGLD).
Let us consider a finite sum risk minimization setting where
\begin{align*}
L(x) = \frac{1}{\ntr} \sum_{i=1}^{\ntr} \ell_i(x),
\end{align*}
for $\ell_i: \cH \to \bbR$ which is Fr\'echet differentiable\footnote{We may generalize the setting to a situation where $\nabla L(x) = \bbE_{\xi}[g(x,\xi)]$ with a stochastic gradient $g(\cdot,\xi)$ in a straightforward way.}.
SGLD makes use of a mini-batch of stochastic gradients \cite{Welling_Teh11} instead of the full gradient $\nabla L(x)$: 
$g_n(x) = \frac{1}{\nbch} \sum_{i \in I_n} \nabla \ell_i(x)$ where $I_n$ is a random subset of $\{1,\dots,N\}$ chosen uniformly at random and $\nbch = |I_n|$. Then, its update rule is given by 
\begin{align}
Y_{n + 1}^N = S_\eta \left(Y_n^N - \eta g_{n,N}(Y_n^N) + \sqrt{2 \tfrac{\eta}{\beta}} P_N \varepsilon_n\right),
\end{align}
where $g_{n,N}(x) \eqdef P_N (g_n(P_N x))$ and $Y_0^N = P_N x_0 \in \cH_N$. 
These approximation techniques significantly reduce the computational cost.

\subsection{Assumptions}
Our goal is to study the convergence of the iterations \Cref{eq:sgld}, i.e., to bound $L(X_n) - L(x^*)$ with high probability.
For this, we need to make assumptions on the RKHS $\cHK$ and on $L$.
We first make the following assumption on $\cHK$, independently of the objective $L$: 
 \begin{assumption}\label{assum:eigenvalue_cvg}
     It holds that:
     \begin{equation}\label{eq:mukDecayRate}
       \mu_k \sim \frac{1}{k^2}.
     \end{equation}
 \end{assumption}
 We note that a finite dimensional situation is also allowed, i.e., $\mu_k = 0~(\forall k \geq k_0)$ for some $k_0 \in \bbN$, as long as \Cref{eq:mukDecayRate} is satisfied for $k \leq k_0$.
 The weaker assumption $\mu_k \sim k^{-p}$ with $p > 1$ is sometimes made in the literature (\citealt[Definition 1.\emph{iii})]{Caponnetto_DeVito07}, \citealt{Steinwart_Christmann08}), but the numerical approximation result used in \Cref{sec:first_term} requires the more restrictive $p=2$ assumption (\citealt[Assumption 2.2 (2)]{Brehier16}). As an example, one can consider the case where $\cH$ itself is an RKHS for a kernel $K'$, with Mercer decomposition $K'(x, y) = \sum_k \nu_k g_k(x)g_k(y)$. Then, the ``rescaled'' kernel $K(x, y) = \sum_k \mu_k \nu_k g_k(x)g_k(y)$ with $\mu_k \sim \frac{1}{k^2}$ satisfies \Cref{assum:eigenvalue_cvg}. In fact, the role of \Cref{assum:eigenvalue_cvg} is to ensure that the trajectories \eqref{eq:sgld},  \eqref{eq:langevin_sde} will remain in the support of the Gaussian process corresponding to the kernel $K$. 
 %

Next, we put assumptions on the objective function $L$. The first one is classical for gradient-based optimization \cite{Nesterov04}.
\begin{assumption}[Smoothness]\label{assum:smoothness}
   $L$ is $M$-smooth:
  \begin{equation}\label{eq:smoothness}
     \forall x, y\in \cH,\quad \norm{\nabla L(x) - \nabla L(y)} \leq M \norm{x - y} .
  \end{equation}
\end{assumption}

In view of \Cref{eq:mercer}, we have that $A \eqdef - \frac{\lambda}{2} \nabla \normin{\cdot}_\cHK^2$ is a diagonal operator, characterized by $A f_k = - \frac{\lambda}{\mu_k} f_k$. 
The following assumptions enforce more smoothness on $L$ w.r.t. a norm induced by $A$ through its second and third order derivatives.
%
\ts{I added this new condition}
\begin{assumption}\label{assum:C1_boundedness}
There exists $\alpha \in (1/4,1)$ and $\lambda_0, C_{\alpha,2} \in (0,\infty)$ such that $\forall x,h,k \in \cH$, 
$$
|D^2 L(x) \cdot (h,k) | 
\leq C_{\alpha,2} \|h\|_{\cH} \|k\|_{\alpha},
$$
where $\norm{x}_\varepsilon \eqdef \left(\sum_{k \geq 0} (\mu_k)^{2\varepsilon} |\scal{x}{f_k}|^2\right)^{1/2}$.
\end{assumption}
This assumption is not standard in the previous works. 
However, we put this assumption so that the time discretized dynamics satisfies geometric ergodicity.
Fortunately, this assumption is not restrictive in machine learning applications
(see \Cref{sec:Examples} for details).

The next one is common in the SPDE discretization literature ({\citet[Assumption 2.7]{Brehier16}}, {\citet[Assumption (2.3)]{debussche2011weak}}). It is used in \Cref{sec:second_term} to obtain the convergence of the stationary distribution $\mu^\eta$ of \Cref{eq:sgld} to that of \Cref{eq:langevin_sde} as $\eta$ goes to zero.
\begin{assumption}[{\citet[Assumption 2.7, $M\rightarrow \infty$]{Brehier16}}]\label{assum:C2_boundedness}
Let $L_N: \cH_N \to \bbR$ be $L_N = L(P_N x)$. $L$ is three times differentiable, and there exists $\alpha' \in [0, 1), C_{\alpha'} \in (0, \infty)$ such that for all $N \in \bbN$ and $\forall x, h, k \in \cH_N,$
%
\begin{align*}
    & \norm{D^3 L_N(x) \cdot (h, k)}_{\alpha'} \leq C_{\alpha'} \norm{h}_0\norm{k}_0,\\
    & \norm{D^3 L_N(x) \cdot (h, k)}_{0} \leq C_{\alpha'} \norm{h}_{-\alpha'}\norm{k}_0.
\end{align*}
\end{assumption}
%
As an example, \Cref{assum:C2_boundedness} is satisfied with $\alpha = 0$ when $L$ is $C^3$ with bounded second and third-order derivatives.
Next, we assume the following condition to ensure the dissipativity (Proposition \ref{prop:dissipative}) which is essential to show geometric ergodicity.
\begin{assumption}\label{assum:dissipative}
    It either holds that
    \begin{assumenum}
        \item  $\lambda  > M\mu_0$ (Strict Dissipativity), or  \label{assum:strict_diss}
        \item  $\norm{\nabla L(\cdot)} \leq B, \quad B >0$ (Bounded gradients). \label{assum:bounded_grad}
    \end{assumenum}
\end{assumption}
The $C_0$-semigroup $(S_t)_{t \geq 0}$ generated by $A$ is the one of diagonal operators determined by $S_t f_k = e^{-\lambda t / \mu_k} f_k$.
It is easy to check that this semigroup is strongly continuous.
Therefore, the Langevin SDE \eqref{eq:langevin_sde} is an instance of the more general semilinear SDE:
\begin{equation}\label{eq:sde_goldys}
  \dX(t) = \Big(A X(t) + F(X(t)) \Big) \dt + \sqrt{Q} \dW(t),
\end{equation}
where $F$ is globally $M$-Lipschitz, $Q$ is bounded and symmetrical and $A$ is a linear unbounded operator on $\cH$ generating a strongly continuous semigroup.
For the SDE \Cref{eq:langevin_sde}, we have $F = - \nabla L $, $Q = 2\beta^{-1} \Id$ and $A = -\tfrac{\lambda}{2} \nabla \normin{\cdot}_\cHK^2$.\bm{emphasize on the fact that this makes $S_\eta$ diagonal in the $f_k$ basis}
The SDE \eqref{eq:sde_goldys} has been extensively studied in finite dimension \citep{Khasminskii11}; in the infinite dimensional case, several results have been shown such as the existence and uniqueness of its invariant measure \citep{DaPrato_Zabczyk92,Maslowski:1989,Sowers:1992}, the exponential convergence of the time $t$ distribution to this invariant measure \citep{Jacquot+Gilles:1995,Shardlow:1999,Hairer:2002} and its explicit convergence rate evaluation \citep{Goldys_Maslowski06}; the invariant measure $\pi$ is given by 
$$
\textstyle 
\frac{\dd \pi}{\dd \nu_\beta}(x) \propto \exp(- \beta L(x)), 
$$
where $\nu_\beta$ is the Gaussian measure in $\cH$ with mean 0 and covariance $(- \beta A)^{-1}$ (see \citet{da_prato_zabczyk_1996} for the precise definition of infinite dimensional Gaussian measures).
If these assumptions are verified, we have a weaker condition than strong convexity: dissipativity.\bm{Explain what dissipativity implies. Namely, if we are far from the origin, then on average the gradients will bring us towards the origin.}\bm{add references to where the dissipativity is used}
\begin{proposition}[Dissipativity]\label{prop:dissipative}
    Under \Cref{assum:smoothness,assum:eigenvalue_cvg} and \Cref{assum:strict_diss} or \Cref{assum:bounded_grad}, there exists constants $m, c > 0$ verifying
 \begin{equation}\label{eq:dissipative}
  \forall x\in \cH,  \scal{Ax - \nabla L(x)}{x} \leq - m \norm{x}^2 + c
   .
\end{equation}
\end{proposition}

The dissipative condition proved in this proposition is quite standard to show the existence of the invariant law.
For example, \citet{Raginsky_Rakhlin_Telgarsky2017,Xu_Chen_Zou_Gu18} showed the convergence to the invariant law under the dissipative condition in the finite dimensional situation. This condition intuitively indicates that the dynamics stays inside a bounded domain in high probability.
If $X_n$ (or $X(t)$) is far away from the origin, then the dynamics are forced to get back around the origin. Thanks to this condition, the dynamics can possess finite moments, which is important to ensure the existence of an invariant law. 

In fact, a result of \Cref{assum:dissipative} is that there exits at least one invariant law.
\begin{proposition}
Under \Cref{assum:dissipative}, the processes $\{X(t)\}_{t \geq 0}$ and $\{X_n\}_{n \in \bbN_+}$ admit (at least) an invariant law.
\end{proposition}
The proof can be found for example Proposition 4.1 of \citet{Brehier16}, which utilizes the Krylov-Bogoliubov criterion \citep[Section 3.1]{da_prato_zabczyk_1996}.
This proposition does not indicates the {\it uniqueness} of an invariant law.
It is shown that the continuous time dynamics $X(t)$ has a unique invariant law and is geometrically ergodic.
As for the discrete time dynamics $X_n$, the uniqueness of the invariant law is already well-known under the strict dissipative condition (\Cref{assum:strict_diss}) (see \citet{Brehier16} for example).
However, the uniqueness has not been shown under the bounded gradient condition (\Cref{assum:bounded_grad}).
In \Cref{sec:first_term}, we will show that the uniqueness also holds under \Cref{assum:bounded_grad} if we assume \Cref{assum:C1_boundedness}, which has not been assumed in previous work.
%
%

%
Finally, in the SGLD setting we put the following stronger assumption on each $\ell_i$.
\begin{assumption}\label{ass:SGLD_cond}
Each $\ell_i$ satisfies \Cref{assum:smoothness,assum:C1_boundedness,assum:C2_boundedness} and \Cref{assum:bounded_grad} instead of $L$,
where the constants in each assumption are uniform over all $\ell_i~(i=1,\dots,\ntr)$.
\end{assumption}

\subsection{Motivation of problem settings}

\label{sec:Examples}
As examples, \Cref{assum:smoothness,assum:C2_boundedness} encompass classification cases (e.g., logistic regression) and ordinary least squares regression, among others.
In the non-convex setting, 
examples include deep learning, 
tensor factorization \citep{signoretto2013learning,suzuki2016} and 
robust classification using non-convex losses such as Savage~\citep{Masnadi2008}.

For the sake of instructive exposition, let us consider a situation where we observe $n$ input-output pairs $(z_i,y_i)_{i=1}^n$, where $z_i \in \cZ$ is an input and $y_i\in \cY$ is the corresponding label. 
Here, we let $\cH$ be a Hilbert space of functions on $\cZ$ (which could be a RKHS) with complete orthonormal system $(f_k)_{k=0}^\infty$.
Accordingly, we define a loss function $\ell(\cdot,y_i) = \ell_i(\cdot): \bbR \to \bbR$ for the $i$-th observation, and consider an empirical risk: $\tilde{L}(f) = \frac{1}{n} \sum_{i=1}^n \ell_i(f(z_i))$ for a function $f: \cZ \to \bbR$.
From the expression \eqref{eq:mercer}, the (sub-)RKHS $\cH_K$ can be expressed as an image of $T_K^{1/2}$, i.e., $\cH_{K} = \{f = T_K^{1/2} h\mid  h \in \cH\}$ and $\|f\|_{\cHK} = \inf_{h \in \cH: f = T_K^{1/2}h}\|h\|_{\cH}$.
More generally, we define an RKHS $\cH_{K^\gamma}$ for $0 < \gamma$ as an image of $T_K^{\frac{\gamma}{2}}$:
$\cH_{K^\gamma} = \{f = T_K^{\gamma/2} h\mid  h \in \cH\}$. We see that $\gamma=1$ corresponds to $\cH_K$.
We employ $\cH_{K^{\gamma}}$ as a model for $f$ and let the corresponding empirical risk be 
$
L(x) = \tilde{L}(T_K^{\frac{\gamma}{2}} x) 
$
(if needed, we may add a smooth regularization term).
Note that, for $x = \sum_{k=0}^\infty \alpha_k f_k \in \cH$, $T_K^{\frac{\gamma}{2}} x (z) = \sum_{k=0}^\infty \mu_k^{\frac{\gamma}{2}} \alpha_k f_k(z)$,
and thus we can obtain a reproducing formula $T_K^{\frac{\gamma}{2}} x (z) = \langle x, \psi_{{\gamma}}(z) \rangle_{\cH}$ where $\psi_{{\gamma}}(z) \eqdef \sum_{k=0}^\infty \mu_k^{\frac{\gamma}{2}} f_k(z)f_k$.
$\psi_\gamma$ defines the kernel function of $\cH_{K^\gamma}$ as $K_\gamma(z,z') = \langle \psi_\gamma(z), \psi_\gamma(z')\rangle_{\cH} = \sum_{k=0}^K  \mu_k^{\gamma} f_k(z)f_k(z')$. Using this, we see that $\|\psi_\gamma(z)\|_\cH^2 = \sum_{k=0}^\infty \mu_k^{\gamma} f_k^2(z) = K_\gamma(z,z)$ and $\|\psi_\gamma(z)\|_\epsilon^2 = \sum_{k=0}^\infty \mu_k^{\gamma + 2\epsilon} f_k^2(z) = K_{\gamma + 2\epsilon}(z,z)$.
In this situation, if we have $\max_{i}\sup_{u} |\ell''_i(u)| \leq G$ and $\sup_{z \in \cZ}K_\gamma(z,z) \leq R_\gamma$ for $G,R_\gamma > 0$, then 
\begin{align}
& \|\nabla L(x) - \nabla L(x')\| 
\leq GR_\gamma  \| x - x' \|_{\cH}, 
\label{eq:LipschitzGradBound}
\\
& \textstyle 
|D^2 L(x) \cdot (h,k)|
 \leq  G   \sqrt{R_\gamma \sum_{k=0}^\infty \mu_k^{\gamma - 2\alpha}},
\label{eq:D2bound}
\end{align}
for $x,h,k \in \cH$ with $\|h\|=1$ and $\|k\|_\alpha = 1$.
The proof of these inequalities is given in Appendix \ref{sec:ProofOfExampleSmoothBound}.
Therefore, Assumptions \ref{assum:smoothness} and \ref{assum:C1_boundedness} are satisfied as long as 
$R_\gamma < \infty$ for $\gamma > 1$ because the condition $\mu_k \lesssim 1/k^2$ makes the right hand of \Cref{eq:D2bound} finite by setting $\alpha = (\gamma -1)/2 + 1/4 > 1/4$. 
\Cref{assum:C2_boundedness} is also verified in the same manner.

Finally, if we let $f = T_{K}^{\frac{\gamma}{2}} x$, then $\|x\|_{\cHK} = \|f\|_{\cH_{K^{1+\gamma}}}$ holds. 
Then, it follows that 
$$
  L(x) + \lambda \|x\|_{\cHK}^2 = \tilde{L}(f)  + \lambda \|f\|^2_{\cH_{K^{1+\gamma}}}.
$$
Therefore, we see that our formulation covers a wide range of kernel regularization learning by adjusting $\gamma$ appropriately.

We would like to remark that we may deal with a situation where $L(x)$ contains a regularization term $\frac{\lambda_0}{2}\|x\|^2$ like $L(x)=\hat{L}(x) + \frac{\lambda_0}{2}\|x\|^2$.
To deal with this situation, we should change the algorithm and analysis a little bit because it could violate \Cref{assum:dissipative} (especially, the bounded gradient condition). 
See \Cref{sec:RemarkRegularization} for more details about how to deal with this setting.

\section{Main Result}

Here, we give our main result on the the non-asymptotic error bound of the GLD algorithm.
Define a constant $\cbeta$ as 
\begin{align*}
\cbeta = 
\begin{cases} 1~~~& (\text{strict dissipativity: \Cref{assum:strict_diss})}), \\ \sqrt{\beta}~~~& (\text{bounded gradient: \Cref{assum:bounded_grad}}). 
\end{cases}
\end{align*}

\begin{theorem}[Main Result, GLD convergence rate]\label{thm:informal_main_result}
Let \Cref{assum:eigenvalue_cvg,assum:smoothness,assum:C2_boundedness,assum:dissipative} hold.
We also assume \Cref{assum:C1_boundedness} under the bounded gradient condition (\Cref{assum:bounded_grad}).
Suppose the initial solution satisfies $\|x_0\| \leq 1$.
Then, 
there exits $\Lambda^*_\eta > 0$ for $\eta \geq 0$ such that for any $0 < \kappa < 1/4$ and $\delta \in (0,1)$, it holds that, 
    \begin{align*}\label{eq:informal_final_rate}
        &\bbP(L(X_n) - L(x^*) > \delta) 
 \lesssim   \frac{1}{\delta}\bigg\{  L(\tilde{x}) - L(x^*)  
\\ & 
~~+ \exp(- \Lambda_\eta^*(\eta n - 1))   + 
\frac{\cbeta}{\Lambda^*_0}\eta^{1/2-\kappa} \\ 
& ~~+ \!\! \textstyle
\Bigg[\frac{1}{\beta}\left(\!\sqrt{\frac{2M}{\lambda}} \!+\! 1\right)\! +\!
     \lambda\left( \frac{\|\tilde{x}\|_{\cHK}}{\sqrt{\beta}} + \|\tilde{x}\|_{\cHK}^2\!\! \right)\Bigg] \Bigg\}.
\numberthis
    \end{align*}
\end{theorem}
The proof is in \Cref{sec:final_rate_proof}.
A precise description of the spectral gap $\Lambda^*_\eta$ is given in \Cref{prop:geom_ergo}.
$\Lambda^*_{\eta}$ could be dependent on $\beta$ and $\eta$, but is uniformly lower bounded with respect to $\eta > 0$.
As can be seen in \Cref{eq:informal_final_rate}, there is a competing effect between the regularization $\Lambda^*_\eta$ (ensuring faster convergence of the discrete chain) and the inverse temperature $\beta$ (ensuring better concentration of the Langevin stationary distribution $\pi$). 
We can see that, for fixed $\lambda$ and $\beta$, 
%
%
by setting $\eta \leq \frac{\log n}{\Lambda_\eta^* n}$, \Cref{eq:informal_final_rate} excluding the optimization unrelated term $L(\tilde{x}) - L(x^*)$ is of order
\begin{equation}
    O\left(\frac{1}{n} + \frac{\cbeta}{\Lambda_0^*} \left(\frac{\log n}{\Lambda_\eta^*  n}\right)^{1/2 - \kappa} + \frac{1}{\beta} + \lambda \right).
\end{equation}
Note also that contrary to the finite dimensional setting where $1$ order weak convergence is possible, the $1/2$ rate in $\eta$ is optimal \citep{Brehier14} -- see \Cref{sec:second_term}.

Next, the convergence rate of SGLD is given as follows.
\begin{theorem}[Main Result, SGLD convergence rate]\label{thm:main_result_SGLD}
Under \Cref{assum:eigenvalue_cvg,ass:SGLD_cond} and $\|x_0\| \leq 1$, 
SGLD has the following convergence rate:
 \begin{align*}
        &\bbP(L(Y_n^N) - L(x^*) > \delta)  \\
& \lesssim   \frac{1}{\delta}\Bigg\{  \Theta_n + 
\frac{\cbeta}{\Lambda_{0}^*}   \mu_{N+1}^{1/2 - \kappa} 
+ \left(\sqrt{r_n} +  \sqrt[4]{r_n}\right)
 \Bigg\},
    \end{align*}
where $\displaystyle r_{n} = \frac{n \beta \eta (\ntr - \nbch)}{\nbch(\ntr-1)}$ and $\Theta_n = \exp(- \Lambda_\eta^*(\eta n - 1)) +
\frac{\cbeta}{\Lambda^*_0}\eta^{1/2-\kappa} 
+\frac{1}{\beta}\left(\!\sqrt{\frac{2M}{\lambda}} \!+\! 1\right)\! +\!
     \lambda\left( \frac{\|\tilde{x}\|_{\cHK}}{\sqrt{\beta}} + \|\tilde{x}\|_{\cHK}^2 \right)
+ L(\tilde{x}) - L(x^*)
$ which is the convergence rate of GLD shown in \Cref{thm:informal_main_result}.
\end{theorem}
The approximation error induced by the Galerkin approximation corresponds to $\frac{\cbeta}{\Lambda_{0}^*}   \mu_{N+1}^{1/2 - \kappa}$.
Since $\mu_{N+1} \lesssim N^{-2}$, the approximation error decreases in a quadratic order as the dimension $N$ is increased.
The error induced by the stochastic gradient corresponds to $\sqrt{r_n} +  \sqrt[4]{r_n}$. 
As the minibatch size $\nbch$ increases, the stochastic gradient error converges to 0.
This rate is slightly better than finite dimensional counter part \cite{Raginsky_Rakhlin_Telgarsky2017,Xu_Chen_Zou_Gu18},
by a factor of $\sqrt{k\eta}$.
This is due to the regularization term $\lambda \|x\|_{\cHK}^2$. 

\paragraph{Proof scheme}

Applying GLD and SGLD for non-convex optimization in a finite dimensional space has been investigated extensively recently in
\citet{Raginsky_Rakhlin_Telgarsky2017,Xu_Chen_Zou_Gu18,NIPS2018_8175} to name a few.
Our analysis could be an infinite dimensional extension of \citet{Raginsky_Rakhlin_Telgarsky2017,Xu_Chen_Zou_Gu18}.
However, unlike in the proof of such existing analyses for the finite dimensional case, 
$\bbE[L(X_n) - L(x^*)]$ cannot be directly bounded as the results used in an infinite-dimensional setting only apply to bounded test functions,  Corollary 1.2 in \citet{Brehier14} in particular. This is to be contrasted with \citet{Xu_Chen_Zou_Gu18} where the finite-dimensional assumption allows to derive results for test functions bounded by a Lyapunov function of the type $C(\norm{x}^k + 1)$.
Instead, sigmoid functions of the form $\phi(x) = \sigma(L(x) - L(x^*))$ with  $\sigma(x) = 1 / ({1 + e^{-x}})-1/2$ are used to bound the probability of the $n$-th iterate $X_n$ of \Cref{eq:sgld} being in a certain level set of $L(x) - L(x^*)$, by bounding $\bbE[\phi(X_n)]$ and applying Markov's inequality.

The seminal paper \citet{Raginsky_Rakhlin_Telgarsky2017} derived the finite time error bound of SGLD for non-convex learning problem utilizing the decomposition 
\begin{align}
&\bbE[\phi(X_n) - \phi(x^*)] = \bbE[\phi(X_n) - \phi(X(n\eta)))] + \notag \\
&\quad  \bbE[\phi(X(n\eta)) - \phi(X^\pi)] + \bbE[\phi(X^\pi)- \phi(x^*)],
\label{eq:phiDecomposeRaginsky}
\end{align}
where $\pi$ is the stationary distributions of the continuous Markov chain $\{X(t)\}_{t \geq 0}$ and 
we denote by $X^\mu$ a random variable obeying a probability distribution $\mu$.
On the other hand, \citet{Xu_Chen_Zou_Gu18} observed that this decomposition could be improved by utilizing the geometric ergodicity of discrete time dynamics and proposed to use the following decomposition:  
 \begin{align}
&\bbE[\phi(X_n) - \phi(x^*)] = \bbE[\phi(X_n) - \phi(X^{\mu_\eta})] + \notag \\
&\quad  \bbE[\phi(X^{\mu_\eta}) - \phi(X^\pi)] + \bbE[\phi(X^\pi)- \phi(x^*)],
\label{eq:phiDecomposeXuChen}
\end{align}
where $\mu_\eta$ is the stationary distribution of the discrete Markov chain $\{X_n \}_{n \in \bbN}$ (the existence of which is not trivial).
By using this, it is shown that some polynomial order term with respect to $n$ can be dropped to obtain a faster rate.\footnote{We would like to point out that we have found some incorrect analysis of the error bound in \citet{Xu_Chen_Zou_Gu18}. In particular, there are several wrong evaluations about dependency of constants (including the spectral gap) on the inverse temperature parameter $\beta$.}
Our analysis employs this strategy.
That is, we control (i) the convergence of the discrete chain to its stationary distribution (whose existence we prove), (ii) the convergence of the GLD stationary distribution to that of the Langevin diffusion, and (iii) the concentration of the Langevin diffusion around the global minimum of $L$.

The extension to an infinite dimensional setting is not trivial. For example, the boundedness of the noise $\epsilon_n$ does no longer hold, and thus we need an additional regularization term $AX(t)$ to make the solution bounded in $\cH$ and hit a compact set with high probability.
The time discretization of the infinite dimensional Langevin dynamics \cite{da_prato_zabczyk_1996} has been studied especially as a numerical scheme of stochastic partial differential equation \cite{Kuksin2001,debussche2011weak,Brehier14,Brehier16,Andersson2016,Chen2017,chen2018fulldiscrete}.
\citet{Brehier14} and \citet{Brehier16} derived a weak approximation error of the time discretization scheme \eqref{eq:sgld} from the stationary distribution $\pi$. However, their proof strategy utilizes the decomposition \Cref{eq:phiDecomposeRaginsky} as in \citet{Raginsky_Rakhlin_Telgarsky2017}.
As we have pointed out above, the error bound could be improved by using the decomposition \Cref{eq:phiDecomposeXuChen} instead. 
Unfortunately, the geometric ergodicity of the discrete time dynamics has not been established so far.
Therefore, we have introduced \Cref{assum:C1_boundedness} so that the geometric ergodicity holds. 
Thanks to this, the decomposition \eqref{eq:phiDecomposeXuChen} analogous to \citet{Xu_Chen_Zou_Gu18} can be employed to yield a better rate.




\section{Bounding the First Term: Geometric Ergodicity of the Discrete Chain}
\label{sec:first_term}

The proof from this section is adapted from \citet{Goldys_Maslowski06} for the discrete chain, i.e. it is shown that the hypothesis of Theorem 2.3 in \citet{meyn1994} are satisfied.
Namely, we prove the existence of a Lyapunov function of the form $V(x) = \norm{x} + 1$, and that a \emph{minorization condition} is satisfied on a ball $\cB_r \subset \cH$.
These two properties act jointly in ensuring the geometric ergodicity (\citet[Ch.15]{meyn1993markov}) of \Cref{eq:sgld}.
Indeed, the Lyapunov condition ensures the attractiveness of $\cB_r$ for the chain $\{ X_n\}_{n\in\bbN}$, while the minorization condition lower bounds the probability of staying in $\cB_r$.

%
%
The following proposition controls the chain in the case when $\nabla L = 0$ and is used as an auxiliary result.
\begin{restatable}{proposition}{Zboundedness}\label{prop:z_boundedness}
  Let $\{Z_n\}_{n \in \bbN}$ solve: $Z_0 = 0$ and 
  \begin{equation}
        Z_{n+1} = S_\eta Z_n+ \sqrt{\frac{2\eta}{\beta}}S_\eta \varepsilon_n,
  \end{equation}
with $\beta > \eta$. Then, $\forall p > 0, k(p) \eqdef \sup_{n \geq 0}\bbE (\norm{Z_n}^p) < \infty.$ 
\end{restatable}
The proof is rather straightforward and is deferred to \Cref{sec:ProofOfZbounded}.
\Cref{prop:norm_control} controls the decrease of $\norm{X_n}$ in expectation and is the key result towards proving the existence of a Lyapunov function.
It relies on the regularizing effect of $A$, through $S_\eta$: indeed, it holds that $\forall k,  S_\eta f_k = (\Id - \eta A)^{-1} f_k = \frac{1}{1 + \lambda\eta/\mu_k}f_k$, hence $S_\eta$ is a bounded linear operator of norm  $\opnorm{S_\eta} = \frac{1}{1 + \lambda\eta/\mu_0} < 1$.

\begin{proposition}\label{prop:norm_control}
    Let \Cref{assum:smoothness,assum:dissipative} hold. We have
    \begin{equation*}
    \bbE_{x_0}\norm{X_n} \leq \rho^n \norm{x_0} + b , \quad \forall n\in\bbN ,
    \end{equation*}
    with (i) (for Strict Dissipativity) $\rho = \frac{1 + \eta M}{1 + \lambda\eta/\mu_0} < 1$, $b = \norm{x^*} + 2 k(1)$, or (ii) (for Bounded gradients)  $\rho = \frac{1}{1 + \lambda\eta/\mu_0} < 1$, $b = \frac{\mu_0}{\lambda}B + k(1)$.
\end{proposition}

The proof is given in \Cref{sec:ProofOfNormControl}.
Combining this Lyapunov condition with a ``minorization condition'', we can show the geometric ergodicity in the following proposition.
\begin{proposition}[Geometric ergodicity]\label{prop:geom_ergo}
Let \Cref{assum:eigenvalue_cvg,assum:smoothness,assum:C2_boundedness,assum:dissipative} hold.
We also assume \Cref{assum:C1_boundedness} under the bounded gradient condition (\Cref{assum:bounded_grad}).
Let $\eta >0, \beta > \eta$ and 
$V(x) = \|x\| + 1$.
Then, there exists a unique invariant measure $\mu_{{\eta}}$ and ${\Lambda_\eta^*} > 0$
such that for all $\phi: \cH \to \mathbb{R}$ with $|\phi(\cdot)| \leq V(\cdot)$ and $\|\phi(x) - \phi(y)\| \leq M' \|x - y\|~(x,y \in \cH)$, it holds that 
\begin{align}\label{eq:geom_ergo}
& 
{|\bbE_{x_0}[\phi(X_n)]  -  \bbE[\phi(X^{\mu_\eta})]|} \leq 
C_{x_0} \exp(- \Lambda^*_\eta (\eta n - 1)),
\end{align}
where $C_{x_0}$ and $\Lambda_{{\eta}}^* > 0$ are given by 
\vspace{-0.2cm}
\begin{enumerate}[parsep=0pt,partopsep=0pt,topsep=0pt,itemindent=-0.3cm,itemsep=-0.2cm,labelsep=0.1cm]
\item[i)] (Strict dissipativity, \Cref{assum:strict_diss})
\begin{align*}
& \Lambda^*_\eta =  \frac{\frac{\lambda}{\mu_0} - M}{1 + \eta \frac{\lambda}{\mu_0}},~C_{x_0} =  M'(\|x_0\|_{\cH} + b),
\end{align*}
\item[ii)] (Bounded Gradient, \Cref{assum:bounded_grad})
\begin{align*}
&\textstyle \Lambda^*_\eta = \frac{\min\left(\frac{\lambda}{2 \mu_0}, \frac{1}{2} \right)}{4 \log(\kappa (\bar{V} + 1)/(1-\delta)) } \delta, \\ &
\textstyle C_{x_0} = \kappa [\bar{V} + 1] + \frac{\sqrt{2} (V(x_0) + b)}{\sqrt{\delta}}, 
\end{align*}
for $0 < \delta< 1$ satisfying $\delta = \Omega(\exp(-O(\beta)))$, $\bar{b} = \max\{b,1\}$, $\kappa = \bar{b} + 1$ and $\bar{V} = \frac{4 \bar{b}}{\sqrt{(1+\rho^{1/\eta})/2} - \rho^{1/\eta}}$\footnote{More detailed evaluation of $\delta$ can be found in the proof.}.
\end{enumerate}

%
%
\end{proposition}
The proof is given in \Cref{sec:ProofOfGeometricErgodicity}.
Unlike existing work, this theorem asserts the geometric ergodicity of the discrete time dynamics, whilst the geometric ergodicity for ``continuous time'' dynamics (\Cref{eq:langevin_sde}) has been well known, see as an example \citet{debussche2011weak,debussche2013ergodicity}.
The proof follows a standard argument that utilizes the Lyapunov condition (\Cref{prop:norm_control}) and the minorization condition.
Here, the minorization condition asserts that the transition kernel with respect to the discrete time Markov process shares a common probability mass on a bounded region uniformly over initial state $x_0$ with some bounded norm. Once this condition is shown then the recurrence probability can be lower bounded combined with the Lyapunov condition, which yields the {\it coupling argument}.
To show a faster convergence, we employed the coupling technique of \citet{mattingly2002} and adopted it to the proof technique of \citet{Goldys_Maslowski06} developed for a continuous time dynamics. Hence, we obtained faster rates than \citet{Goldys_Maslowski06}. In particular, the dependency on $\beta$ is improved.

Transforming the continuous time argument to the discrete time setting is far from trivial because 
there appears a ``integrability'' problem in showing the minorization condition, which makes it difficult to show the geometric ergodicity.
Indeed, \citet{Brehier14,Brehier16} pointed out there has been no work that showed the geometric ergodicity of the time discretized dynamics.
This difficulty does not occur in the finite dimensional setting. We resolved this problem by imposing \Cref{assum:C1_boundedness}.
Thanks to this, we have exponential convergence $\exp(- \Lambda_\eta^* n \eta)$ improving the polynomial order rate $\frac{1}{\Lambda_0^*}(n\eta)^{-1}$ of existing work.





\section{Second Term: Weak Convergence of the Discrete Scheme}
\label{sec:second_term}
%
The second term is linked to the weak convergence of the numerical scheme, i.e., in our case the convergence of $\phi(X_n)$ to $\phi(X(n\eta))$ for any admissible test function $\phi\in C^2_b$.
We rely directly on the results of \citet{Brehier16}, who prove $1/2$ order weak convergence in time and $1$ order weak convergence in space for numerical schemes that have a semi-implicit discretization in time with $\beta = 1$, and a finite elements discretization in space;
that is, they showed
\begin{equation}\label{eq:phiconvergenceBetaOne}
 \left|\int \phi\mathrm{d}\mu^\eta - \int \phi \mathrm{d}\pi\right| \leq C\norm{\phi}_{0,2}\eta^{1/2 - \kappa},
\end{equation}
where $\|\phi\|_{0,2} \eqdef \max\{ \|\phi\|_{\infty}, \sup_{x \in \cH}\|\nabla \phi(x)\|_{\cH}, \sup_{x\in \cH} $\\ $\|D^2 \phi(x)\|_{\mathcal{B}(\cH)} \}$ for $\phi \in C_b^2$.
%
%

In the general setting, $\beta \neq 1$, we need to evaluate the effect of $\beta$.
To that purpose, we essentially consider a re-scaling argument, that is, 
we observe that if we replace $L$ with $L' \eqdef \beta L$, $\lambda$ with $\lambda' \eqdef \beta \lambda$ and $\eta$ with $\eta' \eqdef \frac{\eta}{\beta}$ in \Cref{eq:langevin_sde} and \Cref{eq:sgld}, then it holds that 
\begin{equation*}
\textstyle
    S_\eta \! = \! \left(\Id + \eta \frac{\lambda}{2} \nabla\|\!\cdot\!\|_{\cHK}^2\!\right)^{-1} \!\!\!\!=\!\! \left(\Id +\frac{\eta}{\beta} \frac{\beta\lambda}{2} \nabla\|\!\cdot\!\|_{\cHK}^2\!\right)^{-1} \!\!\!=: \tilde{S}_{\eta'},
\end{equation*}
and thus
\begin{align*}
     X_{n + 1} 
     &= \tilde{S}_{\eta'} X_n - \eta' \tilde{S}_{\eta'} \nabla L'(X_n) + \sqrt{2\eta'}\tilde{S}_{\eta'} \varepsilon_n,
\end{align*}
i.e., $\{X_n\}_{n \in \bbN}$ is the numerical approximation of 
\begin{align*}
      \dX(t) = -\nabla \cL' (X(t)) + \sqrt{2} \dW(t),
\end{align*}
with time step $\eta'$. 
We carefully evaluate how the constant $C$ is \Cref{eq:phiconvergenceBetaOne} will be changed after rescaling.
We can see that $\beta$ affects the rate through the spectral gap $\Lambda^*_0$, which corresponds to the continuous dynamics ($\eta = 0$).
%
Eventually, we get the following result:
\begin{proposition}[Case $\beta \neq 1$\label{prop:scheme_cvg_general}]
Under the same setting as \Cref{prop:geom_ergo}, for any $0<\kappa<1/2$, $0 < \eta_0$,
there exists a constant $C$ such that for any bounded test function $\phi\in C_b^2$ and $0<\eta <\eta_0$, it holds that
\begin{equation}
 \left|\int \phi\mathrm{d}\mu^\eta - \int \phi \mathrm{d}\pi\right| \leq C \frac{\norm{\phi}_{0,2}}{\Lambda^*_0} \cbeta \eta^{1/2 - \kappa}.
\end{equation}
\end{proposition} 
The proof is given in \Cref{sec:ProofOfTimeAndSpaceApprox}.
Note that due to the infinite dimensional setting, the $1/2$ rate w.r.t the time discretization $\eta$ is optimal \citep{Brehier14}. This is to be contrasted with the finite-dimensional case, where $1$ order weak convergence is attainable.



\section{Third Term: Concentration of the Gibbs Distribution Around the Global Minimum}
\label{sec:third_term}
%
The last term corresponds to the concentration of the stationary Gibbs distribution around the global minimum of $L$. In this infinite-dimensional setting, the regularizing effect of operator $A$ is necessary to ensure good convergence properties of the discrete and continuous chains. Hence, even in the limit case $\beta\rightarrow 0$ one cannot expect to have arbitrary tight concentration around the global minimum. This is to be contrasted with the finite dimensional case (\citet{Chiang_Hwang_Sheu87,Gelfand_Mitter91,Roberts_Tweedie96}). 
In fact, $A$ constrains the chain to remain within the support of a Gaussian process which is compactly embedded in $\cH$.

\begin{proposition}\label{prop:gibbs_concentration}
Under \Cref{assum:smoothness,assum:eigenvalue_cvg}, it holds that 
\begin{equation*}
 \int L \dd \pi \!-\! L(\tilde{x}) \lesssim 
     \!\frac{1}{\beta}\!\left(\!\sqrt{\frac{2M}{\lambda}} + 1 \!\right) + 
     \lambda\left( \frac{\|\tilde{x}\|_{\cHK}}{\sqrt{\beta}} \!+\! \|\tilde{x}\|_{\cHK}^2\!\right).
\end{equation*}
\end{proposition}
%
The proof can be found in \Cref{sec:ProofOfThirdTerm}.
The proposition can be shown by utilizing an analogous technique to the convergence rate analysis of Gaussian process regression \cite{JMLR:Vaart&Zanten:2011}. Along with this technique, the {\it Gaussian correlation inequality} \cite{royen2014simple,Latala2017} is used. This inequality gives a powerful tool to lower-bound the Gaussian probability measure of the intersection of two centered convex sets.

\section{Error Bound for the Galerkin Approximation and Stochastic Gradient}
\label{sec:approx_error}
The error induced by the Galerkin approximation can be evaluated as in the following proposition.
\begin{proposition}\label{prop:TimeAndSpaceApproxError}
Let \Cref{assum:eigenvalue_cvg,assum:smoothness,assum:C2_boundedness,assum:dissipative} hold and suppose $\|x\| \leq 1$.
Then, there exists an invariant measure $\mu_{(N,\eta)}$ for the discrete time Galerkin approximation scheme (\Cref{eq:GalerkinTimeDiscrete}), 
and for any $0 < \kappa < 1/2,~0 < \eta_0$, there exists a constant $C > 0$ such that, for any $N \in \bbN$ and $0 < \eta < \eta_0$, 
\begin{align*}
& \bbE[\phi(X^{\mu_{(N,\eta)}}) - \phi(X^\pi)]   \leq 
\frac{C\|\phi\|_{0,2}}{\Lambda_{0}^*} 
\cbeta \left( \mu_{N+1}^{1/2 - \kappa}
+  \eta^{1/2 - \kappa}\right).
\end{align*}
\end{proposition}
The proof is in \Cref{sec:ProofOfTimeAndSpaceApprox}.
We see that, by taking $N \to \infty$, we can replicate \Cref{prop:scheme_cvg_general}.
Moreover, the geometric ergodicity of the time discretized dynamics with the Garelkin approximation holds completely in the same manner as \Cref{prop:geom_ergo}.
The discrepancy between GLD and SGLD with the Garelkin approximation can be bounded as follows.
\begin{proposition}\label{prop:SGLDdiscrepancy}
Suppose $\|x_0\| \leq 1$. There exists a constant $C > 0$ such that, for any $n,N \in \bbN$, any $\beta > 1$ and sufficiently small $\eta > 0$, 
\begin{align*}
& \bbE[\phi(X_n^N) - \phi(Y_n^N)]  \leq 
C \left(\sqrt{r_n} +  \sqrt[4]{r_n}\right).
\end{align*}
where
$r_{n} = \frac{n \beta \eta (\ntr - \nbch)}{\nbch(\ntr-1)}.$
\end{proposition}
The proof is given in \Cref{sec:ProofOfSGLD}. 
From these propositions, we can see that the SGLD with the Garelkin approximation also gives a reasonably good solution for sufficiently large $N \in \bbN$, sufficiently small $\eta > 0$ and sufficiently large mini-batch size.
\Cref{prop:SGLDdiscrepancy} is analogous to those given for finite dimensional situations \cite{Raginsky_Rakhlin_Telgarsky2017,Xu_Chen_Zou_Gu18}.
However, thanks to the regularization term (appearing as $S^\eta$), our rate is better by a factor of $\sqrt{k\eta}$.

%
\section{Other Related Work}
%
In this section, we mention other related work that have not been exposed above.
An analogous assumption to \Cref{assum:C1_boundedness} has already been introduced in the analysis of infinite dimensional dynamics with nonlinear diffusion term, that is, $\dW(t)$ is replaced by a nonlinear quantity $\sigma(X(t)) \dW(t)$ for $\sigma(X(t)) \in \mathcal{B}(\cH)$
\cite{conus2019,debussche2011weak,BREHIER2018193}.
These papers analyzed the existence of stationary distribution for continuous dynamics and discrete time approximation for finite time horizon.
\citet{Chen2017,chen2018fulldiscrete} analyzed linear/nonlinear Schr{\"o}dinger equations and derived geometric ergodicity, but they analyzed much more specific situations or stronger assumptions (e.g. the strong dissipativity condition).

The geometric ergodicity of infinite dimensional Markov processes for discrete time settings has been investigated by \citet{Kuksin2001} and 
infinite dimensional MCMC such as preconditioned Crank–Nicolson (pCN) \cite{hairer2014,eberle2014,Vollmer_2015,Rudolf2018}, 
and in particular the Metropolis-Adjusted Langevin Algorithm (MALA) \cite{Durmus2015,BESKOS2017327}.
Among them, MALA is the most related to our setting (discrete time Langevin dynamics). The biggest difference the presence of a rejection step. Since the purpose of our work is rather optimization than sampling, and since the rejection step is not compatible with stochastic gradient descent, we do not pursue this direction.
\section*{Conclusion and Future Work}
\label{sec:conclusion}
In this paper, we have presented a non-asymptotic analysis of the convergence of GLD and SGLD in a RKHS and for a non-convex objective function. The bounds obtained in this infinite-dimensional setting involve the spectrum of the associated integral operator and a regularization factor instead of the dimension $d$, which to the best of our knowledge is the first result on applying GLD in RKHS to infinite-dimensional nonconvex optimization. 
In future work, we hope to alleviate the somewhat strict \Cref{assum:eigenvalue_cvg} linked to current results from the numerical approximation literature. 
Drawing inspiration from \cite{Xu_Chen_Zou_Gu18}, we also plan to extend our analysis to variance-reduced SGLD algorithms.

\bibliographystyle{abbrvnat}
\bibliography{main_langevin.bib}


\onecolumn
\appendix
%


\section{Proof of Eq.~\eqref{eq:LipschitzGradBound} and Eq.~\eqref{eq:D2bound}}
\label{sec:ProofOfExampleSmoothBound}
If $\|\ell_i''\|_\infty \leq G$, then it is $G$-Lipschitz continuous. Therefore, it holds that  
\begin{align*}
& \|\nabla L(x) - \nabla L(x')\| \\
& \leq \frac{1}{\ntr}\sum_{i=1}^{\ntr} |\ell_i'(\langle x, \psi_\gamma(z)\rangle_{\cH}) - \ell_i'(\langle x', \psi_\gamma(z')\rangle_{\cH})| \|\psi_\gamma(z_i)\|_{\cH} + \lambda_0 \|x - x'\|_{\cH} \\
& \leq \frac{1}{\ntr}\sum_{i=1}^{\ntr} |\ell_i'(\langle x, \psi_\gamma(z)\rangle_{\cH}) - \ell_i'(\langle x', \psi_\gamma(z')\rangle_{\cH})| \sqrt{K_\gamma(z_i,z_i)} + \lambda_0 \|x - x'\|_{\cH} \\
& \leq \sup_z \sqrt{K_\gamma(z,z) }G \frac{1}{\ntr} \sum_{i=1}^{\ntr} \|\langle x - x', \psi_\gamma(z_i)\rangle_{\cH} \| + \lambda_0 \|x - x'\|_{\cH} \\
& \leq G \sup_z K_\gamma(z,z)   \| x - x' \|_{\cH} + \lambda_0 \|x - x'\|_{\cH}  \leq (G R_\gamma + \lambda_0) \| x - x' \|_{\cH}.
\end{align*}
This yields Eq.~\eqref{eq:LipschitzGradBound}.
As for the second order derivative (Eq.~\eqref{eq:D2bound}), 
first note that 
$$
D^2 L(x) \cdot (h,k)= \frac{1}{\ntr} \sum_{i=1}^{\ntr} \ell''(\langle x, \psi_\gamma(z_i)\rangle_{\cH}) \langle \psi_\gamma(z_i), h\rangle_{\cH}
\langle \psi_\gamma(z_i), k \rangle_{\cH} + \lambda_0 \langle h,k \rangle_{\cH}
$$
for $h,k\in \cH$.
Therefore, we have that 
\begin{align*}
&  |D^2 L(x) \cdot (h,k) - \lambda_0 \langle h,k\rangle_{\cH} | \\
& \leq 
\frac{1}{\ntr}\sum_{i=1}^{\ntr} |\ell''(\langle x,\psi_\gamma(z_i)\rangle_{\cH})| |\langle \psi_\gamma(z_i),h \rangle_{\cH}| \|\psi_\gamma(z_i)\|_{-\alpha}
\|k\|_{\alpha} \\
& \leq G  
\max_i \|\psi_\gamma(z_i)\|_{\cH}\|h\|_{\cH}  
\frac{1}{\ntr}\sum_{i=1}^{\ntr} \|\psi_\gamma(z_i)\|_{-\alpha} \|k\|_{\alpha}\\
& = G  \max_i \sqrt{K_\gamma(z_i,z_i)} \|h\|_{\cH}  
\frac{1}{\ntr} \sum_{i=1}^{\ntr} \sqrt{ K_{\gamma - 2\alpha}(z_i,z_i)} \|k\|_{\alpha}\\
& \leq G  \max_i \sqrt{K_\gamma(z_i,z_i)} \|h\|_{\cH}  
 \sqrt{ \frac{1}{\ntr} \sum_{i=1}^{\ntr} K_{\gamma - 2\alpha}(z_i,z_i)} \|k\|_{\alpha}\\
& \leq G  \max_i \sqrt{K_\gamma(z_i,z_i)} \|h\|_{\cH}  \sqrt{\sum_{k=0}^\infty \mu_k^{\gamma - 2\alpha}}\|k\|_{\alpha}
\leq G\sqrt{R_\gamma} \|h\|_{\cH}  \sqrt{\sum_{k=0}^\infty \mu_k^{\gamma - 2\alpha}} \|k\|_{\alpha}.
\end{align*}
$\square$

\subsection{Remark on existence of regularization term}
\label{sec:RemarkRegularization}
As an example of $L(x)$, it is useful to consider a setting where $L(x)$ can be expressed as $L(x) = \tilde{L}(x) + \frac{\lambda_0}{2}\|x\|^2$ for $\tilde{L}(x)$ that satisfies the assumptions listed in the main text and $\lambda_0 \geq 0$.
In this case, $L(x)$ does not satisfy the bounded gradient condition \Cref{assum:bounded_grad}.
However, by considering the following update rule, we can show the same error bound for $L(x)$: 
\begin{align}
  \begin{cases}
      X_0 = x_0 \in \cH  , \\
      X_{n + 1} = S'_\eta (X_n - \nabla \tilde{L}(X_n) + \sqrt{2 \tfrac{\eta}{\beta}} \varepsilon_n),
  \end{cases}
\end{align}
where $S'_\eta = \left[\Id + \eta ( \frac{\lambda_0}{2} \nabla \norm{\cdot}_{\cHK} +\frac{\lambda}{2} \nabla \norm{\cdot}_{\cH} )\right]^{-1}$.

\section{Proof of \Cref{prop:dissipative}}
\begin{proof}
Let us assume $\lambda  > M\mu_0$ (Strict Dissipativity).
  \Cref{assum:eigenvalue_cvg} implies, for $x = \sum_{k=0}^\infty \alpha_k f_k$,
  \begin{align}
    \langle Ax, x \rangle
        &= - \lambda \Big\langle  \sum_{k=0}^\infty \frac{\alpha_k}{\mu_k} f_k,
                              \sum_{k=0}^\infty \alpha_k f_k \Big\rangle  \nonumber \\
        &= - \lambda \sum_{k=0}^\infty \frac{\alpha_k^2}{\mu_k} \nonumber \\
        &\leq - \frac{\lambda}{\mu_0} \sum_{k=0}^\infty \alpha_k^2 = - \frac{\lambda}{\mu_0} \norm{x}^2   ,
  \end{align}
 and \Cref{assum:smoothness} implies
\begin{align}
     \scal{-\nabla L(x)}{x}
     &\leq M\norm{x - x^*}\norm{x} \nonumber\\
     &\leq M\norm{x}^2 + M\norm{x}\norm{x^*}  .
\end{align}
Hence,
\begin{equation*}
\scal{Ax - \nabla L(x)}{x} \leq\! -(\tfrac{\lambda}{\mu_0} - M) \norm{x}^2\! + M\norm{x}\norm{x^*}  .
\end{equation*}
Therefore, if $M < \frac{\lambda}{\mu_0}$, there exists $m, c>0$ such that \Cref{eq:dissipative} holds.
The proof when \Cref{assum:bounded_grad} holds is similar.
\end{proof}

\section{Proof of main result: \Cref{thm:informal_main_result} and \Cref{thm:main_result_SGLD}}
\label{sec:final_rate_proof}
In light of \Cref{sec:first_term,sec:second_term,sec:third_term}, we can now state our final result.
We introduce the following bounded test function:
\begin{equation}
    \phi(x) = \sigma(L(x) - L(x^*))~~~(x \in \cH),
\end{equation}
where $\sigma(u) = \frac{1}{1 + e^{-u}} - \tfrac{1}{2}~(u \in [0,\infty))$ is concave and takes values in $[0, 1)$.
In particular, $\norm{\phi(\cdot)}_V \leq 1$ for $V(x) = M\|x\| + 1$ and $\phi\in C^2_b(\cH)$, hence $\phi$ falls within the scope of \Cref{prop:geom_ergo,prop:scheme_cvg_general}.

First, we note that there exists a unique invariant measure $\mu_\eta$ for the discrete time dynamics $\{X_n\}_n$
and there also exists a unique invariant measure $\mu_{(N,\eta)}$ for the discrete time Garelkin approximated dynamics $\{X^N_n\}_n$ by \Cref{prop:gibbs_concentration}.
To obtain the result, we make use of Markov's inequality: for any $0 < \delta < 1$, 
\begin{align*}
& P(L(X_n) - L(x^*) > \delta)  \\
& \leq 
P(\phi(X_n) > \sigma(\delta)) \\
& \leq  \frac{\bbE[\phi(X_n)]}{\sigma(\delta)}~~~~~(\because \text{Markov's inequality}) \\
& = \frac{1}{\sigma(\delta)}\left( \bbE[\phi(X_n) - \phi(X^{\eta})] +  \bbE[\phi(X^{\eta}) - \phi(X^{\pi})] +  \bbE[\phi(X^\pi)] 
\right).
\end{align*}
The first term ($\bbE[\phi(X_n) - \phi(X^{\eta})]$) can be bounded by \Cref{prop:geom_ergo}.
The second term ($\bbE[\phi(X^{\eta}) - \phi(X^{\pi})]$) can be bounded by \Cref{prop:scheme_cvg_general}.
Next, we bound the third term.
Since $\sigma(u) \leq u$ for all $u \in [0,\infty)$ and $L(x) - L(x^*)\geq 0$ for all $x \in \cH$, it holds that 
\begin{align}
\bbE[\phi(X^{\pi})] \leq \bbE[L(X^\pi) - L(x^*)] = (\bbE[L(X^\pi)] - L(\tilde{x})) + (L(\tilde{x}) - L(x^*)).
\label{eq:LastGibbsConcBound}
\end{align}
Then, the first term ($\bbE[L(X^\pi)] - L(\tilde{x})$) in the right hand side is bounded by \Cref{prop:gibbs_concentration}.
Finally, we observe that $1/\sigma(\delta) \leq 5/\delta$ for all $\delta \in (0,1)$.
Combining all results, we obtain \Cref{thm:informal_main_result}.

As for the \Cref{thm:main_result_SGLD}, we use the following decomposition 
\begin{align*}
\bbE[\phi(X_n)]
= \bbE[\phi(Y^N_n) - \phi(X^N_n)] +  \bbE[\phi(X^N_n) - \phi(X^{\mu_{(N,\eta)}})] +  \bbE[\phi(X^{\mu_{(N,\eta)}}) - \phi(X^\pi)]
+ \bbE[\phi(X^\pi)].
\end{align*}
We apply \Cref{prop:SGLDdiscrepancy} to the first term $(\bbE[\phi(Y^N_n) - \phi(X^N_n)])$
and apply \Cref{prop:TimeAndSpaceApproxError} to the third term  $(\bbE[\phi(X^{\mu_{(N,\eta)}}) - \phi(X^\pi)])$.
As for the remaining terms, the same bound as \Cref{prop:geom_ergo} can be applied to the second term ($\bbE[\phi(X^N_n) - \phi(X^{\mu_{(N,\eta)}})]$), and the last term $\bbE[\phi(X^\pi)]$ can be bounded by \Cref{eq:LastGibbsConcBound} with \Cref{prop:gibbs_concentration}.
This yields \Cref{thm:main_result_SGLD}.

\section{Proof of \Cref{prop:z_boundedness}}
\label{sec:ProofOfZbounded}
\begin{proof}
This is proved in \citet{Brehier14} for $\beta = 1$. The $\beta > \eta$ assumption is necessary to ensure that $k(p)$ can be treated as a constant w.r.t $\beta$ and $\eta$ in the following.
We recall the main arguments of the proof. $\{Z_n\}_{n \in \bbN}$ is the semi-implicit approximation of the continuous Markov chain defined by:
  \begin{equation}\label{eq:def_z}
      \begin{cases}
        \dZ(t) = AZ(t) \dt + \sqrt{\tfrac{2}{\beta}} \dW(t) , \\
        Z(0) = 0 .
      \end{cases}
  \end{equation}
 Under \Cref{assum:eigenvalue_cvg}, it can be shown that $\sup_{t \geq 0}\bbE (\norm{Z(t)}^p) < \infty,~\forall p \geq 1$.
 Finally, $\{Z_n\}$ is a numerical scheme with strong order $\frac{1}{4}$ (\citet[Theorem 3.2]{printems_2001}), which implies the result.
\end{proof}

\section{Proof of \Cref{prop:norm_control}}
\label{sec:ProofOfNormControl}
\begin{proof}
%
The discrete chain $Y_n \eqdef X_n - Z_n, n \geq 0$ satisfies
\begin{equation*}
Y_{n+1} = S_\eta  Y_n  - \eta  S_\eta \nabla L(X_n) .
\end{equation*}
Hence, using \Cref{assum:smoothness} and the fact that $X_n = Y_n + Z_n$, we get
\begin{align*}
  \norm{Y_{n+1}} &\leq \opnorm{S_\eta}\norm{Y_n  - \eta  \nabla L(X_n)}\\
  &\leq \tfrac{1}{1 + \lambda\eta/\mu_0} ( (1 +  \eta M)\norm{Y_n} + \eta M (\norm{x^*}  +\!\norm{Z_n}) ).
\end{align*}
Taking the expectation and using $\bbE\norm{Z_n} \leq k(1)$ (\Cref{prop:z_boundedness}), this yields
\begin{equation*}
  \bbE\norm{Y_{n+1}} \leq \tfrac{1}{1 + \lambda\eta/\mu_0} ( (1 +\eta M)\bbE\norm{Y_n} + \eta M (\norm{x^*}  + k(1)) ),
 \end{equation*}
from which we deduce
\begin{equation}
    \bbE\norm{Y_n} \leq \rho^n \norm{x_0} + \tfrac{\eta  (1 - \rho^n) M}{(1 - \rho)(1 + \lambda\eta / \mu_0)}(\norm{x^*} + k(1)).
\end{equation}
Therefore,
\begin{equation}
\bbE_{x_0}\norm{X_n} \leq \rho^n \norm{x_0} +\frac{\eta M(\norm{x^*} + k(1))}{(1 - \rho)(1 + \lambda\eta / \mu_0)} + k(1).
\end{equation}
%
Finally, we conclude by observing that $\frac{\eta}{1 - \rho}  \frac{M}{1 + \eta\lambda/\mu_0}  = 1$. 

The proof with bounded gradients is similar. Since 
\begin{align*}
  \norm{Y_{n+1}} &\leq \opnorm{S_\eta}\norm{Y_n  - \eta  \nabla L(X_n)}\\
  &\leq \tfrac{1}{1 + \lambda\eta/\mu_0} ( \norm{Y_n} + \eta B ),
\end{align*}
we have 
$$
\norm{Y_n} \leq \rho^n \norm{x_0} + \frac{(1-\rho^n)}{1 - \rho} \frac{\eta B}{1 + \lambda \eta/\mu_0}
\leq \rho^n \norm{x_0} + \frac{\mu_0}{\lambda} B, 
$$
where $\rho = \tfrac{1}{1 + \lambda\eta/\mu_0}$.
Hence, noting $\|X_n\| \leq \|Y_n\| + \|Z_n\|$, we have that $\bbE[\norm{X_n}] \leq \rho^n \norm{x_0} + \frac{\mu_0}{\lambda} B + k(1)$
\end{proof}





\section{Proof of Proposition \ref{prop:geom_ergo}} 
\label{sec:ProofOfGeometricErgodicity}

\paragraph{Proof under the Strict Dissipativity Condition (\Cref{assum:strict_diss})}
First we prove the geometric ergodicity under \Cref{assum:strict_diss}.
To show that we first prove the exponential contraction:
\begin{align}
\|X_n - Y_n\|_{\cH} \leq \left(1 - \eta \frac{\frac{\lambda}{\mu_0} - M}{1 + \eta \frac{\lambda}{\mu_0}} \right)^n \|X_0 - Y_0\|_{\cH}.
\end{align}
Once we have shown this inequality, it is easy to show the geometric ergodicity.

According to the update rule, we have that 
\begin{align*}
X_{n+1} & = S_\eta\left(X_n - \eta \nabla L(X_n) + \sqrt{\frac{2 \eta}{\beta}} \epsilon_n \right), \\
Y_{n+1} & = S_\eta\left(Y_n - \eta \nabla L(Y_n) + \sqrt{\frac{2 \eta}{\beta}} \epsilon_n \right).
\end{align*}
Therefore, by taking difference, we obtain
$$
X_{n+1} - Y_{n+1} = S_{\eta}\left[ (X_n - Y_n) - \eta (L(x_n) - L(Y_n))\right].
$$
Then, by the triangular inequality, this yields
\begin{align*}
\|X_{n+1} - Y_{n+1}\|_{\cH}
& \leq \frac{1}{1 + \eta \frac{\lambda}{\mu_0}} ( \|X_n - Y_n\|_{\cH} + \eta \|L(X_n) - L(Y_n)\|_{\cH}) \\
& \leq \frac{1}{1 + \eta \frac{\lambda}{\mu_0}} ( \|X_n - Y_n\|_{\cH} + \eta M \|X_n - Y_n\|_{\cH} ) \\
& \leq \frac{1 + \eta M}{1 + \eta \frac{\lambda}{\mu_0}} \|X_n - Y_n\|_{\cH} \\ & 
\leq \left( 1 - \eta  \frac{\frac{\lambda}{\mu_0} - M}{1 + \eta \frac{\lambda}{\mu_0}} \right)\|X_n - Y_n\|_{\cH}
\leq \left( 1 - \eta  \frac{\frac{\lambda}{\mu_0} - M}{1 + \eta \frac{\lambda}{\mu_0}} \right)^n \|X_0 - Y_0\|_{\cH}.
\end{align*}
Now, we already know that there exists an invariant low $\mu_\eta$ under the strong dissipativity condition.
By assuming $Y_0 \sim \mu_\eta$ and $X_0 = x_0 \in \cH$, we can show the following geometric convergence:
$$
\bbE[\phi(X_n)] - \bbE_{X\sim \mu_\eta}[\phi(X)]
= 
\bbE[\phi(X_n)] - \bbE[\phi(Y_n)]
\leq 
M' \bbE[\|X_n - Y_n\|_{\cH}]
\leq 
M' \left( 1 - \eta  \frac{\frac{\lambda}{\mu_0} - M}{1 + \eta \frac{\lambda}{\mu_0}} \right)^n 
\bbE[\|X_0 - Y_0\|_{\cH}]. 
$$ 
Now, we see that 
$$
\bbE[\|X_0 - Y_0\|_{\cH}]
\leq \|x_0\|_{\cH} + \bbE[\|Y_0\|_{\cH}]
\leq \|x_0\|_{\cH} + b.
$$
In the last inequality, we used that
$\bbE[\|Y_0\|_{\cH}] = \bbE[\|Y_n\|_{\cH}] \leq \rho^n \bbE[\|Y_0\|_{\cH}] + b$ for all $n = 1,2,\dots$ by Proposition \ref{prop:norm_control}
and we took $n \to \infty$.
As a consequence, we obtain 
$$
\bbE[\phi(X_n)] - \bbE_{X\sim \mu_\eta}[\phi(X)]
\leq M' \exp\left( - n  \eta  \frac{\frac{\lambda}{\mu_0} - M}{1 + \eta \frac{\lambda}{\mu_0}}  \right) (\|x_0\|_{\cH} + b),
$$
where we used the relation $1-a \leq \exp(-a)$ for $a > 0$.
This yields the assertion.

\paragraph{Proof under the Bounded Gradient Condition (\Cref{assum:bounded_grad})}
Next, we prove the theorem under the bounded gradient case (\Cref{assum:bounded_grad}).
Under the strict dissipative condition, the statement can be immediately shown and thus we omit the proof.

We adopt the technique of Theorems 5.2 \& 5.3 from \cite{Goldys_Maslowski06}, and 
show the geometric ergodicity via Theorem 2.5 of \cite{mattingly2002}.
We note that Theorem 2.5 of \cite{mattingly2002} is shown for a finite dimensional setting, but it can be adopted for an infinite dimensional setting if the ``minorization condition'' (Lemma 2.3 of \cite{mattingly2002}) and ``Lyapunov condition'' (Assumption 2.2 of \cite{mattingly2002}) are satisfied. 

Since the Lyapunov condition is already shown by Proposition \ref{prop:norm_control}, we only need to show the minorization condition. 
Let $\mu_{k,\eta}^x$ be the law of
\begin{equation}
Z^x_{k,\eta} = S^k_\eta x + \sqrt{\frac{2\eta}{\beta}} \sum_{l = 0}^{k-1} S^{k-l}_\eta \varepsilon_l,
\end{equation}
and $\mu_{k,\eta}$ be the law of
\begin{equation}
Z_{k,\eta} = \sqrt{\frac{2\eta}{\beta}} \sum_{l = 0}^{k-1} S^{k-l}_\eta \varepsilon_l.
\end{equation}
Let $Q \eqdef \frac{2\eta}{\beta} \Id$, and 
$$ 
Q_k \eqdef \sum_{l=0}^{k-1} Q S_\eta^{2(k-l)},
$$
for $k=1,2,\dots$, and $Q_0 = 0$.
Then, $\mu_{k,\eta}^x$ is the Gaussian process on $\cH$ with mean $S^k_\eta x$ and covariance operator $Q_k$,
%
and $\mu_{k,\eta}$ is the centered Gaussian process on $\cH$ with the same covariance operator. By the Cameron-Martin formula, $\mu^x_{k,\eta}$ and $\mu_{k,\eta}$ are equivalent with density given by 
%
\begin{equation}
    \frac{\dd\mu^x_{k,\eta}}{\dd\mu_{k,\eta}}(y) = \exp\left\{\scal{ Q_k^{-1} S^k_\eta x}{y} - \frac{1}{2}\norm{Q_k^{-1/2} S^k_\eta x}^2\right\},
\end{equation}
(see \citet{da_prato_zabczyk_1996} for example).
We can easily check that $Q_k  \succeq k Q S_\eta^{2 k}$. 
Then, we have that 
\begin{align*}
\langle x, S_\eta^k Q_k^{-1} y\rangle
 - \frac{1}{2}\norm{Q_k^{-1/2} S^k_\eta x}^2
\geq - \frac{\beta}{2}\|x\|^2 - \frac{1}{2\beta}\|S_\eta^k Q_k^{-1} y\|^2 - \frac{\beta}{4 \eta k} \|x\|^2.
\end{align*}

and thus we have the following lower bound of the density:
\begin{align}
\frac{\dd\mu^x_{k,\eta}}{\dd\mu_{k,\eta}}(y) 
\geq \exp\left\{
- \frac{\beta}{2}\left(1 + \frac{1}{2k \eta}\right)\|x\|^2 - \frac{1}{2\beta}\|S_\eta^k Q_k^{-1} y\|^2
\right\}.
\label{eq:dmukdmukxratioLowerBound}
\end{align}

For a given $N$ (where $N$ will be determined later on), let 
$$ 
K_k \eqdef Q_k S_{\eta}^{N-k} Q_N^{-1/2},
$$
for $k = 0,\dots,N$.
Here, we define 
$$
\widehat{Z}_{k,\eta}^{x,y} \eqdef Z_{k,\eta}^x - K_k Q_N^{-1/2}(Z_{N,\eta}^x - y),
$$
for $x,y \in \cH$, and denote $\widehat{Z}_{k,\eta} \eqdef \widehat{Z}_{k,\eta}^{0,0}$. In particular, we notice that 
$$
\widehat{Z}_{k,\eta} = Z_{k,\eta} - K_k Q_N^{-1/2} Z_{N,\eta},
$$
by definition.
Let
$$
Y_k \eqdef \sum_{l=k}^{N-1} S_{\eta}^{N-l} Q^{1/2} \epsilon_l,
$$
$$
H_k \eqdef Q_{N-k}^{-1/2} S_{\eta}^{N-k} Q^{1/2}.
$$
By a simple calculation, we can show that 
$$
Y_k = Z_{N,\eta} - S^{N-k}_{\eta} Z_{k,\eta} = Q_{N-k} Q_N^{-1} Z_{N,\eta} - S_\eta^{N-k} \widehat{Z}_{k,\eta}.
$$

Finally, let 
$$
\alpha_k \eqdef Q_{N-k}^{-1/2} H_k Y_k = 
\underbrace{Q_{N-k}^{1/2} H_k Q_N^{-1}}_{\eqdef B_1(k)} Z_{N,\eta} - \underbrace{Q_{N-k}^{-1/2} H_k S_\eta^{N-k}}_{\eqdef B_2(k)} \widehat{Z}_{k,\eta},
$$
and accordingly, define
$$
\zeta_k \eqdef \epsilon_k - \alpha_k.
$$

Then, we can show that $(\widehat{Z}_{k,\eta})_k$ and $(\zeta_k)_k$ are independent of $Z_{N,\eta}$ by the same reasoning as \cite{Goldys_Maslowski06}.
To see this, we only have to show that their correlation is 0 because they are Gaussian process. 
First, we can show that\footnote{Here, for $x,y \in \cH$, the bounded linear operator $z \mapsto x \langle y,z\rangle$ is denoted by $xy^*$ for simplicity.} 
\begin{align*}
\bbE\left[\epsilon_k \alpha_{k'}^* \right] 
= Q_{N-k'}^{-1/2}H_{k'} \bbE\left[\epsilon_k Y_{k'}^*\right]
& = 
\begin{cases}
Q_{N-{k'}}^{-1/2}H_{k'}(Q^{1/2} S_\eta^{N-k} - S_\eta^{N-{k'}} Q^{1/2}S_\eta^{k' - k}) &(k' < k) \\
Q_{N-{k'}}^{-1/2}H_{k'}(Q^{1/2} S_\eta^{N-k} ) & (k' \geq k)
\end{cases} \\
&
=
\begin{cases}
0 &(k' < k) \\
Q_{N-{k'}}^{-1/2}H_{k'} Q_{N-k}^{1/2} H_k & (k' \geq k)
\end{cases}.
\end{align*}
For $k \leq k'$,
\begin{align*}
\bbE\left[\alpha_k \alpha_{k'}^* \right] 
& = Q_{N-k}^{-1/2}H_k \bbE\left[Y_k Y_{k'}^*\right] H_{k'}Q_{N-k'}^{-1/2} 
= Q_{N-k}^{-1/2}H_k \left(\sum_{l=k'}^{N-1} S_\eta^{2(N-l)} Q \right) H_{k'}Q_{N-k'}^{-1/2} \\
& = 
Q_{N-k}^{-1/2}H_k \left(\sum_{l=0}^{N-k'-1} S_\eta^{2(N-k'-l)} Q \right) H_{k'}Q_{N-k'}^{-1/2} \\
&=
Q_{N-k}^{-1/2}H_k Q_{N-k'} H_{k'}Q_{N-k'}^{-1/2} 
=
Q_{N-k}^{-1/2}H_k Q_{N-k'}^{1/2} H_{k'}.
\end{align*}
Hence, when $k < k'$, it holds that 
$$
\bbE[(\epsilon_k -\alpha_k)(\epsilon_{k'} -\alpha_{k'})^*]
=0,
$$
and when $k = k'$, we have that 
$$
\bbE[(\epsilon_k -\alpha_k)(\epsilon_{k} -\alpha_{k})^*] = \Id - H_k^2.
$$
Finally, we can see that 
\begin{align*}
\bbE[(\epsilon_k - \alpha_k)Z_{N,\eta}^*] 
& = Q^{1/2}S_\eta^{N-k} - \left\{ Q_{N-k}^{1/2}H_k Q_N^{-1}Q_N  
- Q_{N-k}^{-1/2}H_k S_\eta^{N-k}(Q_kS_\eta^{N-k} - K_kQ_N^{-1/2}Q_N) \right\}\\
& = Q^{1/2}S_\eta^{N-k} - Q^{1/2}S_\eta^{N-k} = 0,
\end{align*}
which indicates $\zeta_k = \epsilon_k - \alpha_k$ is independent of $Z_{N,\eta}$.
Furthermore, we have that 
\begin{align*}
& \bbE[Z_{N,\eta} (\widehat{Z}_{k,\eta}^{x,y} - \bbE[\widehat{Z}_{k,\eta}^{x,y}])^*] = 
\bbE[Z_{N,\eta} (\widehat{Z}_{k,\eta}^{x,y})^{*}] = 
\bbE[Z_{N,\eta}Z_{k,\eta}^*] - 
\bbE[Z_{N,\eta}Z_{N,\eta}^*Q_N^{-1/2}K_k] \\
& =Q \sum_{l=0}^{k-1} S_\eta^{k-l}S_\eta^{N-l} - Q_NQ_N^{-1/2}K_k
=Q_kS_\eta^{N-k} - Q_k S_\eta^{N-k} = 0.
\end{align*}
This also yields that $Z_{N,\eta}$ and $\widehat{Z}_{k,\eta}^{x,y}~(k=1,\dots,N-1)$ are independent.

As we have stated, we now show the minorization condition. 
Let $P^\eta_{n}(x,\cdot)$ be the probability measure of the law of $X_n$ with $X_0 = x$, 
then by the Girsanov's theorem, $P^\eta_{N}(x,\cdot)$ is absolutely continuous with respect to $\mu_{N,\eta}^x$ and the Radon-Nikodym density is given by 
$$
\frac{\dd P^\eta_N(x,\cdot)}{\dd \mu_{N,\eta}^x}(y)
= \bbE\left[ \exp\left\{\frac{\beta}{2\eta} \sum_{k=0}^{N-1} \left(\langle - \eta \nabla L(Z_{k,\eta}^x),\epsilon_k\rangle \sqrt{2\eta/\beta}
- \frac{\eta^2}{2} \|\nabla L(Z_{k,\eta}^x)\|^2  \right) \right\} \big| Z_{N,\eta}^x = y\right].
$$
The right hand side can be evaluated as 
\begin{align*}
& \bbE\left[ \exp\left\{\frac{\beta}{2\eta} \sum_{k=0}^{N-1} \left(\langle - \eta \nabla L(Z_{k,\eta}^x),\epsilon_k\rangle \sqrt{2\eta/\beta}
- \frac{\eta^2}{2} \|\nabla L(Z_{k,\eta}^x)\|^2  \right) \right\} \big| Z_{N,\eta} = y - S_\eta^N x \right] \\
=&
\bbE\left[ \exp\left\{\frac{\beta}{2\eta} \sum_{k=0}^{N-1} \left(\langle - \eta \nabla L(Z_{k,\eta}^x),\zeta_k\rangle \sqrt{2\frac{\eta}{\beta}} 
+ \langle - \eta \nabla L(Z_{k,\eta}^x), (B_1(k)Z_{N,\eta} - B_2(k) \widehat{Z}_{k,\eta} ) \rangle \sqrt{2\frac{\eta}{\beta}}  \right. \right. \right. \\
&
\left.\left.\left. - \frac{\eta^2}{2} \|\nabla L(Z_{k,\eta}^x)\|^2  \right) \right\} \big| Z_{N,\eta} = y - S_\eta^N x \right] \\
=&
\bbE\left[ \exp\left\{\frac{\beta}{2\eta} \sum_{k=0}^{N-1} \left(\langle - \eta \nabla L(\widehat{Z}_{k,\eta}^{x,y}),\zeta_k\rangle \sqrt{2\frac{\eta}{\beta}} \right. \right. \right. \\
& + \langle - \eta \nabla L(\widehat{Z}_{k,\eta}^{x,y}), B_1(k)(y - S^N_\eta x) - B_2(k) \widehat{Z}_{k,\eta}  \rangle \sqrt{\frac{2\eta}{\beta}}  
\left.\left.\left. - \frac{\eta^2}{2} \|\nabla L(\widehat{Z}_{k,\eta}^{x,y})\|^2  \right) \right\}  \right], 
\end{align*}
where we used the fact that $(\widehat{Z}_k)_k$ and $(\zeta_k)_k$ are independent of $Z_{N,\eta}$. 
Therefore, by Jensen's inequality, the right hand side is lower bounded by
\begin{align*}
 \exp\left\{\frac{\beta}{2\eta} \sum_{k=0}^{N-1} \left( 
  \bbE\left[\langle  - \eta \nabla L(\widehat{Z}_{k,\eta}^{x,y}), B_1(k)(y - S^N_\eta x) - B_2(k) \widehat{Z}_{k,\eta}  \rangle \right] \sqrt{\frac{2\eta }{\beta}}  
 - \frac{\eta^2}{2} \bbE[ \|\nabla L(\widehat{Z}_{k,\eta}^{x,y})\|^2]  \right) \right\}.
\end{align*}
Thus, by the assumption that $\|\nabla L(\cdot)\| \leq B$, the right hand side is lower bounded by 
\begin{align}
& \exp\left\{ - \sqrt{\frac{\beta}{2\eta}}   \sum_{k=0}^{N-1}
\left(\bbE\left[\langle  - \eta \nabla L(\widehat{Z}_{k,\eta}^{x,y}), B_1(k)(y - S^N_\eta x) \right]
+ \eta B \bbE[\|B_2(k) \widehat{Z}_{k,\eta}\|] \right)
- \frac{\beta \eta N}{2} B^2   \right\} \notag \\ 
& \geq \exp\left\{ 
\sqrt{\frac{\beta \eta}{2}} \sum_{k=0}^{N-1}
\bbE\left[\langle   \nabla L(\widehat{Z}_{k,\eta}^{x,y}), B_1(k)(y - S^N_\eta x) \right]
- \frac{\beta \eta N}{2}B^2  - 
\sum_{k=0}^{N-1} \bbE[\|B_2(k) \widehat{Z}_{k,\eta}\|^2]
- \frac{\beta \eta N}{2} B^2  \right\} \notag \\
& \geq \exp\left\{ 
\sqrt{\frac{\beta \eta}{2}} \sum_{k=0}^{N-1}
\bbE\left[\langle   \nabla L(\widehat{Z}_{k,\eta}^{x,y}), B_1(k)(y - S^N_\eta x) \right]
- \beta \eta N B^2  - 
 \sum_{k=0}^{N-1} \bbE[\|B_2(k) \widehat{Z}_{k,\eta}\|^2] \right\}.
\label{eq:dmukdPetakLowerBound}
\end{align}

For $z \in \cH$, we have 
\begin{align*}
& \sqrt{\frac{\beta}{2\eta}} \sum_{k=0}^{N-1}\bbE\left[\langle   \eta \nabla L(\widehat{Z}_{k,\eta}^{x,y}), B_1(k) z \rangle \right] \\
&=
\sqrt{\frac{\beta}{2\eta}} \sum_{k=0}^{N-1}\bbE\left[\langle   \eta \nabla L(0), B_1(k) z \rangle \right]
+
\sqrt{\frac{\beta}{2\eta}}  \sum_{k=0}^{N-1}\bbE\left[\langle   \eta (\nabla L(\widehat{Z}_{k,\eta}^{x,y}) - \nabla L(0)), B_1(k) z \rangle \right]
\\
&=
\sqrt{\frac{\beta \eta}{2}} \left\langle \left (\sum_{k=0}^{N-1} B_1(k)\right) \nabla L(0), z \right\rangle 
+
\sqrt{\frac{\beta \eta}{2}}  \sum_{k=0}^{N-1}\bbE\left[\langle    (\nabla L(\widehat{Z}_{k,\eta}^{x,y}) - \nabla L(0)), B_1(k) z \rangle \right].
\end{align*}
The first term of the right hand side can be lower bounded by
$$
- \frac{\beta \eta N}{4}
-
\frac{1}{2}
\sum_{k=0}^{N-1} \left\langle  B_1(k) \nabla L(0), z \right\rangle ^2
=
- \frac{\beta \eta N}{4}
-
\frac{1}{2}
\sum_{k=0}^{N-1} \left\langle  Q S_\eta^{N-k} Q_N^{-1} \nabla L(0), z \right\rangle ^2.
$$ 
The second term can be evaluated as
\begin{align*}
\sqrt{\frac{\beta \eta}{2}}  \sum_{k=0}^{N-1}\bbE\left[\langle   (\nabla L(\widehat{Z}_{k,\eta}^{x,y}) - \nabla L(0)), B_1(k) z \rangle \right]
=
\sqrt{\frac{\beta \eta}{2}}  \sum_{k=0}^{N-1}\bbE\left[\langle  (D^2 L(\widetilde{Z}_{k,\eta}^{x,y}) \cdot \widehat{Z}_{k,\eta}^{x,y}), B_1(k) z \rangle \right],
\end{align*}
where $\widetilde{Z}_{k,\eta}^{x,y}$ is an intermediate point between $\widehat{Z}_{k,\eta}^{x,y}$ and $0$, i.e., there exists $\theta \in [0,1]$ such that $\widetilde{Z}_{k,\eta}^{x,y}=\theta \widehat{Z}_{k,\eta}^{x,y}$.
By Assumption \ref{assum:C1_boundedness}, this can be further evaluated as 
\begin{align*}
&\sqrt{\frac{\beta \eta}{2}}  \sum_{k=0}^{N-1}\bbE\left[\langle  (D^2 L(\widetilde{Z}_{k,\eta}^{x,y}) \cdot \widehat{Z}_{k,\eta}^{x,y}), B_1(k) z \rangle \right]
\geq 
-\sqrt{\frac{\beta \eta}{2}}  \sum_{k=0}^{N-1}
C_{\alpha,2}
\bbE\left[\|\widehat{Z}_{k,\eta}^{x,y}\|_{\cH} \|B_1(k)z\|_\alpha\right] \\ 
&\geq
-\frac{\beta \eta }{4} C_{\alpha,2}^2 
\sum_{k=0}^{N-1}\bbE\left[\|\widehat{Z}_{k,\eta}^{x,y}\|_{\cH}^2\right]
-\frac{1}{2}
\sum_{k=0}^{N-1}\|B_1(k)z\|^2_\alpha \\
& =
-\frac{\beta \eta}{4} C_{\alpha,2}^2 
\sum_{k=0}^{N-1}\bbE\left[\|\widehat{Z}_{k,\eta}^{x,y}\|_{\cH}^2\right]
-\frac{1}{2}
\sum_{k=0}^{N-1}\|Q S_\eta^{N-k} Q_N^{-1} z\|^2_\alpha.
\end{align*}
Here, we have 
\begin{align*}
& \bbE\left[\|\widehat{Z}_{k,\eta}^{x,y}\|_{\cH}^2\right]
= \mathrm{Tr}[Q_k Q_{N-k} Q_N^{-1}] + \|(S_\eta^{k}  - K_k Q_N^{-1/2}S_\eta^N)x + K_kQ_N^{-1/2}y\|_{\cH}^2 \\
& \leq \mathrm{Tr}[Q S_\eta^2(\Id - S_\eta^{2N})^{-1}]
+ 2 \|x\|_{\cH}^2 + 2\|y\|_{\cH}^2,
\end{align*}
where we used $\|S_\eta^{k}  - K_k Q_N^{-1/2}S_\eta^N\| \leq 1$ and $\|K_kQ_N^{-1/2}\| \leq 1$.
Therefore, we obtain 
\begin{align*}
\sqrt{\frac{\beta}{2\eta}} \sum_{k=0}^{N-1}\bbE\left[\langle  \eta \nabla L(\widehat{Z}_{k,\eta}^{x,y}), B_1(k) z \rangle \right] 
\geq &
- \frac{\beta \eta N}{4}
-
\frac{\beta\eta}{4}C_{\alpha,2}^2
\sum_{k=0}^{N-1}
(\mathrm{Tr}[Q S_\eta^2(\Id - S_\eta^{2N})^{-1}]
+ 2 \|x\|_{\cH}^2 + 2\|y\|_{\cH}^2) \\
&
-
\frac{1}{2}
\sum_{k=0}^{N-1} 
(\left\langle  Q S_\eta^{N-k} Q_N^{-1} \nabla L(0), z \right\rangle ^2
+\|Q S_\eta^{N-k} Q_N^{-1} z\|^2_\alpha).
\end{align*}
Next we give another bound for $z = S_\eta^N x$. 
In this situation, thanks to the factor $S_\eta^N$, we have a simpler bound:
\begin{align*}
& \sqrt{\frac{\beta}{2\eta}} \sum_{k=0}^{N-1}\bbE\left[\langle   \eta \nabla L(\widehat{Z}_{k,\eta}^{x,y}), B_1(k) S^N_\eta x \rangle \right]  
 \geq - \sqrt{\frac{\beta \eta}{2}} 
\sum_{k=0}^{N-1}B \| B_1(k) S^N_\eta x \| 
\geq 
- \frac{\beta \eta N}{4} B^2 + 
\frac{1}{2}\sum_{k=0}^{N-1} \| B_1(k) S^N_\eta x \|^2.
\end{align*}
Notice that
\begin{align*}
& 
 \sum_{k=0}^{N-1} B_1(k)^2
= \sum_{k=0}^{N-1} (Q_{N-k}^{1/2} H_k Q_N^{-1})^2 = \sum_{k=1}^{N-1} Q_{N-k} H_k^2 Q_N^{-2}
=
\sum_{k=0}^{N-1} Q_{N-k} Q_{N-k}^{-1}S_\eta^{2(N-k)} Q Q_N^{-2} \\
& =
\sum_{k=0}^{N-1} S_\eta^{2(N-k)} Q Q_N^{-2}
=
Q_N^{-1}.
\end{align*}
Therefore, $\sum_{k=0}^{N-1}\|B_1(k) S^N_\eta x \|^2$ can be bounded as 
\begin{align*}
& 
\sum_{k=0}^{N-1}\|B_1(k) S^N_\eta x\|^2 =
 \|Q_N^{-1/2} S_\eta^{N} x\|^2  
\leq 
 \frac{1}{N}Q^{-1} \|x\|^2 
=
 \frac{\beta}{2N\eta} \|x\|^2,
\end{align*}
where we used $Q_N \succeq N Q S_\eta^{2 N}$ and $Q = \frac{2\eta}{\beta}\Id$.
Therefore, we have 
\begin{align*}
& \sqrt{\frac{\beta}{2\eta}} \sum_{k=0}^{N-1}\bbE\left[\langle   \eta \nabla L(\widehat{Z}_{k,\eta}^{x,y}), B_1(k) S^N_\eta x \rangle \right]  
 \geq - \frac{\beta \eta N}{4} B^2 - \frac{\beta}{4N\eta} \|x\|^2.
\end{align*}
%
%
\begin{align*}
& \sum_{k=0}^{N-1} \bbE[\|B_2(k) \widehat{Z}_{k,\eta}\|^2]
=
\sum_{k=0}^{N-1} \mathrm{Tr}[(Q_{N-k}^{-1} S_\eta^{2(N-k)}Q^{1/2})^2(Q_k - 2Q_k^2 Q_N^{-1}S_\eta^{2(N-k)} +Q_N)] \\
& 
\leq
\sum_{k=0}^{N-1} \mathrm{Tr}[(Q_{N-k}^{-1} S_\eta^{2(N-k)}Q^{1/2})^2(Q_k - 2Q_k Q_N^{-1}Q_N +Q_N)]
=
\sum_{k=0}^{N-1} \mathrm{Tr}[(Q_{N-k}^{-1} S_\eta^{2(N-k)}Q^{1/2})^2(Q_N - Q_k)]
\\
&
\leq 
\sum_{k=0}^{N-1} \mathrm{Tr}\left[Q_{N-k}^{-2} S_\eta^{4(N-k)}Q
\left(\sum_{l=0}^{N-k-1} S_\eta^{N-l}\right)\right] =\sum_{k=0}^{N-1} \mathrm{Tr}[Q_{N-k}^{-2} S_\eta^{4(N-k)}Q Q_{N-k} S_\eta^{2k}] \\
&
=\sum_{k=0}^{N-1} \mathrm{Tr}[Q_{N-k}^{-1} S_\eta^{2N} S_\eta^{2(N-k)}Q  ]
=\sum_{k=0}^{N-1} \mathrm{Tr}\{(S_\eta^{-2} - \Id)[Q(\Id - S_\eta^{2(N-k)})]^{-1} S_\eta^{2N} S_\eta^{2(N-k)}Q  \} \\
&
= \mathrm{Tr}\left[(S_\eta^{-2} - \Id)S_\eta^{2N} \sum_{k=0}^{N-1} (S_\eta^{-2(N-k)} - \Id)^{-1}   \right]
\leq 
\mathrm{Tr}\left[(S_\eta^{-2} - \Id)S_\eta^{2N} (S_\eta^{-2} - \Id)^{-1} \sum_{k=0}^{N-1} S_\eta^{2k}  \right] \\
&=
\mathrm{Tr}\left[(S_\eta^{-2} - \Id)S_\eta^{2N} (S_\eta^{-2} - \Id)^{-1} (S_\eta^{2N} - \Id)(S_\eta^{2} - \Id)^{-1}  \right] 
=
\mathrm{Tr}\left[S_\eta^{2N} (S_\eta^{2N} - \Id)(S_\eta^{2} - \Id)^{-1}  \right] \\
& \leq
\mathrm{Tr}\left[ (S_\eta^{4N} -S_\eta^{2N})(S_\eta^{2} - \Id)^{-1}  \right]
\leq
\mathrm{Tr}\left[ (S_\eta^{2N+2} -S_\eta^{2N})(S_\eta^{2} - \Id)^{-1}  \right] ~~(\because N \geq 1)\\
& \leq 
\mathrm{Tr}\left[ S_\eta^{2N}\right]
\leq \mathrm{Tr}\left[ (\Id + 2N \eta A)^{-1} \right].
\end{align*}
Therefore, we obtain, for all $y \in \mathrm{Im}(Q_N^{1/2})$, 
\begin{align*}
\frac{\dd P^\eta_N(x,\cdot)}{\dd \mu_{N,\eta}^x}(y)
\geq 
\exp  \Bigg\{ &
- \frac{\beta \eta N}{4}
-
\frac{\beta\eta}{4}C_{\alpha,2}^2
\sum_{k=0}^{N-1}
(\mathrm{Tr}[Q S_\eta^2(\Id - S_\eta^{2N})^{-1}]
+ 2 \|x\|_{\cH}^2 + 2\|y\|_{\cH}^2)  \\
&~~~
-
\frac{1}{2}
\sum_{k=0}^{N-1} 
(\left\langle  Q S_\eta^{N-k} Q_N^{-1} \nabla L(0), y \right\rangle ^2
+\|Q S_\eta^{N-k} Q_N^{-1} y\|^2_\alpha) \\
&~~~ - \frac{\beta \eta N}{4} B^2 - \frac{\beta}{4N\eta} \|x\|^2  \\
&~~~  - \beta \eta N B^2 -\mathrm{Tr}\left[ (\Id + 2N \eta A)^{-1} \right] \Bigg\} \\
\geq \exp\Bigg\{ & 
\underbrace{- \frac{\beta \eta N ( 1 + 5B^2)}{4}
- \frac{\beta\eta N }{4}C_{\alpha,2}^2 
\mathrm{Tr}[Q S_\eta^2(\Id - S_\eta^{2N})^{-1}]  -\mathrm{Tr}\left[ (\Id + 2N \eta A)^{-1} \right]}_{=: - C_{\eta,N,\beta}} \\
& \underbrace{- \left( \frac{\beta\eta N }{2}C_{\alpha,2}^2 + \frac{\beta}{4N\eta}\right) \|x\|^2}_{=: - \tilde{\Lambda}_{\mathrm{x}}(x)} \\
& 
\underbrace{- \frac{\beta\eta N }{2}C_{\alpha,2}^2 \|y\|^2 - 
\frac{1}{2}
\sum_{k=0}^{N-1} 
(\left\langle  Q S_\eta^{N-k} Q_N^{-1} \nabla L(0), y \right\rangle ^2
+\|Q S_\eta^{N-k} Q_N^{-1} y\|^2_\alpha)}_{=: - \tilde{\Lambda}_{\mathrm{y}}(y)} \Bigg\}.
\end{align*}
Combining the inequalities \eqref{eq:dmukdmukxratioLowerBound} and \eqref{eq:dmukdPetakLowerBound}, we finally obtain that
\begin{align}
&\frac{\dd P^\eta_N(x,\cdot)}{\dd \mu_{N,\eta}}(y)
= 
\frac{\dd P^\eta_N(x,\cdot)}{\dd \mu_{N,\eta}^x}(y)
\frac{\dd \mu_{N,\eta}^x}{\dd \mu_{N,\eta}}(y) \notag \\
& \geq 
\exp\left\{ 
- \frac{\beta}{2}\left(1 + \frac{1}{2\eta N}\right)\|x\|^2 - \frac{1}{4\beta}\|S_\eta^N Q_N^{-1} y\|^2
- C_{\eta,N,\beta} - \tilde{\Lambda}_{\mathrm{x}}(x) - \tilde{\Lambda}_{\mathrm{y}}(y)
\right\}.
\label{eq:PNetaLowerBoundBasic}
\end{align}
Fro now on, we give a lower bound of the right hand side.
To do so, we set $N = 1/\eta$.
Under this setting, let $\Lambda_{\mathrm{x}}(x) :=  \frac{\beta}{4}\|x\|^2 + \tilde{\Lambda}_{\mathrm{x}}(x)$ and $\Lambda_{\mathrm{y}}(y) :=  \frac{1}{4\beta}\|S_\eta^N Q_N^{-1} y\|^2 + \tilde{\Lambda}_{\mathrm{y}}(y)$, i.e.,
\begin{align}
\frac{\dd P^\eta_N(x,\cdot)}{\dd \mu_{N,\eta}}(y) \geq 
\exp\left\{ 
- C_{\eta,N,\beta} 
- \Lambda_{\mathrm{x}}(x)
- \Lambda_{\mathrm{y}}(y)
\right\}.
\end{align}
We evaluate the terms in the exponent in the right hand side one by one.

(i) (Bound of $C_{\eta,N,\beta}$): 
Note that 
\begin{align}
\|(\mathrm{Id} - S_\eta^{2N})^{-1}\|_{\cB(\cH)} & \leq 
[1 - (1 + \eta \lambda/\mu_0)^{-2N}]^{-1} 
\leq 
(1 + \eta \lambda/\mu_0)^{2N} [ (1 + \eta \lambda/\mu_0)^{2N} - 1]^{-1} \notag \\
& \leq 
\frac{\exp(2N \eta \lambda/\mu_0)}{2N \eta \lambda/\mu_0} = \frac{\exp(2\lambda/\mu_0)}{2\lambda/\mu_0},
\label{eq:LdSetaTwoBound}
\end{align}
and thus 
\begin{align*}
\mathrm{Tr}[QS_\eta^2(\mathrm{Id} - S_\eta^{2N})^{-1}] 
& = \frac{2 \eta}{\beta} \mathrm{Tr}[S_\eta^2(\mathrm{Id} - S_\eta^{2N})^{-1}] 
\leq 
\frac{2 \eta}{\beta} \mathrm{Tr}[S_\eta^2] \|S_\eta^2(\mathrm{Id} - S_\eta^{2N})^{-1}\|_{\cB(\cH)} \\
&\leq  \frac{2 \eta}{\beta}   \frac{\exp(2 \lambda/\mu_0)}{ 2 \lambda/\mu_0} \sum_{k=0}^\infty (1 + \eta \lambda/ \mu_k)^{-2} 
\leq C_\mu \frac{2 \eta}{\beta}   \frac{\exp(2 \lambda/\mu_0)}{ 2 \lambda/\mu_0} \sqrt{\frac{1}{\eta \lambda}} \\
& = C_\mu \frac{\sqrt{\eta}}{\beta} \frac{ \mu_0 \exp(2 \lambda/\mu_0)}{\lambda^{3/2}},  
\end{align*}
where $C_\mu$ is a constant depending on $(\mu_k)_{k=1}^\infty$ and we used $\mu_k \lesssim 1/k^2$ in the last inequality.
This converges to 0 as $\eta \to 0$ and $\beta \to \infty$, thus $\mathrm{Tr}[QS_\eta^2(\mathrm{Id} - S_\eta^{2N})^{-1}]  = O(1)$.
Consequently, we have 
$$ 
C_{\eta,N,\beta} = \frac{\beta(1 + 5B^2)}{4} + \frac{1}{4} C_{\alpha,2}^2 
 C_\mu \sqrt{\eta} \frac{ \mu_0 \exp(2 \lambda/\mu_0)}{\lambda^{3/2}}
+ \mathrm{Tr}\left[ (\Id + 2 A)^{-1}\right] = O(\beta).
$$
(ii) (Bound of $\Lambda_{\mathrm{x}}(x)$): 
By the definition of $\Lambda_{\mathrm{x}}(x)$, it holds that 
$$
\Lambda_{\mathrm{x}}(x) = \left(\frac{\beta}{2}\left(1 + \frac{1}{2\eta N}\right) + \frac{\beta\eta N }{2}C_{\alpha,2}^2 + \frac{\beta}{4N\eta} \right)\|x\|^2
=\left(\beta + \frac{\beta}{2}C_{\alpha,2}^2  \right)\|x\|^2 = O(\beta \|x\|^2).
$$
(ii) (Bound of $\Lambda_{\mathrm{y}}(y)$): Finally, we evaluate $\Lambda_{\mathrm{y}}(y)$. When $\eta = 1/N$, 
\begin{align*}
\Lambda_{\mathrm{y}}(y)
=
\frac{1}{4\beta}\|S_\eta^N Q_N^{-1} y\|^2 + 
\frac{\beta }{2}C_{\alpha,2}^2 \|y\|^2 +
\frac{1}{2}
\sum_{k=0}^{N-1} 
(\left\langle  Q S_\eta^{N-k} Q_N^{-1} \nabla L(0), y \right\rangle ^2
+\|Q S_\eta^{N-k} Q_N^{-1} y\|^2_\alpha).
\end{align*}
We can show that $\Lambda_{\mathrm{y}}(Z) < \infty$ for $Z \sim \mu_{N,\eta}$ almost surely, as follows. 
Since $0 \leq \Lambda_{\mathrm{y}}(y)$, we only have to evaluate $\bbE_{Z \sim \mu_{N,\eta}}[\Lambda_{\mathrm{y}}(Z)]$.
To do so, we note that $\mu_{N,\eta}$ is a Gaussian process in $\cH$ with mean $0$ and covariance $Q_N$, which can be easily checked by its definition. By using this, we evaluate the expectation of each term as follows.
\begin{align*}
& \bbE_{Z \sim \mu_{N,\eta}}\left[\frac{1}{4\beta}\|S_\eta^N Q_N^{-1} y\|^2 \right]
= 
\frac{1}{4\beta} \mathrm{Tr}[S_\eta^{2N} Q_N^{-2} Q_N]= \frac{1}{4\beta} \mathrm{Tr}[S_\eta^{2N} Q_N^{-1}] 
= 
\frac{1}{4\beta} \mathrm{Tr}[S_\eta^{2N}(S_\eta^2 + \dots + S_\eta^{2N})^{-1}Q^{-1}] \\
& = 
\frac{1}{4\beta} \mathrm{Tr}[Q^{-1} S_\eta^{2N}(\Id - S_\eta^2)(S_\eta^2(\Id - S_\eta^{2N}))^{-1}] 
\leq 
\frac{1}{4\beta} \mathrm{Tr}[Q^{-1} S_\eta^{2N}(\Id - S_\eta^2)(S_\eta^2)^{-1}] \frac{\exp(2 \lambda/\mu_0)}{2 \lambda/\mu_0}
~~~(\text{$\because$ \Cref{eq:LdSetaTwoBound}})
\\ 
& =
\frac{1}{4\beta} \mathrm{Tr}[Q^{-1} S_\eta^{2N} (S_\eta^{-2} - \Id)] \frac{\exp(2 \lambda/\mu_0)}{2 \lambda/\mu_0}
\leq 
\frac{1}{8\eta} \mathrm{Tr}[ S_\eta^{2N} (2 \eta A + \eta^2 A^2)] \frac{\exp(2 \lambda/\mu_0)}{2 \lambda/\mu_0} 
\\
& 
=
 \frac{1}{8}\frac{\exp(2 \lambda/\mu_0)}{2 \lambda/\mu_0}
\sum_{k=0}^\infty \frac{2 \lambda/\mu_k + \eta(\lambda/\mu_k)^2}{(1 + \eta \lambda/\mu_k)^{2N}}
=
 \frac{1}{8}\frac{\exp(2 \lambda/\mu_0)}{2 \lambda/\mu_0}
\sum_{k=0}^\infty  \frac{1}{(1 + \eta \lambda/\mu_k)^{N}}
\left( \frac{2 \lambda/\mu_k}{(1 + \eta \lambda/\mu_k)^{N}} 
+ 
\frac{\eta(\lambda/\mu_k)^2}{ (1 + \eta \lambda/\mu_k)^{\frac{N}{2}2}}
\right) \\
& \leq
 \frac{1}{8}\frac{\exp(2 \lambda/\mu_0)}{2 \lambda/\mu_0}
\sum_{k=0}^\infty  \frac{1}{(1 + \lambda/\mu_k)}
\left( \frac{2 \lambda/\mu_k}{(1 +  \lambda/\mu_k)} 
+ 
\frac{\eta(\lambda/\mu_k)^2}{ (1 + \frac{1}{2} \lambda/\mu_k)^{2}}
\right)
\leq 
 \frac{1}{8}\frac{\exp(2 \lambda/\mu_0)}{2 \lambda/\mu_0}
\sum_{k=0}^\infty  \frac{1}{(1 + \lambda/\mu_k)} (2 + 4\eta) \\
& \leq 
 \frac{1 + 2\eta}{4}\frac{\exp(2 \lambda/\mu_0)}{2 \lambda/\mu_0} \frac{C'_\mu}{\sqrt{\lambda}} = O(1),
\end{align*}
where $C'_\mu$ is a constant depending only on $(\mu_k)_k$ and we again used $\mu_k \lesssim 1/k^2$ in the last inequality.
\begin{align*}
& \bbE_{Z \sim \mu_{N,\eta}}\left[ \frac{\beta }{2}C_{\alpha,2}^2 \|Z\|^2 \right] = 
\frac{\beta }{2}C_{\alpha,2}^2 \mathrm{Tr}[Q_N] = 
\frac{\beta }{2}C_{\alpha,2}^2 \mathrm{Tr}[Q(S_\eta^2 - S_\eta^{2(N+1)})(\Id - S_\eta^2)^{-1}] \\
& = \eta C_{\alpha,2}^2 \mathrm{Tr}[(\Id - S_\eta^{2N})(S_\eta^{-2} - \Id)^{-1}]
= \eta C_{\alpha,2}^2 \sum_{k=0}^\infty (1 - (1 + \eta \lambda/\mu_k)^{-2N})(2 \eta \lambda/\mu_k + \eta^2 (\lambda/\mu_k)^2)^{-1} \\
& \leq \eta C_{\alpha,2}^2 \sum_{k=0}^\infty (2 \eta \lambda/\mu_k )^{-1}
=  \frac{C_{\alpha,2}^2}{2 \lambda} \sum_{k=0}^\infty \mu_k \leq \frac{C_{\alpha,2}^2}{2 \lambda} C''_\mu = O(1),
\end{align*}
where $C''_\mu$ is a constant depending only on $(\mu_k)_k$ and we again used $\mu_k \lesssim 1/k^2$ in the last inequality.
\begin{align*}
& \bbE_{Z \sim \mu_{N,\eta}} \left[\sum_{k=0}^{N-1} 
(\left\langle  Q S_\eta^{N-k} Q_N^{-1} \nabla L(0), Z \right\rangle ^2
+\|Q S_\eta^{N-k} Q_N^{-1} Z \|^2_\alpha) \right]  \\
& = 
\sum_{k=0}^{N-1} \left\{
\left\langle  Q S_\eta^{N-k} Q_N^{-1} \nabla L(0), Q_N Q S_\eta^{N-k} Q_N^{-1} \nabla L(0) \right\rangle 
+\mathrm{Tr}[ Q S_\eta^{N-k} Q_N^{-1} \sqrt{Q_N} \Id_\alpha  \sqrt{Q_N} Q S_\eta^{N-k} Q_N^{-1} ]
\right\}
\end{align*}
where $\Id_\alpha: \cH \to \cH$ is a linear operator defined by $\langle x, \Id_\alpha y\rangle = \sum_{k=0}^\infty (\mu_k)^{2\alpha}x_k y_k$
for $x = (x_k)_k, y = (y_k)_k \in \cH$.
The first term in the right hand side can be evaluated as 
\begin{align*}
\left\langle  Q S_\eta^{N-k} Q_N^{-1} \nabla L(0), Q_N Q S_\eta^{N-k} Q_N^{-1} \nabla L(0) \right\rangle 
= \left\langle  Q^2 S_\eta^{2(N-k)} Q_N^{-1} \nabla L(0),  \nabla L(0) \right\rangle,
\end{align*}
and its summation becomes
\begin{align*}
\sum_{k=0}^\infty \left\langle  Q^2 S_\eta^{2(N-k)} Q_N^{-1} \nabla L(0),  \nabla L(0) \right\rangle
= \left\langle  Q Q_N Q_N^{-1} \nabla L(0),  \nabla L(0) \right\rangle = 
\left\langle  Q \nabla L(0),  \nabla L(0) \right\rangle = \frac{2\eta}{\beta} \|L(0)\|^2 = O(\eta/\beta).
\end{align*}
The second term can be evaluated as 
\begin{align*}
& \sum_{k=0}^\infty 
\mathrm{Tr}[ Q S_\eta^{N-k} Q_N^{-1} \sqrt{Q_N} \Id_\alpha  \sqrt{Q_N} Q S_\eta^{N-k} Q_N^{-1} ]
=
\sum_{k=0}^\infty 
\mathrm{Tr}[Q^2 S_\eta^{2(N-k)} Q_N^{-1}\Id_\alpha] \\
= & 
\sum_{k=0}^\infty 
\mathrm{Tr}[Q Q_N Q_N^{-1}\Id_\alpha]
=
\sum_{k=0}^\infty 
\frac{2\eta }{\beta}\mathrm{Tr}[\Id_\alpha]
= 
\frac{2\eta }{\beta}\sum_{k=0}^\infty \mu_k^{2\alpha} = \frac{2\eta}{\beta} C_{\mu,\alpha} = O(\eta/\beta),
\end{align*}
where we used the assumption $\alpha > 1/4$ and $\mu_k \lesssim 1/k^2$.
Summarizing the above arguments, we obtain that 
\begin{align}
\bbE_{Z \sim \mu_{N,\eta}}[\Lambda_{\mathrm{y}}(Z)] \leq O(1).
\label{eq:ELambdaZisFinite}
\end{align}

(iv) (Combining all bounds (i), (ii), (iii)). Combining these bounds for $C_{\eta,N,\beta}, \Lambda_{\mathrm{x}}(x), \Lambda_{\mathrm{y}}(y)$, we may give a lower bound of $P^\eta_N(x,\Gamma)$ for a measurable set $\Gamma \subset \cH$
uniformly for all $x$ with norm smaller than a given $R$, which is required to show the minorization condition.
Let 
$$
c_R \eqdef  \exp\left(-C_{\eta,N,\beta} - \frac{\beta}{2}(2 + C_{\alpha,2}^2) R^2\right)
$$ 
for $R \geq \frac{3}{2} k(1)$ which will be determined later, 
then we have shown that for all $x \in \cH$ with $\|x\| \leq R$,
$$
\exp(-C_{\eta,N,\beta} - \Lambda_{\mathrm{x}}(x)) \geq c_R.
$$
%
By \Cref{eq:PNetaLowerBoundBasic}, this gives that 
\begin{equation*}
   P^\eta_{N}(x, \Gamma) \geq c_R\! \int_\Gamma\!  e^{-\Lambda_{\mathrm{y}} (z)} \mu_{N,\eta} (dz),
\end{equation*}
for all $x \in \cB_R$ and a measurable set $\Gamma \subset \cH$.
In particular, if we define 
\begin{equation*}
    \bar{\mu}(\Gamma) \eqdef \frac{1}{\bar{Z}} \int_{\Gamma\cap \cB_R} e^{-\Lambda_{\mathrm{y}} (z)} \mu_{N,\eta}(dz)
\end{equation*}
where $\bar{Z}= \int_{\cB_R} e^{-\Lambda_{\mathrm{y}} (z)} \mu_{N,\eta}(dz) $ so that $\bar{\mu}$ is a probability measure, then
\begin{align}
 P^\eta_{N}(x, \Gamma) 
& 
\geq c_R\! \int_{\Gamma \cap \cB_R}  e^{-\Lambda_{\mathrm{y}} (z)} \mu_{N,\eta} (dz) 
\geq \delta \bar{\mu}(\Gamma),
\label{eq:PnDeltamuLowerbound}
\end{align}
where $$\delta  \eqdef c_R \bar{Z}.$$
Here, we give a lower bound of $\delta$. By Proposition \ref{prop:z_boundedness}, 
\begin{align*}
\mu_{N,\eta}(\cB_R) \geq 1 - \frac{\bbE_{Z \sim \mu_{N,\eta}}[\|Z\|]]}{R} \geq 1 - \frac{1}{R} k(1) \geq \frac{1}{3}, 
\end{align*}
where we used $R \geq \frac{3}{2}k(1)$
and thus $\delta$  can be lower bounded as
    \begin{align*}
    \delta  &= c_R  \int_{\cB_R} e^{-\Lambda_{\mathrm{y}} (z)} \mu_{N,\eta}(dz) = 
c_R \mu_{N,\delta}(\cB_R)  \frac{1}{\mu_{N,\delta}(\cB_R)} \int_{\cB_R} e^{-\Lambda_{\mathrm{y}} (z)} \mu_{N,\eta}(dz) 
\\
    &
\geq 
c_R \mu_{N,\delta}(\cB_R) 
\exp\left( - \frac{1}{\mu_{N,\delta}(\cB_R)} \int_{\cB_R} \Lambda_{\mathrm{y}} (z) \mu_{N,\eta}(dz)  \right) \\
& 
\geq 
\frac{1}{3} c_R
\exp\left( - 2  \int_{\cH} \Lambda_{\mathrm{y}} (z) \mu_{N,\eta}(dz)  \right) \\
& 
\geq 
\frac{1}{3} c_R \exp\left( -  O(1)  \right),
    \end{align*}
where we used Eq.~\eqref{eq:ELambdaZisFinite} in the final inequality.
Therefore, we have shown that there exists a probability measure $\bar{\mu}$, with $\bar{\mu}(\cB_R) = 1$ and $\bar{\mu}(\cB_R^c) = 0$,
such that Eq.~\eqref{eq:PnDeltamuLowerbound} is satisfied for any $x \in \cB_R$ and a measurable set $\Gamma \in \mathbb{B}(\cH)$, where $\delta \geq \frac{1}{3} c_R \exp\left( -  O(1)  \right) \geq \frac{1}{3} \exp(-C_{\eta,N,\beta} - \Lambda_{\mathrm{x}}(x) -O(1)) \gtrsim \exp(- O(\beta))$.

    %
%
%


By Proposition \ref{prop:norm_control}, the following contraction condition holds for $\alpha_N = \rho^N = (\frac{1}{1 + \lambda\eta/\mu_0})^N 
\leq \exp(- \lambda/\mu_0) < 1$, $\bar{b} = \max\{\frac{\mu_0}{\lambda}B +k(1),1\}$ under the bounded gradient condition:
\begin{equation*}
\bbE_{x_0}\norm{X_N} \leq \alpha_N \norm{x_0} + \bar{b} \quad (\forall n\in\bbN).
\end{equation*}

Set $V(x) = \|x\| + 1$ and $\cC = \left\{x \in \cH \mid V(x) \leq \frac{2 \bar{b}}{\sqrt{(1+\alpha_N)/2} - \alpha_N}\right\}$, then
we have that $\cC = \cB_R$ for $R = \frac{2 \bar{b}}{\sqrt{(1+\alpha_N)/2} - \alpha_N} - 1$.
Here, we give lower and upper bounds of $R$. 
As for the lower bound, we can easily see that $R \geq \frac{5}{2} \bar{b} -1 \geq \frac{3}{2} \bar{b} \geq \frac{3}{2} k(1)$.
Next, we give an upper bound.
Jensen's inequality and the fact $0 < \alpha_N < 1$ yield $\sqrt{(1+\alpha_N)/2} - \alpha_N \geq \frac{1 + \sqrt{\alpha_N}}{2} - \sqrt{\alpha_N}
=\frac{1 - \sqrt{\alpha_N}}{2}$.
Here for $a > 0$, it is easy to see $(1+a)^{N/2} \geq 1 + aN/2$ and thus we have $1 - (1+a)^{-N/2} \geq 1 - (1 + aN/2)^{-1} = \frac{a N/2}{1 + a N/2}$. Substituting $a=\lambda\eta/\mu_0$, $\frac{1 - \sqrt{\alpha_N}}{2}\geq  \frac{N\lambda\eta/(2\mu_0)}{1 + N\lambda\eta/(2\mu_0)}$. Then, by using $\eta = 1/N$, we obtain that $\frac{2 \bar{b}}{\sqrt{(1+\alpha_N)/2} - \alpha_N} \leq \frac{4b \mu_0 (1 + \lambda/(2\mu_0))}{\lambda} =2 \bar{b} (1 + 2\frac{\mu_0}{\lambda}).$


Then, Theorem 2.5 of \cite{mattingly2002} asserts that there exits a invariant measure $\mu^\eta$ for the Markov chain $(X_{lN})_{l}$ and the chain satisfies the geometric ergodicity: for $\phi: \cH \to \mathbb{R}$ such that $|\phi(\cdot)| \leq V(\cdot)$, 
\begin{align}
\bbE[\phi(X_{l N})] - \bbE_{X\sim \mu^\eta}[\phi(X)] \leq 
\kappa [\bar{V} + 1] (1 - \delta)^{a l} + \sqrt{2} V(x_0) \gamma^{l} (\kappa [\bar{V} + 1])^{a l} \frac{1}{\sqrt{\delta}},
\label{eq:fXlNgeometric_firstbound}
\end{align}
where $\kappa = \bar{b} + 1$, $\bar{V} = 2 \sup_{x \in \cC} V(x) = \frac{4 \bar{b}}{\sqrt{(1+\alpha_N)/2} - \alpha_N}$, $\gamma = \sqrt{(\alpha_N + 1)/2}$ and $a \in (0,1)$ so that $\gamma (\kappa [\bar{V} + 1])^{a} \leq (1- \delta)^a$.
In particular, we may choose $a \in (0,1)$ as 
$$
a = \frac{\log(1/\gamma)}{\log(\kappa (\bar{V} + 1)/(1-\delta))}.
$$
Here, by noting that 
$$
\log(1/\gamma) 
= - \frac{1}{2} \log\left( \frac{1 + \alpha_N}{2}\right)
= - \frac{1}{2} \log\left(1  - \frac{1 - \alpha_N}{2}\right)
\geq 
\frac{1}{2}\left(\frac{1 - \alpha_N}{2}\right) \geq \frac{1}{4}\min\left(\frac{\lambda}{2 \mu_0}, \frac{1}{2} \right) = \Omega(\lambda/\mu_0),
$$
we may assume 
$$
a \geq \frac{\min\left(\frac{\lambda}{2 \mu_0}, \frac{1}{2} \right)}{4 \log(\kappa (\bar{V} + 1)/(1-\delta)) }.
$$
Then Eq.~\eqref{eq:fXlNgeometric_firstbound} is simplified to 
\begin{align}
\bbE[\phi(X_{l N})] - \bbE_{X\sim \mu^\eta}[\phi(X)] \leq 
\left(\kappa [\bar{V} + 1] + \frac{\sqrt{2} V(x_0)}{\sqrt{\delta}} \right)  (1 - \delta)^{a l}.
\end{align}
This shows the geometric ergodicity of the sequence $(X_{lN})_{l=1}^\infty$.
To extend this result to ``unsampled'' sequence $(X_n)_{n=1}^\infty$, 
we may apply the same argument to the sequence $(X_{lN + n})_{l=0}^\infty$ for each $n = 1,\dots, N-1$.
Applying Eq. \eqref{eq:fXlNgeometric_firstbound} where $x_0$ is replaced with $X_n$ and taking expectation with respect to $X_n$, we have 
\begin{align}
\bbE[\phi(X_{l N + n})] - \bbE_{X\sim \mu^\eta}[\phi(X)] 
& \leq 
\left(\kappa [\bar{V} + 1] + \frac{\sqrt{2} \bbE[V(X_n)]}{\sqrt{\delta}} \right)  (1 - \delta)^{a l} \notag \\
& \leq
\left(\kappa [\bar{V} + 1] + \frac{\sqrt{2} (\rho^n \|x_0\| + b + 1)}{\sqrt{\delta}} \right)  (1 - \delta)^{a l} ~~~(\because \text{Proposition \ref{prop:norm_control}})\notag \\
& \leq
\left(\kappa [\bar{V} + 1] + \frac{\sqrt{2} (V(x_0) + b)}{\sqrt{\delta}} \right)  (1 - \delta)^{a l}. 
\end{align}
Finally, we note that for $0 \leq n < N$, 
\begin{align*}
(1 - \delta)^{a l} & \leq (1 - \delta)^{a (l N + n - N)/N} \leq (1 - \delta)^{a [(l N + n)/N - 1]}
\leq \exp\left(-\delta a [(l N + n)/N - 1] \right) \\
& \leq \exp\left( - \Lambda_\eta^* [\eta (lN + n) - 1] \right),
\end{align*}
where we set 
$$
 \Lambda_\eta^* \eqdef a \delta \geq  \frac{\min\left(\frac{\lambda}{2 \mu_0}, \frac{1}{2} \right)}{4 \log(\kappa (\bar{V} + 1)/(1-\delta)) } \delta
= \Omega(\exp(-O(\beta))).
$$
This yields the assertion. $\square$

\section{Proof of \Cref{prop:gibbs_concentration}}
\label{sec:ProofOfThirdTerm}

\begin{lemma}[Gaussian correlation inequality]
\label{eq:GaussianCorrelationInfinite}
Let $\nu_\infty$ be the Gaussian measure in $\cH$ given by a random variable 
$\sum_{i=0}^\infty \xi_i \gamma_i f_i$ where $(\xi)_{i=0}^\infty$ is a sequence of i.i.d. standard normal variables and $(\gamma_i)_{i=0}^\infty$ is a sequence of real variables with $0 < \sum_{i=0}^\infty \gamma_i^2 < \infty$.
For two sets $\cC^1 = \{X = \sum_{i=0}^\infty \alpha_i f_i \in \cH \mid \sum_{i=0}^\infty \alpha_i^2 \mu_i^{(1)} \leq 1 \}$
and $\cC^2 = \{X = \sum_{i=0}^\infty \alpha_i f_i \in \cH  \mid | \sum_{i=0}^\infty \alpha_i
 \mu_i^{(2)} |\leq 1 \}$
where $(\mu_i^{(1)})_{i=1}^\infty$ is a fixed non-negative sequence and $(\mu_i^{(2)})_{i=1}^\infty$ is a fixed sequence of real numbers satisfying 
$\sum_{i=0}^\infty (\mu_i^{(2)})^2 < \infty$,
we have
$$
\nu_\infty(\cC^1 \cap \cC^2) \geq \nu_\infty(\cC^1) \nu_\infty(\cC^2).
$$
\end{lemma}
\begin{proof}
Let $\cC^1_n$ an $\cC^2_n$ be the cylinder set that ``truncates'' $\cC^1$ an $\cC^2$ up to index $n$:
$\cC^1_n = \{X = \sum_{i=0}^\infty \alpha_i f_i \in \cH \mid \sum_{i=0}^n \alpha_i^2 \mu_i^{(1)} \leq 1 \}$
and $\cC^2_n = \{X = \sum_{i=0}^\infty \alpha_i f_i \in \cH \mid 
| \sum_{i=0}^n \alpha_i \mu_i^{(2)} | \leq 1 \}$.
By the Gaussian correlation inequality \cite{royen2014simple,Latala2017}, it holds that
$$
\nu_\infty(\cC^1_n \cap \cC^2_n) \geq \nu_\infty(\cC^1_n) \nu_\infty(\cC^2_n).
$$
We note that $(\cC^1_n)_{n}$ is a monotonically decreasing sequence, i.e., $\cC^1_n \subseteq \cC^1_m$ for $m < n$, and we see that $\cap_{n=1}^\infty \cC^1_n = \cC^1$.
By the continuity of probability measure, this yields that $\lim_{n \to \infty} \nu_\infty(\cC^1_n\backslash \cC^1) = 0$ and $\lim_{n \to \infty} \nu_\infty(\cC^1_n) = \nu(\cC^1)$.
On the other hand, for any $\epsilon > 0$, there exists $N$ such that $\sum_{i=N}^\infty (\gamma_i \mu_i^{(2)})^2 \leq \epsilon$ by the assumption ($\sum_{i=0}^\infty \gamma_i^2 < \infty$ and $\sum_{i=0}^\infty (\mu_i^{(2)})^2 < \infty$). Hence, it holds that 
$\bbE[(\sum_{i=N}^\infty \gamma_i \xi_i \mu_i^{(2)})^2] = \sum_{i=N}^\infty (\gamma_i \mu_i^{(2)})^2 \leq \epsilon$,
which indicates that, by Markov's inequality, 
$$
\textstyle
\nu_\infty\left( \right\{  \sum_{i=0}^\infty \alpha_i f_i \mid 
|\sum_{i=N}^\infty \alpha_i \mu_i^{(2)} | > \delta  \left\} \right) \leq \epsilon/\delta^2
$$
for any $\delta > 0$. 
If we set $\cC^2_{(\epsilon)} = \{\sum_{i=0}^\infty \alpha_i f_i \in \cH  \mid | \sum_{i=0}^\infty \alpha_i
 \mu_i^{(2)} |\leq 1 + \epsilon \}$, 
then this and the continuity of Gaussian measures (note that $\sum_{i=0}^\infty \xi_i \gamma_i \mu_i^{(2)}$ is a one dimensional Gaussian measure and has density with respect to the Lebesgue measure) yield that,  
for any $\epsilon > 0$, there exists $N$ such that for all $n \geq N$, it holds that 
\begin{align*}
& \nu_\infty(\cC^2_{(-\epsilon)}) - \epsilon \leq \nu_\infty(\cC^2_n) 
\leq \nu_\infty(\cC^2_{(\epsilon)}) + \epsilon, \\
& \nu_\infty(\cC^1 \cap\cC^2_{(-\epsilon)}) - \epsilon
\leq \nu_\infty(\cC^1 \cap \cC^2_n) \leq 
 \nu_\infty(\cC^1 \cap \cC^2_{(\epsilon)}) + \epsilon.
\end{align*}
Since $\lim_{n \to \infty} \nu_\infty(\cC^1_n \backslash \cC^1) = 0$, the second inequality also gives 
$$
\nu_\infty(\cC^1_n \cap\cC^2_{(-\epsilon)}) - 2 \epsilon
\leq \nu_\infty(\cC^1_n \cap \cC^2_n) \leq 
 \nu_\infty(\cC^1_n \cap \cC^2_{(\epsilon)}) + 2 \epsilon.
$$
for any $n \geq N'$ with sufficiently large $N'$.
Therefore, since 
$\lim_{\epsilon \to 0} \nu_\infty((\cC^2_{(\epsilon)} \backslash \cC^2) \cup (\cC^2 \backslash \cC^2_{(\epsilon)})) = 0$ by the continuity of Gaussian measures and $\lim_{n \to \infty} \nu_\infty(\cC^1_n\backslash \cC^1)$,
by taking the limit of $\epsilon$ and $n$ of this inequality, 
we have 
$$
\nu_\infty(\cC^1 \cap \cC^2)
= \lim_{n \to \infty} \nu_\infty(\cC^1_n \cap \cC^2_n).
$$
Hence, applying the Gaussian correlation inequality to the right hand side yields
\begin{align*}
& \nu_\infty(\cC^1 \cap \cC^2)
= \lim_{n \to \infty} \nu_\infty(\cC^1_n \cap \cC^2_n) \\
& \geq 
\lim_{n \to \infty} \nu_\infty(\cC^2_n ) \nu_\infty(\cC^2_n)
= \nu_\infty(\cC^2 ) \nu_\infty(\cC^2).
\end{align*}
\end{proof}

\begin{proof}[Proof of \Cref{prop:gibbs_concentration}]
The proof relies on comparing the stationary distribution $\pi$ of \Cref{eq:langevin_sde} to the Gaussian stationary distribution $\nu_\infty^{(\beta)}$ of \Cref{eq:def_z} (case $F = 0$). We then conclude by using the small ball probability theorem \citep{JFA:Kuelbs&Li:1993,SPTM:Li&Shao:2001} and \Cref{eq:GaussianCorrelationInfinite} on $\nu_\infty$.
First note that
\begin{align*}
& \int L(x) \dd \pi(x) - L(\tilde{x}) \\
& = - \tfrac{1}{\beta}\! \int\! \log\{ e^{- \beta(L(x) - L(\tilde{x}))}\} \dd \pi(x) \\
& = - \tfrac{1}{\beta}\!\! \int\! \log \{\tfrac{1}{\Lambda}e^{- \beta(L(x) - L(\tilde{x}))}\} \dd \pi(x) -  \tfrac{1}{\beta} \log \Lambda \\
& = - \tfrac{1}{\beta} \mathrm{KL}(\pi || \nu_\infty^{(\beta)}) - \tfrac{1}{\beta} \log(\Lambda), \numberthis
\end{align*}
%
where $\nu_\infty^{(\beta)}$ is the invariant distribution of \Cref{eq:def_z}, i.e., the centered Gaussian on $\cH$ with covariance operator $(-\beta A)^{-1}$, $\Lambda \eqdef \int \exp[- \beta(L(x) - L(\tilde{x}))]  \dd \nu_\infty^{(\beta)}(x)$,
and $\mathrm{KL}(\mu || \nu) \eqdef \int  \log(\dd \mu/\dd \nu)\dd \mu$ for probability measures $\mu$ and $\nu$ that are mutually absolutely continuous.
Since the KL-divergence $\mathrm{KL}(\pi || \nu_\infty^{(\beta)})$ is non-negative, the right hand side is upper bounded by
$- \frac{1}{\beta} \log(\Lambda)$.
By definition of $\tilde{x}$ (\Cref{eq:x_tilde}), it holds that
\begin{equation*}
    \nabla L (\tilde{x}) = - \frac{\lambda}{2}\nabla \norm{\tilde{x}}^2_{\cH_K} = - \lambda \sum_{k\geq 0} \tfrac{\scal{\tilde{x}}{f_k}}{\mu_k}f_k
     .
\end{equation*}
Hence, using the $M$-smoothness of $L$, we obtain
\begin{align*}
    & L(x) - L(\tilde{x}) \\ 
    &\leq \tfrac{1}{2}M\|x - \tilde{x}\|^2 + \lambda \langle \tilde{x}, x - \tilde{x} \rangle_{\cH_K} \\ 
    &\leq \tfrac{1}{2}M\|x - \tilde{x}\|^2 + \lambda \|\tilde{x}\|_{\cH_K} \left\langle \frac{\tilde{x}}{\|\tilde{x}\|_{\cH_K}}, x - \tilde{x} \right\rangle_{\cH_K}.
\end{align*}
Therefore, $\log(\Lambda)$ can be lower bounded by
\begin{align*}
    &\log(\Lambda) 
     \geq \log \int \exp\left\{- \beta
    \left[
    \tfrac{1}{2}M\|x - \tilde{x}\|^2  \right.\right.\\
     &\qquad + \left.\left. \lambda \|\tilde{x}\|_{\cH_K} \left\langle \frac{\tilde{x}}{\|\tilde{x}\|_{\cH_K}}, x - \tilde{x} \right\rangle_{\cH_K} 
    \right]
    \right\} \dd \nu_\infty^{(\beta)} (x) \\
    & \quad \geq 
    - \beta \Big[\tfrac{1}{2}M \varepsilon^2 
    +\lambda \|\tilde{x}\|_{\cHK} U
    \Big] \\ 
    & \qquad + \log[\nu_\infty^{(\beta)}(\{x \in \tilde{x} + \cC_{\varepsilon,U} \})],
\end{align*}
where $\cC_{\varepsilon,U} \eqdef \{x \in \cH \mid \| x\| \leq \varepsilon,
{\scriptstyle | \langle \frac{\tilde{x}}{\|\tilde{x}\|_{\cHK}},x\rangle_{\cHK}| } \leq U \}$ for arbitrary $\varepsilon > 0$ and $U > 0$ (if $\|\tilde{x}\|_{\cHK}=0$, then we treat $\frac{\tilde{x}}{\|\tilde{x}\|_{\cHK}} = 0$). Then, by Borell's inequality (\citet{borell1975brunn}, \citet[Lemma 5.2]{vaart2008}), we have
\begin{align*}
     \log[\nu_\infty^{(\beta)}(\{x \in \tilde{x} + \cC_{\varepsilon,U} \})]
    & \geq \log(\nu_\infty^{(\beta)}(\cC_{\varepsilon,U})) - \frac{\beta \lambda}{2} \|\tilde{x}\|_{\cHK}^2.
\end{align*}
Finally, we lower bound $\log(\nu_\infty^{(\beta)}(\cC_{\varepsilon,U}))$.
Let $\cC_{\varepsilon}^{(1)} \eqdef \{x \in \cH \mid \| x\| \leq \varepsilon \}$
and $\cC_{U}^{(2)} \eqdef \{x \in \cH \mid {\scriptstyle | \langle \frac{\tilde{x}}{\|\tilde{x}\|_{\cHK}},x\rangle_{\cHK}| } \leq U \}$
(that is, $\cC_{\epsilon,U} = \cC_{\varepsilon}^{(1)} \cap \cC_{U}^{(2)}$), then
by Lemma \ref{eq:GaussianCorrelationInfinite}, it holds that
\begin{align*}
\log(\nu_\infty^{(\beta)}(\cC_{\varepsilon,U}))
\geq \log(\nu_\infty^{(\beta)}(\cC_{\varepsilon}^{(1)})) + \log(\nu_\infty^{(\beta)}(\cC_{U}^{(2)})).
\end{align*}
By the small ball probability theorem \citep{JFA:Kuelbs&Li:1993,SPTM:Li&Shao:2001},
we can lower bound the first term of the left hand side as
\begin{equation*}
    - \log(\nu_\infty^{(\beta)}(\cC_{\varepsilon}^{(1)})) \lesssim
    (\sqrt{\beta \lambda}\varepsilon)^{-2}.
\end{equation*}

To evaluate $\nu_\infty^{(\beta)}(\cC_{U}^{(2)})$, we note that 
$$
\bbE_{x \sim \nu_\infty^{(\beta)}}
\left[\left\langle \frac{\tilde{x}}{\|\tilde{x}\|_{\cHK}},x\right\rangle_{\cHK}^2\right]
\leq \beta^{-1}.
$$
Therefore, by the Markov's inequality,
$$
\nu_\infty^{(\beta)}(\cC_{U}^{(2)}) \geq 1 - \frac{1}{\beta U^2}.
$$
By setting $U = \sqrt{2 / \beta}$, we also have
\begin{equation*}
    - \log(\nu_\infty^{(\beta)}(\cC_{U}^{(2)})) \leq \log(1/2).
\end{equation*}
Combining these inequalities, we finally arrive at
\begin{align*}
    & \int\!\! L \dd \pi - L(\tilde{x})\!   \leq  - \frac{1}{\beta} \log(\Lambda) \\
     &\quad \leq \tfrac{1}{2}M \varepsilon^2 + 
      \lambda \|\tilde{x}\|_{\cHK} \sqrt{\frac{2}{\beta}} + \frac{\lambda}{2} \|\tilde{x}\|_{\cHK}^2 \\
      & \quad \quad + \beta^{-1} [C(\sqrt{\beta \lambda}\varepsilon)^{-2} + \log(1/2)]  \\
     & \quad \lesssim 
     \frac{M}{2}\varepsilon^2 +
     \lambda\left( \|\tilde{x}\|_{\cHK} \beta^{-1/2} + \|\tilde{x}\|_{\cHK}^2\right) \\
     & \quad \quad + (\beta \sqrt{\lambda} \varepsilon)^{- 2} + \beta^{-1}
%
    \numberthis.
\end{align*}
Finally, differentiating the above w.r.t. $\varepsilon$, we get that the optimal bound is attained for $\varepsilon = \left(\frac{2}{M\beta^2\lambda}\right)^{1/4}$, and is then equal to
\begin{equation*}
     \frac{1}{\beta}\left(\sqrt{\frac{2M}{\lambda}} + 1\right) +
     \lambda\left( \|\tilde{x}\|_{\cHK} \beta^{-1/2} + \|\tilde{x}\|_{\cHK}^2\right).
\end{equation*}
%
\end{proof}


\section{Proof of time and space approximation error (\Cref{prop:scheme_cvg_general} and \Cref{prop:TimeAndSpaceApproxError})}
\label{sec:ProofOfTimeAndSpaceApprox}

In this section, we prove \Cref{prop:scheme_cvg_general} and \Cref{prop:TimeAndSpaceApproxError}. 
As we have noted, \Cref{prop:scheme_cvg_general} is obtained as a corollary of \Cref{prop:TimeAndSpaceApproxError} by taking the limit of $N \to \infty$. 
More strongly, we can show the following lemma. 
Let 
\begin{align}\label{eq:defcbeta}
\cbeta = 
\begin{cases} 1~~~& (\text{strict dissipativity condition: \Cref{assum:strict_diss})}), \\ \sqrt{\beta}~~~& (\text{bounded gradient condition: \Cref{assum:bounded_grad}}). 
\end{cases}
\end{align}

\begin{lemma}
    \label{lem:TimeAndSpaceErrorSum}
    Suppose $\|x_0\| \leq 1$. Under the same assumptions and notations as in \Cref{prop:scheme_cvg_general,prop:TimeAndSpaceApproxError}, it holds that:
    \begin{align}
        \begin{split}
            &\frac{1}{n}\sum_{k=0}^{n-1}\bbE\left[\phi\left(X^N_k\right)-\phi\left(X^\pi\right)\right]\\
            &\leq \frac{C_1}{\Lambda_0^*}\cbeta \left(1 + (n\eta')^{-1+\kappa}+(n\eta')^{-1}\right)\left(\eta^{1/2-\kappa}+\mu_{N+1}^{1/2-\kappa}+(n\eta')^{-1}\right).
        \end{split}
    \end{align}
\end{lemma}

\begin{proof}[Proof of \Cref{prop:scheme_cvg_general} and \Cref{prop:TimeAndSpaceApproxError}]
Once we obtain this lemma (\Cref{lem:TimeAndSpaceErrorSum}), then it is easy to show both propositions
by taking into account that 
\begin{align}
\lim_{n \to \infty}\frac{1}{n}\sum_{k=0}^{n-1} \bbE\left[\phi\left(X^N_k\right)\right] = \bbE[\phi(X^{\mu_{(N,\eta)}})]
\label{eq:GeomErgEachN}
\end{align}
where $\mu_{(N,\eta)}$ is the invariant measure of $(X_k^N)_k$ whose existence and uniqueness can be shown in the same manner as \Cref{prop:geom_ergo} (see also \citet{Brehier16} for this argument).
This gives  \Cref{prop:TimeAndSpaceApproxError} under the condition $\|x_0\| \leq 1$.
Since the invariant measure $\mu_{(N,\eta)}$ is independent of the initial solution $x_0$,
we may drop the condition $\|x_0\|\leq 1$. Then, we obtain \Cref{prop:TimeAndSpaceApproxError}.

We can see that the proof of \Cref{prop:geom_ergo} is valid to show the convergence of \Cref{eq:GeomErgEachN} uniformly over all $N$
and its convergence is uniform over all $N$.  
Moreover, it has been already shown (see \citet{Brehier14} for example)  that 
\begin{align*}
\lim_{N \to \infty} \bbE[\phi(X^N_k)] = \bbE[\phi(X_k)]~~~(\forall k \in \bbN).
\end{align*}
Consequently, we can exchange the order of limit, and by applying the geometric ergodicity $\lim_{k \to \infty} \bbE[\phi(X_k)] = \bbE[\phi(X^{\mu_\eta})]$  (\Cref{prop:geom_ergo}) again, we also have 
\begin{align*}
\lim_{N \to \infty} \bbE[\phi(X^{\mu_{(N,\eta)}})] = &
\lim_{N \to \infty} \lim_{n \to \infty}\frac{1}{n}\sum_{k=0}^{n-1} \bbE\left[\phi\left(X^N_k\right)\right] \\
= & \lim_{n \to \infty} \lim_{N \to \infty} \frac{1}{n}\sum_{k=0}^{n-1} \bbE\left[\phi\left(X^N_k\right)\right]
~~~(\because \text{uniformity of convergence})
\\
= &
\lim_{n \to \infty}  \frac{1}{n}\sum_{k=0}^{n-1} \bbE\left[\phi\left(X_k\right)\right]
=
\bbE[\phi(X^{\mu_\eta})].
\end{align*}
Therefore, \Cref{lem:TimeAndSpaceErrorSum} gives the proof of  \Cref{prop:scheme_cvg_general} by taking the limit of $n \to \infty$ and $N \to \infty$. 
Here again, we would like to note that the assumption $\|x\| \leq 1$ can be dropped in the limit because the invariant measure $\mu_\eta$ is independent of the initial solution.
\end{proof}

In the following, we prove \Cref{lem:TimeAndSpaceErrorSum}.
Our proof follows the line of \citet{Brehier16}. For lighter notation, our constants may differ from line to line.

\subsection{Preliminaries}
In this subsection, we prepare some notations and state lemmas necessary to prove the statement. 
Here, we introduce the continuous time dynamics with the Galerkin approximation as
\begin{align}
    \label{galerkin}
    \begin{cases}
         & X^N(0) = P_Nx_0\in\mathcal{H}_N,                                                                        \\
         & \mathrm{d}X^N(t) = (AX^N(t) - \nabla {L_N}(X^N(t)))\mathrm{d}t + \sqrt{\frac{2}{\beta}}P_N\mathrm{d}W(t).
    \end{cases}
\end{align}
Here, we denote by $X(t,x)$ to represent $X(t)$ with $X_0 = x$ and similarly we write $Y(t,x)$ for a continuous time process $\{Y(t)\}_{t}$ to indicate $Y(t)$ with $Y(0) = x$.
Notice that our constants should not depend on $\beta$ while \citet{Brehier16} sets $\beta=1$. Our technical contribution is to extend the work of \citet{Brehier16} to general case. To this end, we apply a change of variables with $t'\triangleq 2 t/\beta$. Accordingly, \Cref{eq:langevin_sde} transforms into
\begin{align}
    \begin{cases}
         X(0) & = x_0\in\mathcal{H},                                                   \\
         \mathrm{d}X(t')&  = -\nabla{\mathcal{L'}}(X(t'))\mathrm{d}t' + \mathrm{d}W(t') \\
         & = (A'X(t') - \nabla{L'}(X(t')))\mathrm{d}t' + \mathrm{d}W(t'),
    \end{cases}
\end{align}
where $A'\triangleq (\beta/2) A,\: L'\triangleq(\beta/2) L,\: \mathcal{L'}\triangleq(\beta/2)\mathcal{L}$. 
Accordingly, let the ``time re-scaled version'' of the process  $X(t)$ be 
$$
\Xhat(t') \eqdef X\left( {\textstyle \frac{\beta}{2}} t' \right),~~~\Xhat^N(t') \eqdef X^N\left( {\textstyle \frac{\beta}{2}} t' \right)~~~~~~(t' \geq 0).
$$
Similarly, the change of variables with $\eta'=\eta/(\beta/2)$ translates the time discretized Galerkin approximation scheme \eqref{eq:GalerkinTimeDiscrete} into
\begin{align}
    \begin{cases}
         & X^N_0  =P_Nx_0\in\mathcal{H}_N,                                                                \\
         & X^N_{n+1}  = X^N_n - \eta'(A'X^N_{n+1} - \nabla{L'_N}(X^N_n)) + \sqrt{\eta'}\epsilon_n         \\
         & ~~~~\Leftrightarrow X^N_{n+1}  = \Setap(X^N_n - \eta'\nabla{L'_N}(X_n^N)+\sqrt{\eta'}\epsilon_n),
    \end{cases}
\end{align}
where $L'_N\triangleq(\beta/2) L_N,\: \Setap=S_\eta$. Here we used the abuse of notation to let $A',\Setap$ indicate the map from $\cH_N$ to $\cH_N$
which is naturally defined by $A',\Setap:\cH \to \cH$ through the canonical imbedding $\iota: x \in \cH_N \hookrightarrow \cH$:
$A x \eqdef (A \circ \iota) x$ for $x \in \cH_N$ (the same argument is also applied to $\Setap$). 
Note that we may set $\Setap = S_\eta$ because $\eta' A' = \eta A$.
We write the (rescaled) continuous time corresponding to the $k$-th step as 
$$t_k\triangleq k\eta'.$$ 

Our approach is to follow the proofs of \citet{Brehier16,Kopec} and uncover the dependency of $\beta$ step by step. For completeness, we restate the results of \citet{Brehier16} in our notations.
\begin{proposition}
    \begin{enumerate}
        \item We have for any $N\in\mathbb{N},\: \gamma\in[-1/2,1/2]$, and $x\in\mathcal{H}$,
              \begin{align}
                  \norm{(-A')^\gamma P_N x} \leq \norm{(-A')^\gamma x}.
              \end{align}
        \item For $P_N$, we have the following error estimate:
              \begin{align}
                  \norm{(-A')^{s/2}(I-P_N)(-A')^{-r/2}}_{\mathcal{B}(\mathcal{H})}= \left(\frac{\mu'_{N+1}}{\beta/2}\right)^{(r-s)/2}\quad \forall 0\leq s\leq 1,\: s\leq r\leq 2,
              \end{align}
              where $$\mu'_i \eqdef \mu_i / \lambda.$$
    \end{enumerate}
\end{proposition}
    The corresponding result in \citet{Brehier16} is on finite element approximations; see \citet{Andersson2015} for more details. However, it can be naturally extended to spectral Galerkin projection as pointed out in \citet{Brehier16}. In fact, these two approximations are essentially the same; see \citet{Kruse}.
As a consequence, we get the following result.
\begin{proposition}
    \label{prop2}
    For any $\kappa > 0$, the linear operator on $\mathcal{H}$, $P_N (-A')^{-1/2-\kappa}P_N$ is continuous, self-adjoint and positive semi-definite. Moreover, there exists $C_\kappa>0$ such that for any $\beta>0$
    \begin{align}
        \sup_{N\in\mathbb{N}} \mathrm{Tr}\left(P_N (-A')^{-1/2-\kappa}P_N\right)< \frac{C_\kappa}{\beta^{1/2+\kappa}}.
    \end{align}
\end{proposition}
This is an extension of Proposition 3.4 in \citet{Kopec}, where $\beta$ is fixed to 1. The following fundamental inequality is important for the proof of \Cref{prop2}.
\begin{proposition}
    For $M, N \in \mathcal{B}(\mathcal{H})$ such that $M$ is symmetric and positive semi-definite,
    \begin{align}
        |\mathrm{Tr}(MN)|\leq \norm{M}_{\mathcal{B}(\mathcal{H})}|\mathrm{Tr}(N)|.
    \end{align}
\end{proposition}
Our next step is to extend Lemma 3.7 of \citet{Kopec} to our case.
\begin{lemma}
    \label{norm-bound}
    For any $0\leq\kappa\leq 1,\: N\in\mathbb{N},\: \beta\geq\eta_0$, and $j\geq 1$,
    \begin{align}
        \norm{(-A')^{1-\kappa}\Setap^j P_N}_{\mathcal{B}(\mathcal{H})}\leq\frac{(\beta/2)^{1-\kappa}}{(j\eta)^{1-\kappa}}\frac{1}{(1+\eta/\mu'_0)^{j\kappa}}
=
\frac{1}{t_j^{1-\kappa}}\frac{1}{(1+\eta/\mu'_0)^{j\kappa}}.
    \end{align}
    Moreover,
    \begin{align}
        \forall\: \gamma\geq 1\quad \exists\: C_\gamma > 0\quad \forall\: j\geq\gamma\quad
        \norm{(-A')^{\gamma}\Setap^j P_N}_{\mathcal{B}(\mathcal{H})}\leq\frac{C_\gamma}{(j\eta)^\gamma}(\beta/2)^\gamma=\frac{C_\gamma}{t_j^\gamma},
    \end{align}
    and for any $0\leq \gamma\leq 1$,
    \begin{align}
        \norm{(-A')^{-\gamma}(\Setap-I)P_N}_{\mathcal{B}(\mathcal{H})}\leq 2\frac{\eta^\gamma}{(\beta/2)^\gamma} = 2 {\eta'}^\gamma.
    \end{align}
\end{lemma}
The proof is almost the same as in Lemma 3.7 of \citet{Kopec}, and thus we omit the proof. The relationship that $\eta'=\eta/(\beta/2),\: A'=(\beta/2) A$ specifies the dependence on $\beta$.

As in \citet{Brehier16}, we have the following expression of $X_k^N$:
\begin{align}
     & X_k^N = \Setap^kP_Nx - \eta' \sum_{l=0}^{k-1}\Setap^{k-l}P_N\nabla{L}(X_l^N)+\sqrt{\eta'}\sum_{l=0}^{k-l}\Setap^{k-l}P_N\epsilon_{l+1}, \\
     & \sqrt{\eta'}\sum_{l=0}^{k-l}\Setap^{k-l}P_N\epsilon_{l+1}=\int_0^{t_k}\Setap^{k-l_s}P_N\mathrm{d}W(s),
\end{align}
where $l_s\triangleq \floor{\frac{s}{\eta'}}$ with the notation $\floor{\cdot}$ is the floor function. The advantage of this expression is that we can handle each term by simple estimates.

We introduce the following interpolation processes: for $0\leq k\leq m-1$ and $t_k\leq t\leq t_{k+1}$, it holds that:
\begin{align}
    \label{tilde-X}
    \tilde{X}^N(t) = X^N_k + \int_{t_k}^t\Setap\left[A'X^N_k - P_N\nabla L'(X^N_k) \right] \mathrm{d}s + \int_{t_k}^t \Setap P_N\mathrm{d}W(s).
\end{align}
The process $\{\tilde{X}^N(t)\}_{t \geq 0}$ is a natural interpolation of the discrete scheme $\{X_k^N\}_{k\in\mathbb{N}}$: $\{\tilde{X}^N(t_k)\}_{k\in\mathbb{N}}$ and $\{X_k^N\}_{k\in\mathbb{N}}$ have the same joint distribution.

\subsection{Bounds on Moments}
In this subsection, we give a few bounds on moments of $\{X(t)\}_{t\geq 0},\: \{X^N(t)\}_{t\geq 0},\: \{X^N_k\}_{k\in\mathbb{N}}$.
Note that the constants are uniform with respect to $N\in\mathbb{N},\: 0<\eta\leq\eta_0$ and $\beta\geq\eta_0$.
\begin{lemma}
    \label{moments1}
    For any $p\geq 1$, there exists a constant $C_p>0$ such that for every $N\in\mathbb{N},\: t\geq 0,\: \beta \geq \eta_0$ and $x\in\mathcal{H}$,
    \begin{align}
        \bbE\left[\norm{X(t,x)}^p\right],\:,\bbE\left[\norm{\Xhat(t,x)}^p\right],\: \bbE\left[\norm{X^N(t, x)}^p\right],\:\bbE\left[\norm{\Xhat^N(t, x)}^p\right] \leq C_p(1+\norm{x_0}^p).
    \end{align}
\end{lemma}
\begin{lemma}
    \label{moments2}
    For any $p\geq 1, \eta_0>0$, there exists a constant $C_p$ such that for every $N\in\mathbb{N},\: 0<\eta\leq \eta_0,\: \beta\geq\eta_0,\: k\in\mathbb{N},\: t\geq 0$ and $x\in\mathcal{H}$,
    \begin{align}
        \bbE\left[\norm{X^N_k}^p\right],\: \bbE\left[\norm{\tilde{X^N}(t)}^p\right]\leq C_p(1+\norm{x_0}^p).
    \end{align}
\end{lemma}
Intuitively, these lemmas hold thanks to dissipativity, a kind of boundedness of a global optimum.
\begin{proof}[Proof of \Cref{moments1} and \Cref{moments2}]
The proof is very similar to that of \Cref{prop:norm_control}.
    We only prove the statement for the bounded gradient condition. For the strict dissipativity condition, see Proposition 3.2 of \citet{brehier2016high}. 
    We prove the statement following the same line as Lemma 4.1 and 4.2 of \citet{Brehier14}. There is no essentially new ingredient, but we need to take care of the effect of $\beta$. We define $Z(t)=\Xhat(t)-W^{A'}(t)$ where $W^{A'}(t)=\int^t_0e^{(t-s)A'}\mathrm{d}W(s)$.
    It holds that $W^{A'}(t/(\beta/2))=W^A(t)/\sqrt{\beta/2}$. (2.6) in \citet{Kopec} implies:
    \begin{align*}
        \bbE\sup_{t' \geq 0}\norm{W^{A'}(t')}^p=\bbE\sup_{t\geq 0}\norm{W^{A'}\left(\frac{t}{\beta/2}\right)}^p=\bbE\sup_{t\geq 0}\norm{\frac{W^A(t)}{\sqrt{\beta/2}}}^p< \frac{C_p}{(\beta/2)^{p/2}}< C'_p,
    \end{align*}
    where $C_p, C'_p>0$ are constants independent from $\beta$.
    
    Then, we study $\norm{Z(t)}$. We have $Z(0)=\Xhat(0)=x_0$,
    \begin{align*}
        \dv{Z(t)}{t}=\frac{\beta}{2}(AZ(t)-\nabla L(\Xhat(t))),
    \end{align*}
    and by \Cref{prop:dissipative},
    \begin{align*}
        \frac{1}{2}\dv{\norm{Z(t)}^2}{t} & = \frac{\beta}{2}\langle AZ(t)-\nabla L(\Xhat(t)),Z(t)\rangle                                                        \\
                                         & = \frac{\beta}{2}\langle AZ(t)-\nabla L(Z(t)),Z(t)\rangle + \frac{\beta}{2}\langle \nabla L(Z(t))-\nabla L(\Xhat(t)),Z(t)\rangle \\
                                         & \leq \frac{\beta}{2}(-m\norm{Z(t)}^2 + c + \norm{\nabla L}_\infty\norm{Z(t)})                                    \\
                                         & \leq \frac{\beta}{2}(-m'\norm{Z(t)}^2 + C'),
    \end{align*}
	where $m'$ and $C'$ are positive constants depending only on $m,c,B$.
    Thus, we have for any $t\geq 0$
    \begin{align*}
       &  |\norm{Z(t)}^2 - C'/m'| \leq \exp(- \beta m' t) |\norm{x_0}^2 - C'/m'| \\
      \Longrightarrow~~& \norm{Z(t)}^2 \leq \exp(- \beta m' t) | \norm{x_0}^2 - C'/m'|  + C'/m'  \leq C (\norm{x_0}^2 + 1), 
    \end{align*}
    for a constant $C > 0$, which concludes the proof of \Cref{moments1}, since the estimates do not depend on the dimension parameter $N$.
    
    Similarly, we introduce $Z_k=X_k-w_k$, where $\{w_k\}_k$ is the numerical approximation of $W^{A'}$ defined by
    \begin{align*}
        w_{k+1}=\Setap w_k+\sqrt{\eta'}\Setap\xi_{k+1}.
    \end{align*}
    The same argument yields
    \begin{align}
        \bbE\norm{w_k}^2\leq \frac{C}{\beta}\leq C'.
    \end{align}
    Now we have $Z_0=X_0=x_0$,
    \begin{align*}
        Z_{k+1}=\Setap Z_k-\eta'\Setap\nabla L'(X_k),
    \end{align*}
    since $\norm{\Setap}_{\mathcal{B}(\mathcal{H})}\leq \frac{1}{1+\eta/\mu'_0}$, we obtain the almost sure estimates
    \begin{align*}
        \norm{Z_{k+1}}\leq \frac{1}{1+\eta/\mu'_0}\norm{Z_k}+C\eta',
    \end{align*}
    and therefore for $\beta \geq 1$
    \begin{align*}
        \norm{Z_k}\leq C(1+\norm{x_0}),
    \end{align*}
    which concludes the proof of \Cref{moments2}.
\end{proof}

\subsection{The Rate of Convergence to the Invariant Measure}
Our focus in this subsection is just to state the convergence result to an invariant measure. For the existence and uniqueness of the invariant measure of the continuous time dynamics, see \citet{DebusscheBSDE2011}, \citet{Goldys_Maslowski06} and \citet{Brehier16}.
We have the following result thanks to a coupling argument presented in \citet{DebusscheBSDE2011}.
\begin{proposition}
\label{prop:geometric_ergodicity_continuous_time}
Under \Cref{assum:eigenvalue_cvg,assum:smoothness,assum:C2_boundedness,assum:dissipative}.
    There exist the ``spectral gap'' $\lambda^*$ and a constant $C>0$ such that for any bounded test function $\phi:\mathcal{H}\rightarrow\mathbb{R}$, $t\geq 0,\: N\in\mathbb{N},\: \beta\geq\eta_0$ and $x_1,x_2\in\mathcal{H}_N$,
    \begin{align}
        \left|\bbE[\phi(X^N(t,x_1))]-\bbE[\phi(X^N(t,x_2))]\right|\leq C\norm{\phi}_\infty (1 + \norm{x_1}^2 + \norm{x_2}^2)e^{-\lambda^*t}.
    \end{align}
This also implies 
    \begin{align}
        \left|\bbE[\phi(\Xhat^N(t,x_1))]-\bbE[\phi(\Xhat^N(t,x_2))]\right|\leq C\norm{\phi}_\infty (1 + \norm{x_1}^2 + \norm{x_2}^2)e^{-\beta \lambda^*t}.
    \end{align}
\end{proposition}
    A proof of this result in case $\beta=1$ can be found in \citet{DebusscheBSDE2011}. We can easily see the statement holds if $\beta$ is arbitrary but we have to notice the convergence rate $\lambda^*$ can be varied depending on $\beta$. 
More concrete characterization of $\lambda^*$ will be given in \Cref{rem:LambdaStarIdentification}.
As pointed out in \citet{Raginsky_Rakhlin_Telgarsky2017}, this spectral gap is supposed to decrease exponentially with respect to $\beta$.

\begin{corollary}
    \label{cor:GeomErgCont}
    For any $N\in\mathbb{N}$, the process $X^N$ admits a unique invariant probability measure $\pi^N$ and satisfies the following bound:
    \begin{align}
        \begin{split}
            &\exists\: c,C,\lambda^*>0,\:\forall\: \phi:\mathcal{H}\rightarrow\mathbb{R},\: t\geq 0,\:x\in\mathcal{H}_N,\\
            &\left|\bbE[\phi(X^N(t,x))]-\int_{\mathcal{H}_N}\phi\mathrm{d}\pi^N\right|\leq C\norm{\phi}_\infty (1+\norm{x}^2)e^{-\lambda^* t}.
        \end{split}
    \end{align}
\end{corollary}
    These results naturally extend to an infinite dimensional scheme by similar arguments.

\begin{remark}[Characterization of $\lambda^*$]\label{rem:LambdaStarIdentification}
\citet[Theorem 1.1]{Brehier14} showed that 
$$
\lim_{\eta \to 0} |\bbE[\phi(X_{\floor{t/\eta}})] - \bbE[\phi(X(t))]| = 0.
$$
In addition to that we have shown in \Cref{prop:geom_ergo} that the discrete time dynamics satisfies the geometric ergodicity:
$$
| \bbE[\phi(X_{n})] - \bbE[\phi(X^{\mu_\eta})] \leq C (1 + \|x\|) \exp(-\Lambda^*_{\eta}(n \eta -1 )) (\leq C' (1 + \|x\|^2) \exp(-\Lambda^*_{\eta}(n \eta ))),
$$ 
where we used a fact that we may set $\Lambda^*_\eta \leq 1$ (if this is not satisfied, we may set $\Lambda^*_\eta \leftarrow \min\{\Lambda^*_\eta, 1\}$).
Moreover, \citet[Corollary 1.2]{Brehier14} gives that 
$$
\lim_{\eta \to 0} |\bbE[\phi(X^{\mu_\eta}) - \bbE[\phi(X^\pi)]| = 0.
$$
Combining these arguments, we see that
$$
|\bbE[\phi(X(t))] - \bbE[\phi(X^\pi)]| = \lim_{\eta \to 0} |\bbE[\phi(X_{\floor{t/\eta}})] - \bbE[\phi(X^{\mu_\eta})]|
\leq  \lim_{\eta \to 0} C (1 + \|x\|^2) \exp(-\Lambda^*_{\eta}(n \eta )).
$$
Finally, we note that \Cref{prop:geometric_ergodicity_continuous_time} and \Cref{cor:GeomErgCont} are used only for $\phi:\cH \to \bbR$ satisfying $\|\phi\|_\infty \geq c$ for a positive constant $c > 0$.
Hence, we may set $\lambda^* = \lim_{\eta \to 0} \Lambda^*_\eta = \Lambda^*_0$.

The same argument is also applied to $\{X_k^N\}_k$, $\{X^N(t)\}_t$ and $\{\Xhat^N(t)\}_t$ with the same value of $\Lambda^*_\eta$.
In the following, we use the notation $\lambda^*$ to indicate $\Lambda^*_0$.
\end{remark}

\begin{lemma}
    \label{convergence}
    For any bounded test function $\phi:\mathcal{H}\rightarrow\mathbb{R}$, we have
    \begin{align}
        \lim_{N\to\infty}\bar{\phi}_N:=\lim_{N \to \infty}\int_{\mathcal{H}_N}\phi\mathrm{d}\pi^N=\int_{\mathcal{H}}\phi\mathrm{d}\pi=:\bar{\phi}.
    \end{align}
\end{lemma}
\begin{proof}
    For any $t\geq 0$ and any fixed initial condition $x\in\mathcal{H}$, we have
    \begin{align*}
        \bar\phi_N-\bar\phi = & \bar\phi_N-\bbE\phi(X^N(t))            \\
                              & +\bbE\phi(X^N(t))-\bbE\phi(X(t)) \\
                              & +\bbE\phi(X(t))-\bar\phi.
    \end{align*}
    Since $\lim_{N \to \infty} \bbE\phi(X^N(t))-\bbE\phi(X(t)) = 0$ \cite{Brehier14}, we get that for any $t\geq 0$
    \begin{align*}
        \limsup_{N\rightarrow\infty}|\bar\phi_N-\bar\phi|\leq Ce^{-\lambda^* t},
    \end{align*}
    and then we may take $t\rightarrow\infty$. Notice that the constants are independent of the dimensionality $N$.
\end{proof}

\subsection{Proof of \Cref{lem:TimeAndSpaceErrorSum}}
In this subsection, we present a technical proof procedure of \Cref{lem:TimeAndSpaceErrorSum}.
As in \citet{Brehier16}, we will use the following decomposition:
\begin{align*}
    \begin{split}
        \frac{1}{n}\sum_{k=0}^{n-1}\bbE\phi(X_k^N)-\bar{\phi}=&\frac{1}{n}\sum_{k=0}^{n-1}\bbE\phi(P_{N'}X^N_k)-\bar{\phi}_{N'}\\
        &+\bar\phi_{N'}-\bar\phi+\frac{1}{n}\sum_{k=0}^{n-1}\left(\bbE\phi(X^N_k)-\bbE\phi(P_{N'}X^N_k)\right).
    \end{split}
\end{align*}
Our aim is to derive a $N'$-free bound of each term of this decomposition and to take $N'\rightarrow\infty$. It is obvious the last two terms converges to 0 as $N'\rightarrow+\infty$ thanks to \Cref{convergence} and $P_{N'}X_k^N = X_k^N$ if $N' \geq N$.

It remains to bound the first term. We decompose the term by the solution of the Poisson equation defined in the following. Let $N'\in\mathbb{N},\: \phi\in C^2_b(\mathcal{H})$. We define $\Psi^{N'}$ as the unique solution of the Poisson equation
\begin{align}
    \label{poisson}
    \mathcal{L}^{N'}\Psi^{N'}=\phi \circ P_{N'} - \bar\phi_{N'} \:\text{and}\: \int_{\mathcal{H}_{N'}}\Psi^{N'}\mathrm{d}\pi^{N'} = 0,
\end{align}
where $\mathcal{L}^{N'}$ is the infinitesimal generator of the SPDE\footnote{Note that from here we also use the notation $t$ to indicate $t'$ for notational simplicity.}:
\begin{align*}
    \begin{cases}
         & \Xhat^{N'}(0) = P_{N'}x_0\in\mathcal{H}_{N'},                                                    \\
         & \mathrm{d}\Xhat^{N'}(t) = (A'\Xhat^{N'}(t) - \nabla{L'_{N'}}(\Xhat^{N'}(t)))\mathrm{d}t + P_{N'}\mathrm{d}W(t),
    \end{cases}
\end{align*}
defined for $C^2$ functions $\psi:\mathcal{H}\rightarrow\mathbb{R}$ and $x\in\mathcal{H}$ by
\begin{align*}
    \mathcal{L}^{N'}\psi(x)=\langle A'P_{N'}x - P_{N'}\nabla{L'}(x),D\psi(x)\rangle+\frac{1}{2}\mathrm{Tr}(P_{N'}D^2\psi(x)).
\end{align*}
The following proposition is essential for our result. This is an extension of Proposition 6.1 in \citet{Brehier16} in that dependence on $\beta$ is specified.

\begin{proposition}
    \label{essential}
    Let $N'\in\mathbb{N}$ and $\phi\in C^2_b(\mathcal{H})$. The function $\Psi^{N'}$ defined for any $x\in\mathcal{H}_{N'}$ by
    \begin{align*}
        \Psi^{N'}(x) = \int_0^\infty \bbE\left[\phi(\Xhat^{N'}(t, x))-\bar\phi_{N'}\right]\mathrm{d}t,
    \end{align*}
    is of class $C^2_b$ and the unique solution of \Cref{poisson}. Moreover, we have the following estimates: for any $0\leq\epsilon,\gamma<1/2$ there exist $C,C_\epsilon,C_{\epsilon,\gamma}$, which are independent of $N'$ and $\beta$, such that for any $x\in\mathcal{H}_{N'}$
    \begin{align*}
        \norm{\Psi^{N'}(x)}                                                        & \leq \frac{C}{\lambda^* \beta}(1+\norm{x}^2)\norm{\phi}_\infty,                \\
        \norm{(-A')^{\epsilon}D\Psi^{N'}(x)} & \leq \frac{C_\epsilon}{\lambda^* \beta} \cbeta \beta^\epsilon (1+\norm{x}^2)\norm{\phi}_{0,1},         \\
        \norm{(-A')^\epsilon D^2\Psi^{N'}(x)(-A')^\gamma}_{\mathcal{B}(\mathcal{H}_M)} & \leq \frac{C_{\epsilon,\gamma} }{\lambda^* \beta} \cbeta^2 \beta^{\epsilon + \gamma}(1+\norm{x}^2)\norm{\phi}_{0,2},
    \end{align*}
    where $\norm{\phi}_{0,i}\triangleq \max \left\{\max_{0 < j \leq i }\|\phi\|_{(j)},\norm{\phi}_\infty\right\}$ for $\|\phi\|_{(1)} := \sup_{x \in \cH}\|\nabla \phi(x)\|$ and $\|\phi\|_{(2)} := \sup_{x \in \cH}\|D^2 \phi(x)\|_{\mathcal{B}(\cH)}$.
\end{proposition}
We give the proof of this proposition in \Cref{eq:ProofOfPsiNbound}.

To show the proof, we prepare more theoretical tools.
We define the function $\tilde\Psi^{N'}$ for $x\in\mathcal{H}$ by
\begin{align*}
    \tilde\Psi^{N'}(x) = \Psi^{N'}(P_{N'} x).
\end{align*}
It can be interpreted as an extension of $\Psi^{N'}$ to the entire domain $\cH$. 
Then we have for any $x\in\mathcal{H}$ and $h,k \in \cH$,
\begin{align*}
     & \langle D\tilde\Psi^{N'}(x),h \rangle = \langle D\Psi^{N'}(P_{N'}x), h P_{N'}\rangle,      \\
     & D^2\tilde\Psi^{N'}(x)\cdot (h,k)= D^2\Psi^{N'}(P_{N'}x) \cdot (P_{N'}h,P_{N'}k).
\end{align*}
\Cref{essential} can be also applied to $\tilde\Psi^{N'}$ by these equations.

Then we define the generator $\mathcal{L}^{\eta',k,N}$, discrete-time version of $\mathcal{L}^{N'}$, for all $k\in\mathbb{N}$ as
\begin{align*}
    \begin{split}
        &\text{for}\; x_0 \in\mathcal{H}_N,\: \phi\in\mathcal{B}(\mathcal{H}),\\
        &\mathcal{L}^{\eta',k,N}\phi(x) = \langle \Setap (A' X^N_k - P_N \nabla L'(X^N_k)),D\phi(x)\rangle + \frac{1}{2}\mathrm{Tr}(\Setap S^*_{\eta'} P_ND^2\phi(x)).        
    \end{split}
\end{align*}
Thanks to the It\^o formula and \Cref{essential}, we have
\begin{align*}
    \bbE\tilde\Psi^{N'}(X^N_{k+1})-\bbE\tilde\Psi^{N'}(X^N_k)=\int^{t_{k+1}}_{t_k}\bbE\mathcal{L}^{\eta',k,N}\tilde\Psi^{N'}(\tilde{X}^N(s))\mathrm{d}s.
\end{align*}
Similarly, we define the generator $\mathcal{L}^N$ of $X^N$ by
\begin{align*}
    \mathcal{L}^N\phi(x)=\langle A'x - P_N\nabla L'(x),D\phi(x)\rangle + \frac{1}{2}\mathrm{Tr}(P_N D^2\phi(x)).
\end{align*}
Putting all of the operators defined above, we have the following decomposition:
\begin{align*}
    \begin{split}
        \bbE\tilde\Psi^{N'}(X^N_{k+1})-\bbE\tilde\Psi^{N'}(X^N_k)=&\int^{t_{k+1}}_{t_k}\bbE\left(\mathcal{L}^{\eta',k,N}-\mathcal{L}^N\right)\tilde\Psi^{N'}(\tilde X^N(s))\mathrm{d}s\\
        &+\int^{t_{k+1}}_{t_k}\bbE\left(\mathcal{L}^{N}-\mathcal{L}^{N'}\right)\tilde\Psi^{N'}(\tilde X^N(s))\mathrm{d}s\\
        &+\int^{t_{k+1}}_{t_k}\bbE\mathcal{L}^{N'}\tilde\Psi^{N'}(\tilde X^N(s))\mathrm{d}s.
    \end{split}
\end{align*}
Furthermore, the following equality for $x\in\mathcal{H}$
\begin{align*}
    \mathcal{L}^{N'}\tilde\Psi^{N'}(x)=\mathcal{L}^{N'}\Psi^{N'}(x)+\langle -P_{N'}\nabla L'(x) + P_{N'} \nabla L'(P_{N'}x),D\Psi^{N'}(P_{N'}x)\rangle,
\end{align*}
and the definition of $\Psi^{N'}$ yields
\begin{align*}
    \begin{split}
        &\bbE\tilde\Psi^{N'}(X^N_{k+1})-\bbE\tilde\Psi^{N'}(X^N_k)\\
        =&\int_{t_k}^{t_{k+1}}\bbE\left(\mathcal{L}^{\eta',k,N}-\mathcal{L}^{N}\right)\tilde\Psi^{N'}(\tilde{X}^N(s))\mathrm{d}s\\
        &+\int_{t_k}^{t_{k+1}}\bbE\left(\mathcal{L}^{N}-\mathcal{L}^{N'}\right)\tilde\Psi^{N'}(\tilde{X}^N(s))\mathrm{d}s\\
        &+\eta'\left(\bbE\phi(P_{N'}X_k^N)-\bar\phi_{N'}\right)\\
        &+\int_{t_k}^{t_{k+1}}\bbE\left[\phi(P_{N'}\tilde{X}^N(s))-\phi(P_{N'}X^N_k)\right]\mathrm{d}s\\
        &+\int_{t_k}^{t_{k+1}}\bbE\langle P_{N'}\left(-\nabla L'(\tilde{X}^N(s))+\nabla L'(P_{N'}\tilde{X}^N(s))\right),D\Psi^{N'}(P_{N'}\tilde{X}^N(s))\rangle\mathrm{d}s,
    \end{split}
\end{align*}
and therefore
\begin{align*}
    \begin{split}
        &\frac{1}{n}\sum_{k=0}^{n-1}\bbE\phi(P_{N'}X^N_k)-\bar\phi_{N'}\\
        =&\frac{1}{n\eta'}\bbE\left[\Psi^{N'}(P_{N'}X^N_k)-\Psi^{N'}(P_{N'}X^N_1)\right]\\
        &+\frac{1}{n}\left(\phi(P_{N'}x)-\bar\phi_{N'}\right)\\
        &+\frac{1}{n\eta'}\sum_{k=0}^{n-1}\int_{t_k}^{t_{k+1}}\bbE\left(\mathcal{L}^{N'}-\mathcal{L}^N\right)\tilde\Psi^{N'}(\tilde{X}^N(s))\mathrm{d}s\\
        &+\frac{1}{n\eta'}\sum_{k=0}^{n-1}\int_{t_k}^{t_{k+1}}\bbE\left(\mathcal{L}^N-\mathcal{L}^{\eta',k,N}\right)\tilde\Psi^{N'}(\tilde{X}^N(s))\mathrm{d}s\\
        &-\frac{1}{n\eta'}\sum_{k=0}^{n-1}\int_{t_k}^{t_{k+1}}\bbE\left[\phi(P_{N'}\tilde{X}^N(s))-\phi(P_{N'}X^N_k)\right]\mathrm{d}s\\
        &+\frac{1}{n\eta'}\sum_{k=0}^{n-1}\int_{t_k}^{t_{k+1}}\bbE\langle P_{N'}\left(\nabla L'(\tilde{X}^N(s))-\nabla L'(P_{N'}\tilde{X}^N(s))\right),D\Psi^{N'}(P_{N'}\tilde{X}^N(s))\rangle\mathrm{d}s\\
        =:& I_1 + I_2 + I_3 + I_4 + I_5 + I_6.
    \end{split}
\end{align*}
As in \citet{Brehier16}, the fact that $\nabla L'$ is Lipschitz, \Cref{essential} and \Cref{moments2} yield
\begin{align*}
    \lim_{N' \rightarrow\infty}I_6 = 0,
\end{align*}
and \Cref{essential} and \Cref{moments2} yield for $0<\eta\leq\eta_0$ and $\beta\geq\eta_0$,
\begin{align*}
    |I_1+I_2|\leq \frac{C}{\lambda^* \beta n\eta'}(1+\norm{x_0}^2).
\end{align*}
The remaining three terms are controlled by the following lemmas, whose proofs we omit for the sake of conciseness. However, they can be shown by carefully tracing the proof line of \citet{Brehier16,Kopec} with the estimates in \Cref{essential}, \Cref{norm-bound} and \Cref{sec:MalliavinEvaluation}.
\begin{lemma}[The control of $I_3$; space discretization]
    \label{I3}
    For any $0<\kappa<1/2$ and $\eta_0$, there exists a constant $C>0$ such that for any $\phi\in C^2_b(\mathcal{H}),\: x\in\mathcal{H},\: \beta\geq\eta_0$ and $0<\eta\leq\eta_0$
    \begin{align}
        \begin{split}
            &\limsup_{N'\rightarrow\infty} \frac{1}{n\eta'}\sum_{k=0}^{n-1}\int_{t_k}^{t_{k+1}}\bbE\left(\mathcal{L}^{N'}-\mathcal{L}^N\right)\Psi^{N'}(\tilde{X}^N(s))\mathrm{d}s\\
            &\leq \frac{C}{\lambda^*}(1+\norm{x_0}^3)\norm{\phi}_{0,2} \cbeta  \mu_{N+1}^{1/2-\kappa}(1+(n\eta')^{-1}).
        \end{split}
    \end{align}
\end{lemma}
\begin{lemma}[The control of $I_4$; time discretization]
    \label{I4}
    For any $0<\kappa<1/2$ and $\eta_0$, there exists a constant $C>0$ such that for any $\phi\in C^2_b(\mathcal{H}),\: N'\in\mathbb{N},\: x\in\mathcal{H},\: \beta\geq\eta_0$ and $0<\eta\leq\eta_0$
    \begin{align*}
        \begin{split}
            &\left|\frac{1}{n\eta'}\sum_{k=0}^{n-1}\int_{t_k}^{t_{k+1}}\bbE\left(\mathcal{L}^N-\mathcal{L}^{\eta',k,N}\right)\tilde\Psi^{N'}(\tilde{X}^N(s))\mathrm{d}s \right|\\
            &\leq \frac{C}{\lambda^*}\norm{\phi}_{0,2}(1+\norm{x_0}^3)\cbeta \eta^{1/2-\kappa}(1+(n\eta')^{-1+\kappa}+(n\eta')^{-1}).
        \end{split}
    \end{align*}
\end{lemma}
\begin{lemma}[The control of $I_5$; more time discretization]
    \label{I5}
    For any $0<\kappa <1/4$ and $\eta_0$, there exists a constant $C,c'>0$ such that for any $\phi\in C^2_b(\mathcal{H}),\: N'\in\mathbb{N},\: x\in\mathcal{H},\: \beta\geq\eta_0$ and $0<\eta\leq\eta_0$
    \begin{align*}
        \begin{split}
            &\left|\frac{1}{n\eta'}\sum_{k=0}^{n-1}\int_{t_k}^{t_{k+1}}\bbE\left[\phi(P_{N'}\tilde{X}^N(t))-\phi(P_{N'}X^N_k)\right]\mathrm{d}t\right|\\
            &\leq C\norm{\phi}_{0,2} \cbeta \eta^{1/2-2\kappa}\left(1+\frac{\norm{x_0}}{(n\eta')^{1-\kappa}}\right).
        \end{split}
    \end{align*}
\end{lemma}
Putting them together, we get the main result (\Cref{lem:TimeAndSpaceErrorSum}).

\subsection{A Malliavin Integration by Parts Formula}\label{sec:MalliavinEvaluation}
In the proofs of \Cref{I3}, \ref{I4} and \ref{I5}, an integration by parts formula issued from Malliavin calculus is necessary to transform irregular stochastic integral terms into controllable ones; see \citet{Nualart,Sanz-Sole}.
Therefore, we restate the statement in this subsection.
The notations are the same as in \citet{Brehier16,debussche2011weak}.
\begin{lemma}
    Let $N' \in \mathbb{N}$. For any $G\in\mathbb{D}^{1,2}(\mathcal{H}_{N'}),\: u\in C^2_b(\mathcal{H}_{N'})$ and $\Psi\in L^2(\Omega \times [0, T],\mathcal{L}_2(\mathcal{H}_{N'}))$, an adapted process,
    \begin{align*}
        \bbE\left[Du(G).\int_0^T\Psi(s)\mathrm{d}W^{N'}(s)\right]=\bbE\left[\int_0^T \mathrm{Tr}(\Psi(s)^*D^2u(G)\mathcal{D}_sG)\mathrm{d}s\right],
    \end{align*}
    where $\mathcal{D}_s G:x \in\mathcal{H}\mapsto \mathcal{D}^x_sG\in\mathcal{H}_{N'}$ stands for th Malliavin derivative of $G$, and $\mathbb{D}^{1,2}(\mathcal{H}_{N'})$ is the set of $\cH_{N'}$-valued random variables $G=\sum_{i\leq N'}G_if_i$, with $G_i\in\mathbb{D}^{1,2}$ the domain of the Malliavin derivative for $\mathbb{R}$-valued random variables for any $i$.
\end{lemma}
In the proof of \Cref{I3}, \ref{I4} and \ref{I5}, we use the following estimates; see \cite{Brehier16,Brehier14,Kopec} for details.
\begin{lemma}
    For any $0\leq\gamma<1$ and $\eta_0>0$, there exists a constant $C>0$ such that for every $h\in(0,1),\: k\geq 1,\: 0<\eta\leq\eta_0,\: \beta>\eta_0$ and $s\in[t_k - 1/\beta,t_k]$
    \begin{align}
        \norm{(-A')^\gamma\mathcal{D}_s^x X^{N'}_k}_{\cH_N} 
\leq C(1+M\eta)^{k-l_s}\left(\beta^\gamma+\frac{1}{(1+\eta/\mu'_0)^{(1-\gamma)(k-l_s)}t^\gamma_{k-l_s}}\right) \|x\|_{\cH_{N'}},
    \end{align}
	for all $x \in \cH_{N'}$.
    Moreover, if $t_k\leq t<t_{k+1}$, we have
    \begin{align}
        \norm{(-A')^\gamma \mathcal{D}^x_s\tilde{X}^{N'}(t)}_{\mathcal{H}_{N'}}\leq C\norm{(-A')^\gamma \mathcal{D}^x_sX^{N'}_k}_{\mathcal{H}_{N'}},
    \end{align}
for $x \in \cH_{N'}$.
\end{lemma}
Note that the constant $C>0$ is uniform with respect to ${N'}\in\mathbb{N},\: \beta>\eta_0$.
\begin{proof}
    The proof is almost the same as that of Lemma 6.5 in \citet{Kopec}.
    
    The second inequality is a consequence of the following equality for $s\leq t_k\leq t<t_{k+1}$, thanks to \eqref{tilde-X}:
    \begin{align*}
        \mathcal{D}_s^x \tilde{X}^{N'}(t)=\mathcal{D}_s^x X_k^{N'}+(t-t_k)(A'\Setap\mathcal{D}_s^x X_k^{N'}-\Setap D(P_{N'}\nabla L')(X_k^{N'})\cdot \mathcal{D}_s^x X_k^{N'}),
    \end{align*}
    and the conclusion follows since
    \begin{align*}
        \sup_{{N'}\in\mathbb{N}}\norm{\eta'A'\Setap}_{\mathcal{B}(\mathcal{H}_{N'})}\leq C,
    \end{align*}
    where $C$ is a constant that does not depend on $\beta$ and the norm $\|\cdot\|_{\cB(\cH_{N'})}$ is taken as a linear map from $\cH_{N'}$ to $\cH_{N'}$.
    
    Then we prove the first estimate. For any $k\geq 1,\: x \in\mathcal{H}_N$, and $s\in[t_k - 1/\beta,t_k]$, we have
    \begin{align*}
        \mathcal{D}_s^x X_k^{N'}=\Setap^{k-l_s}x -\eta'\sum_{i=l_s+1}^{k-1}\Setap^{k-i}D(P_{N'}\nabla L')(X_i^{N'}).\mathcal{D}_s^x X_i^{N'}.
    \end{align*}
    We recall that $l_s=\floor{s/\eta'}$, so that when $i\leq l_s$ we have $\mathcal{D}_s^x X_i^{N'}=0$.
    
    As a consequence, the discrete Gronwall's inequality ensures that for $k\geq l_s+1$ and a constant $C > 0$,
    \begin{align*}
        \norm{\mathcal{D}^x_sX_k^{N'}}_{\cH_{N'}}\leq (1+ M \eta)^{k-l_s}\norm{x}_{\cH_{N'}},
    \end{align*}
where we used $\eta' L' = \eta L$ and the Lipchitz continuity of $\nabla L$.
    Now using \Cref{norm-bound}, we have
    \begin{align*}
        \norm{(-A')^{\gamma}\mathcal{D}^x_sX_k^{N'}}_{\cH_{N'}} \leq \frac{1}{(1+\eta/\mu'_0)^{(1-\gamma)(k-l_s)}t_{k-l_s}^{\gamma}}\norm{x}_{\cH_{N'}}+M \eta\sum_{i = l_s+1}^{k-1}\frac{(1+M\eta)^{i-l_s}}{(1+\eta/\mu'_0)^{(1-\gamma)(k-i)}t_{k-i}^\gamma}\norm{x}_{\cH_{N'}}.
    \end{align*}
    Note that $k - l_s \leq 1/(\eta' \beta) \leq 1/\eta$ yields $(1 + M\eta)^{k - l_s} \leq C$. To conclude, we see that when $0<\eta\leq\eta_0$, it holds that for a constant $c_0$ (could be dependent on $\eta_0,\mu_0'$),
    \begin{align*}
       &  \eta\sum_{i=l_s+1}^{k-1}\frac{1}{(1+\eta/\mu'_0)^{(1-\gamma)(k-i)}t_{k-i}^\gamma}
\leq \beta C\int_0^\infty \frac{t^{-\gamma}}{(1+\eta/\mu'_0)^{(1-\gamma)t/\eta'}}\mathrm{d}t \\
& \leq \beta C\int_0^\infty t^{-\gamma} \exp\left[ - c_0(1-\gamma)(t/\eta')( \eta/\mu'_0) \right]\mathrm{d}t \\
& \leq \beta C\int_0^\infty t^{-\gamma} \exp\left[ - \frac{\beta}{2} c_0(1-\gamma) t/\mu'_0 \right]\mathrm{d}t \\
& \leq  C \beta^\gamma.
\end{align*}
\end{proof}

\subsection{Proof of \texorpdfstring{\Cref{essential}}{Proposition A.9}}
\label{eq:ProofOfPsiNbound}
In this subsection, we prove \Cref{essential}. Our argument follows the same line as \citet{Brehier16}.
Let $\phi\in C^2_b(\cH)$. For lighter notation, we assume $\bar\phi = 0$ in this section. We define the function $u$ for any $t>0$ and $x\in\mathcal{H}_{N'}$ by
\begin{align}
    \label{u}
    u(t,x) = \bbE\left[\phi(\Xhat^{N'}(t,x))\right],
\end{align}
which is the solution of a finite dimensional Kolmogorov equation associated with \eqref{galerkin} where $N = N'$:
\begin{align*}
    \dv{u}{t}(t,x)=Lu(t,x)=\frac{1}{2}\mathrm{Tr}(D^2u(t,x))+\langle A'x - \nabla
 L'_{N'}(x),Du(t,x)\rangle.
\end{align*}
To prove \Cref{essential}, we only need to show that $u\in C^2$ and that $u$ and its two first derivatives have estimates which are integrable with respect to $t$. Specifically we prove the following proposition.
\begin{proposition}
    \label{u-bound}
    Let $\phi\in C^2_b$ such that $\bar\phi=0$ and $u$ defined by \eqref{u}. 
Remember that $\cbeta$ is defined in \Cref{eq:defcbeta} as $$\cbeta = \begin{cases} 1~~~& (\text{strict dissipativity condition: \Cref{assum:strict_diss})}), \\ \sqrt{\beta}~~~& (\text{bounded gradient condition: \Cref{assum:bounded_grad}}). \end{cases}$$ 
There exist constant $c,\: C>0$ such that for any $0\leq \epsilon, \gamma<1/2$ there exist constants $C_\epsilon$ and $C_{\epsilon,\gamma}$, which is independent of $\beta$, such that for any $t>0$ and $x\in\mathcal{H}_{N'}$,
    \begin{align}
        \label{bound0}
        \norm{u(t,x)}\leq Ce^{-\beta \lambda^* t}(1+\norm{x}^2)\norm{\phi}_\infty,
    \end{align}
    \begin{align}
        \label{bound1}
        \norm{(-A')^{\epsilon}Du(t,x)}\leq C_\epsilon \cbeta \beta^{\epsilon} (1+\frac{1}{(\beta t)^\epsilon})e^{-\beta \lambda^* t}(1+\norm{x}^2)\norm{\phi}_{0,1},
    \end{align}
    \begin{align}
        \label{bound2}
        \norm{(-A')^\epsilon D^2u(t,x)(-A')^\gamma}_{\mathcal{B}(\mathcal{H})}\leq C_{\epsilon,\gamma} \cbeta^2 \beta^{\epsilon + \gamma} \left(1+\frac{1}{(\beta t)^{\alpha'}}+\frac{1}{(\beta t)^{\epsilon+\gamma}}\right)e^{-\beta \lambda^* t}(1+\norm{x}^2)\norm{\phi}_{0,2},
    \end{align}
    where $\lambda^*>0$  is the spectral gap introduced in \Cref{rem:LambdaStarIdentification} (see also \Cref{prop:geometric_ergodicity_continuous_time}) and $\alpha'\in[0,1]$ is the constant introduced in \Cref{assum:C2_boundedness}.
\end{proposition}
    In fact the estimation \eqref{bound1} is true for $\alpha<1$.
    The proof is a slight modification of the proof of Proposition 8.1 in \cite{Kopec}. Since $\phi\in C^2$, bounded and with bounded derivatives, $u\in C^2$ and the derivatives can be calculated in the following way:
    \begin{itemize}
        \item For any $h\in\mathcal{H}_{N'}$, we have
              \begin{align}
                  \label{eta}
                  Du(t,x).h=\bbE\left[D\phi(\Xhat^{N'}(t,x)).\eta^{h,x}(t)\right],
              \end{align}
              where $\eta^{h,x}(t)$ is the solution of
              \begin{align*}
                  \dv{\eta^{h,x}(t)}{t}=A'\eta^{h,x}(t)-D^2 L'_{N'}(\Xhat^{N'}(t,x)).\eta^{h,x}(t),
              \end{align*}
              \begin{align*}
                  \eta^{h,x}(0)=h.
              \end{align*}
        \item For any $h,k\in\mathcal{H}_{N'}$, we have
              \begin{align}
                  \label{zeta}
                  D^2u(t,x).(h,k)=\bbE\left[D^2\phi(\Xhat^{N'}(t,x)).(\eta^{h,x}(t),\eta^{k,x}(t))+D\phi(\Xhat^{N'}(t,x)).\zeta^{h,k,x}(t)\right],
              \end{align}
              where $\zeta^{h,k,x}$ is the solution of
              \begin{align*}
                  \dv{\zeta^{h,k,x}}{t}=A'\zeta^{h,k,x}(t)-D^2 L'(\Xhat^{N'}(t,x)).\zeta^{h,k,x}(t)-D^3L'(\Xhat^{N'}(t,x)).(\eta^{h,x}(t),\eta^{k,x}(t)),
              \end{align*}
              \begin{align*}
                  \zeta^{h,k,x}(0)=0.
              \end{align*}
              Moreover, we already have the inequality \eqref{bound0} thanks to \Cref{cor:GeomErgCont}.
    \end{itemize}
The proof requires several steps. First in \Cref{finite-time} below we prove estimates for $0<t\leq 1/\beta$ and general $0\leq\alpha,\gamma<1/2$; then in \Cref{long-time} we study the long-time behavior in case $\alpha=\gamma=0$; we finally conclude with the proofs of \Cref{u-bound}.

\begin{lemma}
    \label{finite-time}
Assume \Cref{assum:bounded_grad} (bounded gradient condition). For any $0\leq \epsilon,\gamma <1/2$, there exist constants $C_\epsilon, C_{\epsilon,\gamma}$ such that for any $x\in\mathcal{H}_{N'}$, and any $0<t\leq 1/\beta$,
    \begin{align*}
         \norm{(-A')^{\epsilon}Du(t,x)} & \leq \frac{C_\epsilon}{t^\epsilon}\norm{D\phi}_\infty, \\
         \norm{(-A')^\epsilon D^2u(t,x)(-A')^\gamma}_{\mathcal{B}(\mathcal{H}_{N'})}& \leq C_{\epsilon,\gamma}\beta^{\epsilon + \gamma} \left(\frac{1}{(\beta t)^{\alpha'}}+\frac{1}{(\beta t)^{\epsilon + \gamma}}\right)\left(\norm{D\phi}_\infty+\norm{D^2\phi}_\infty\right),
    \end{align*}
    where $\alpha'$ is defined in \Cref{assum:C2_boundedness}.
\end{lemma}
\begin{proof}
    Owing to \eqref{eta} and \eqref{zeta}, we only need to prove the following almost sure estimates for some constants - which may vary from line to line below: for any $0<t\leq 1/\beta$
    \begin{align*}
        \begin{split}
            \norm{\eta^{h,x}(t)}&\leq \frac{C_\epsilon}{(\beta t)^\epsilon}\norm{h}_{\epsilon},\\
            \norm{\zeta^{h,k,x}(t)}&\leq C_{\epsilon, \gamma}\beta^{\epsilon + \gamma} \left(\frac{1}{(\beta t)^{\alpha'}} + \frac{1}{(\beta t)^{\epsilon + \gamma}}\right)\norm{h}_{\epsilon}\norm{k}_{\gamma}.
        \end{split}
    \end{align*}
    To show these inequalities, first note that 
\begin{align}
\norm{e^{tA'} h} & = \norm{t^{-\epsilon} (-t A')^{\epsilon}e^{tA'} (-A')^{-\epsilon}h}
= t^{-\epsilon} \norm{(-t A')^{\epsilon}e^{tA'}}_{\cB(\cH)} \norm{(-A')^{-\epsilon}h} \notag \\
& \leq t^{-\epsilon} \sup_{x \geq 0}\{x^\epsilon e^{-x}\}\norm{(-A')^{-\epsilon}h} = \frac{C_\epsilon}{t^\epsilon}\norm{(-A')^{-\epsilon}h}
\label{eq:etAhbound}
\end{align}
where $C_\epsilon \eqdef \sup_{x \geq 0}\{x^\epsilon e^{-x}\}$.
From this, we deduce that
    \begin{align*}
        \norm{\eta^{h,x}(t)} & = \norm{e^{tA'}h-\int^t_0 e^{(t-s)A'}D^2L'(\Xhat(s,x)).\eta^{h,x}(s)\mathrm{d}s}            \\
                             & \leq \frac{C_\epsilon}{t^\epsilon}\norm{(-A')^{-\epsilon}h}+C\int_0^t \beta \norm{\eta^{h,x}(s)}\mathrm{d}s.
    \end{align*}
    and by the Gronwall's inequality and $t \leq 1/\beta$, we get the result.
    
    For the second-order derivative, we moreover use the properties of $L$ to get 
    \begin{align*}
        \norm{\zeta^{h,k,x}(t)} = & \left|\int^t_0 e^{(t-s)A'}D^2L'(\Xhat(s,x)).\zeta^{h,k,x}(s)\mathrm{d}s\right.                                                                                                                          \\
                                  & \left.+\int^t_0 e^{(t-s)A'}D^3L'(\Xhat(s,x)).(\eta^{h,x}(s),\eta^{k,x}(s))\mathrm{d}s\right|                                                                                                            \\
        \leq                      & C\int_0^t \beta \norm{\zeta^{h,k,x}(s)}\mathrm{d}s + \int^t_0 \frac{C_{\alpha'}\beta^{1-\alpha'}}{(t-s)^{\alpha'}} \norm{\eta^{h,x}(s)}\norm{\eta^{k,x}(s)}\mathrm{d}s                                 \\
        \leq                      & C\int_0^t \beta \norm{\zeta^{h,k,x}(s)}\mathrm{d}s + C_{\alpha',\epsilon,\gamma}\norm{(-A')^{-\epsilon}h}\norm{(-A')^{-\gamma}k}\beta^{\epsilon + \gamma}(\beta t)^{1-\alpha'-\epsilon-\gamma}\int_0^1 \frac{1}{(1-s)^{\alpha'} s^{\epsilon+\gamma}}\mathrm{d}s.
    \end{align*}
    The Gronwall's inequality yields the conclusion
 since for any $0<\beta t\leq 1$ we have $(\beta t)^{1-\alpha'-\epsilon-\gamma}<(\beta t)^{-\alpha'}$ due to the assumption $\epsilon + \gamma < 1$. 
\end{proof}
\begin{lemma}
    \label{long-time}
	Assume \Cref{assum:bounded_grad} (bounded gradient condition). 
    There exist constants $C,c>0$ such that for any $t\geq 0$, and any $x\in\mathcal{H}$,
    \begin{align*}
        \norm{Du(t,x)}\leq C \sqrt{\beta}e^{-\beta \lambda^* t}(1+\norm{x}^2)\norm{\phi}_\infty,
    \end{align*}
    and
    \begin{align*}
        \norm{D^2u(t,x)}_{\mathcal{B}(\mathcal{H})}\leq C\beta e^{-\beta \lambda^* t}\left(1+\frac{1}{(\beta t)^{\alpha'}}\right)(1+\norm{x}^2)\norm{\phi}_\infty.
    \end{align*}
\end{lemma}
\begin{proof}[Proof of \Cref{long-time}]
    As in \citet{Kopec}, we use the Bismut-Elworthy-Li formula to get for $\Phi : \mathcal{H}_{N'} \rightarrow \mathbb{R}$ which belongs to class $C^2$ with bounded derivative and with at most quadratic growth, i.e.,
    \begin{align*}
        \exists\: M(\Phi)>0, \: \forall\: x\in\mathcal{H}_{N'},\: \norm{\Phi(x)}\leq M(\Phi)(1+\norm{x}^2),
    \end{align*}
    and $v(t,x)\triangleq\bbE\Phi(\Xhat^{N'}(t,x))$, we have two following formula:
    \begin{align*}
        \begin{split}
            Dv(t,x).h & =\frac{1}{t}\bbE\left[\int_0^t \langle \eta^{h,x}(s),\mathrm{d}W(s)\rangle\Phi(\Xhat^{N'}(t,x))\right].
		\end{split}
\end{align*}
Moreover, by the Markov property $v(t,x)=\bbE v(t/2,\Xhat^{N'}(t/2,x))$, we obtain
    \begin{align*}
        \begin{split}
           Dv(t,x).h & =\frac{2}{t}\bbE\left[\int_0^{t/2}\langle \eta^{h,x}(s),\mathrm{d}W(s)\rangle v(t/2,\Xhat^{N'}(t/2,x))\right].
        \end{split}
    \end{align*}
    and thus
    \begin{align*}
        \begin{split}
            D^2v(t,x).(h,k) = & \frac{2}{t}\bbE\left[\int_0^{t/2}\langle \zeta^{h,k,x}(s),\mathrm{d}W(s)\rangle v(t/2,\Xhat^{N'}(t/2,x))\right]                \\
            & +\frac{2}{t}\bbE\left[\int_0^{t/2}\langle \eta^{h,x}(s),\mathrm{d}W(s)\rangle Dv(t/2,\Xhat^{N'}(t/2,x)).\eta^{k,x}(t/2)\right].
        \end{split}
    \end{align*}
    We then see, using \Cref{moments1} and \Cref{finite-time} with $\epsilon=\gamma=0$ that there exists $C>0$ such that for any $0<t\leq 1/\beta,\: x,h,k\in\mathcal{H}_{N'}$,
    \begin{align}
        \label{v-bound}
        \begin{split}
            \norm{Dv(t,x).h}&\leq \frac{C}{\sqrt{t}}M(\Phi)(1+\norm{x}^2)\norm{h},\\
            \norm{D^2v(t,x).(h,k)}&\leq \frac{C}{t}M(\Phi)(1+\norm{x}^2)\norm{h}\norm{k}.
        \end{split}
    \end{align}
    Indeed, to see the first inequality, the Cauchy-Schwartz inequality gives 
\begin{align*}
Dv(t,x).h & =\frac{1}{t}\bbE\left[\int_0^t \langle \eta^{h,x}(s),\mathrm{d}W(s)\rangle\Phi(\Xhat^{N'}(t,x))\right] \\
& \leq \frac{1}{t} \sqrt{\bbE\left[\left(\int_0^t \langle \eta^{h,x}(s),\mathrm{d}W(s)\rangle \right)^2\right] } \sqrt{\bbE[\Phi(\Xhat^{N'}(t,x))^2]},
\end{align*}
and the isometry property of Ito integral and \Cref{finite-time} give a bound of the first term as 
$$
\sqrt{\bbE\left[\left(\int_0^t \langle \eta^{h,x}(s),\mathrm{d}W(s)\rangle \right)^2\right] } 
=
\sqrt{ \int_0^t \|\eta^{h,x}(s)\|^2 \dd s } \leq C \sqrt{t} \|h\|,
$$
for $t \leq 1/\beta$ and \Cref{moments1} gives a bound of the second term as 
$$
 \sqrt{\bbE[\Phi(\Xhat^{N'}(t,x))^2]} \leq C M(\Phi)(1 + \|x\|^2).
$$
    Now when $\beta t\geq 1$ the Markov property implies that $u(t,x)=\bbE[u(t-1/\beta,\Xhat^{N'}(1/\beta,x))]$ and by \Cref{cor:GeomErgCont}, we have
    \begin{align*}
        \norm{u(t-1/\beta,x)-\int_{\mathcal{H}_{N'}}\phi\mathrm{d}\bar\mu}\leq Ce^{- \beta\lambda^*(t-1/\beta)}(1+\norm{x}^2)\norm{\phi}_\infty.
    \end{align*}
    If we choose $\Phi_t(x)=u(t-1/\beta,x)-\int_{\mathcal{H}}\phi\mathrm{d}\bar\mu$, we have $u(t,x)=\bbE\Phi_t(\Xhat^{N'}(1/\beta,x))+\int_{\mathcal{H}}\phi\bar\mu$, with $M(\Phi_t)\leq Ce^{-\beta \lambda^*(t-1/\beta)}\norm{\phi}_\infty$. With \eqref{v-bound} at $t=1/\beta$, we obtain for $t\geq 1/\beta$,
    \begin{align*}
        \norm{Du(t,x).h}       & \leq C \sqrt{\beta} \norm{\phi}_\infty e^{-\beta \lambda^* (t-1/\beta)}(1+\norm{x}^2)\norm{h},        \\
        \norm{D^2u(t,x).(h,k)} & \leq C\beta \norm{\phi}_\infty e^{-\beta \lambda^*(t-1/\beta)}(1+\norm{x}^2)\norm{h}\norm{k}.
    \end{align*}
    We have a control when $0\leq t\leq 1/\beta$ in \Cref{finite-time}, so with a change of constants we get the result.
\end{proof}
Next we show a corresponding lemma for the strict dissipativity condition in the following lemma.
\begin{lemma}
\label{lemm:StrongDissUconverge}
Assume \Cref{assum:strict_diss} (strict dissipativity condition).
For any $0\leq \epsilon,\gamma <1/2$, there exist constants $C_\epsilon, C_{\epsilon,\gamma}$ such that for any $x\in\mathcal{H}_{N'}$, and any $0<t$,
    \begin{align*}
        \norm{(-A')^{\epsilon}Du(t,x)} & \leq C_\epsilon \beta^\epsilon  \left(1 + \frac{1}{(\beta t)^\epsilon}\right)  e^{- t \beta \lambda^*}  \norm{D\phi}_\infty, \\
        \norm{(-A')^\epsilon D^2u(t,x)(-A')^\gamma}_{\mathcal{B}(\mathcal{H}_{N'})}& \leq C_{\epsilon,\gamma}\beta^{\epsilon + \gamma} \left(1 + \frac{1}{(\beta t)^{\alpha'}}+\frac{1}{(\beta t)^{\epsilon + \gamma}}\right) e^{ -t\beta \lambda^*} \left(\norm{D\phi}_\infty+\norm{D^2\phi}_\infty\right),
    \end{align*}
    where $\alpha'$ is defined in \Cref{assum:C2_boundedness}.
\end{lemma}
\begin{proof}
From the definition of $\eta^{h,x}$, we have that
\begin{align*}
\norm{\eta^{h,x}(t)} & = \norm{e^{tA'}h-\int^t_0 e^{(t-s)A'}D^2L'(\Xhat^{N'}(s,x)).\eta^{h,x}(s)\mathrm{d}s}            \\
                  & \leq \norm{e^{tA'}h} +  \int^t_0 e^{-(t-s)\lambda/\mu_0} M \beta \norm{\eta^{h,x}(s)} \mathrm{d}s. 
\end{align*}
As in \Cref{eq:etAhbound}, for any $0 \leq c_0 < 1$, the first term can be bounded by 
\begin{align}
\norm{e^{tA'} h} & = \norm{t^{-\epsilon} (-t A')^{\epsilon}e^{c_0 tA'} (-A')^{-\epsilon}e^{(1 - c_0)t A'} h}
= t^{-\epsilon} \norm{(-t A')^{c_0 \epsilon}e^{tA'}}_{\cB(\cH)} \norm{(-A')^{-\epsilon}e^{(1 - c_0)t A'} h} \notag \\
& \leq t^{-\epsilon} \sup_{x \geq 0}\{x^\epsilon e^{-c_0 x}\}\norm{(-A')^{-\epsilon}h} = \frac{C_{\epsilon,c_0}}{t^\epsilon}\norm{(-A')^{-\epsilon} e^{(1 - c_0)t A'} h}
\notag 
\end{align}
where $C_{\epsilon,c_0} \eqdef \sup_{x \geq 0}\{x^\epsilon e^{-c_0 x}\}$.
Then, Gronwall's inequality gives 
\begin{align*}
& e^{ t \beta \lambda/\mu_0} \norm{\eta^{h,x}(t)} \leq 
\frac{C_{\epsilon,c_0}}{t^\epsilon}e^{c_0 t \beta \lambda/\mu_0}\norm{(-A')^{-\epsilon} h} + \int_0^t \beta M e^{s\beta \lambda/\mu_0 } \norm{\eta^{h,x}(s)} \dd s \\
\Rightarrow~
& e^{ t \beta \lambda/\mu_0} \norm{\eta^{h,x}(t)} \leq 
\frac{C_{\epsilon,c_0}}{t^\epsilon} e^{c_0 t \beta \lambda/\mu_0} \norm{(-A')^{-\epsilon} h} 
+ \int_0^t \beta M e^{s \beta \lambda/\mu_0 } \norm{\eta^{h,x}(s)} \dd s \\
\Rightarrow~  & e^{t \beta \lambda/\mu_0} \norm{\eta^{h,x}(t)} \leq 
\frac{C_{\epsilon,c_0}}{t^\epsilon} e^{c_0 t \beta \lambda/\mu_0} \norm{(-A')^{-\epsilon} h} + C_{\epsilon,c_0} 
\int_0^t \frac{e^{c_0 s \beta \lambda/\mu_0} }{s^\epsilon} \beta M \exp((t-s) \beta M)  \dd s 
\norm{(-A')^{-\epsilon} h} \\
& ~~~~~~~~~~~~~~~~~~~~~~~~~~~~~~~~~~~~\leq 
C_{\epsilon,c_0} \norm{(-A')^{-\epsilon} h}  \left[ \frac{1}{t^\epsilon} e^{c_0 t \beta \lambda/\mu_0}  
+  \beta^\epsilon M e^{t \beta M} \int_0^\infty \frac{e^{\tau (c_0 \lambda/\mu_0 - M) } }{\tau^\epsilon}  \dd \tau \right]
\\
\Rightarrow~ &
 \norm{\eta^{h,x}(t)} \leq  \frac{C_{\epsilon,c_0}}{t^\epsilon} \left(e^{-(1-c_0) t \beta \lambda/\mu_0 }+ (\beta t M)^\epsilon  \int_0^\infty \frac{e^{\tau (c_0 \lambda/\mu_0 - M) } }{(M \tau)^\epsilon}  M \dd \tau e^{- t \beta  (\lambda/\mu_0 - M) } \right).
\end{align*}
Therefore, if we choose $c_0$ as $c_0 = (\lambda/\mu_0)^{-1} M/2 $, 
then $0 \leq c_0 < 1$ by the strict dissipativity assumption 
and we obtain
\begin{align}
\norm{\eta^{h,x}(t)} 
&\leq C \frac{1+(t\beta M)^\epsilon}{t^\epsilon}  \exp[- t \beta (\lambda/\mu_0 - M)]  \norm{(-A')^{-\epsilon} h} \notag \\
&= C \frac{1+(t\beta M)^\epsilon}{t^\epsilon}  \exp[- t \beta \lambda^*]  \norm{(-A')^{-\epsilon} h}, \label{eq:etahkStDisBound}
\end{align}
where we used $\lambda^* = \lambda/\mu_0 - M~(>0)$.
Applying this to \Cref{eta}, we have the first inequality.

The second inequality is also shown in the same way as \Cref{finite-time}.
Notice that by the Lipschitz continuity of $\nabla L$, we have  
    \begin{align*}
        \norm{\zeta^{h,k,x}(t)} 
        \leq                      & \int_0^t e^{-(t-s)\beta \lambda/\mu_0}\beta M \norm{\zeta^{h,k,x}(s)}\mathrm{d}s + \int^t_0 \frac{C_{\alpha',c_0}\beta^{1-\alpha'}}{(t-s)^{\alpha'}} e^{-(1-c_0)(t-s)\beta\lambda/\mu_0}  \norm{\eta^{h,x}(s)}\norm{\eta^{k,x}(s)}\mathrm{d}s                                 \\
        \leq                      & \int_0^t e^{-(t-s)\beta\lambda/\mu_0} \beta M \norm{\zeta^{h,k,x}(s)}\mathrm{d}s \\
& + C'_{\alpha',\epsilon, \gamma}\norm{(-A')^{-\epsilon}h}\norm{(-A')^{-\gamma}k}
\beta^{1-\alpha'} \int_0^t  \frac{(1+(M\beta s)^{\epsilon + \gamma}) }{(t-s)^{\alpha'} s^{\epsilon + \gamma}} e^{- 2 s \beta (\lambda/\mu_0 - M)} 
e^{-(1-c_0)(t-s)\beta\lambda/\mu_0}
\dd s.
\end{align*}
From this inequality, we have 
\begin{align*}
& e^{t \beta\lambda/\mu_0}       \norm{\zeta^{h,k,x}(t)} \\
\leq                      & \int_0^t  \beta M  e^{s\beta\lambda/\mu_0}\norm{\zeta^{h,k,x}(s)}\mathrm{d}s \\
& + C_{\alpha',\epsilon, \gamma}\norm{(-A')^{-\epsilon}h}\norm{(-A')^{-\gamma}k}
\beta^{1-\alpha'} 
e^{t \beta\lambda/\mu_0 -t\beta \min\{2(\lambda/\mu_0 - M),(1-c_0)\lambda/\mu_0\}} 
\int_0^t  \frac{(1+(M\beta s)^{\epsilon + \gamma}) }{(t-s)^{\alpha'} s^{\epsilon + \gamma}} 
\dd s
\\
\leq & \int_0^t  \beta M  e^{s\beta\lambda/\mu_0}\norm{\zeta^{h,k,x}(s)}\mathrm{d}s \\
& + C_{\alpha',\epsilon, \gamma}\norm{(-A')^{-\epsilon}h}\norm{(-A')^{-\gamma}k}
\beta^{1-\alpha'} 
e^{t \beta\lambda/\mu_0 -t\beta \min\{2(\lambda/\mu_0 - M),(1-c_0)\lambda/\mu_0\}} 
\times \\
& (1+(M\beta t)^{\epsilon + \gamma}) t^{1-\alpha'-\epsilon - \gamma} \int_0^1  \frac{1 }{(1-\tilde{s})^{\alpha'} \tilde{s}^{\epsilon + \gamma}} 
\dd \tilde{s} \\
        \leq                     
         & \int_0^t  \beta M  e^{s\beta\lambda/\mu_0}\norm{\zeta^{h,k,x}(s)}\mathrm{d}s \\
& + C'_{\alpha',\epsilon, \gamma}\norm{(-A')^{-\epsilon}h}\norm{(-A')^{-\gamma}k}
\beta^{1-\alpha'} 
(1+(M\beta t)^{\epsilon + \gamma}) t^{1-\alpha'-\epsilon - \gamma}
e^{t \beta\lambda/\mu_0 -t\beta \min\{2(\lambda/\mu_0 - M), (1-c_0) \lambda/\mu_0\}}. 
\end{align*}
Here, we set $c_0 = (\lambda/\mu_0)^{-1}M/2$, 
then we further obtain
\begin{align*}
& e^{t \beta\lambda/\mu_0}       \norm{\zeta^{h,k,x}(t)} \\
\leq &     
 C'_{\alpha',\epsilon, \gamma}\norm{(-A')^{-\epsilon}h}\norm{(-A')^{-\gamma}k}
\beta^{1-\alpha'}  \times \\
& \Big( 
e^{t\beta M}
\int_0^t 
M \beta (1+(M\beta s)^{\epsilon+\gamma}) s^{1-\alpha'-\epsilon - \gamma}
e^{ - s \beta \min\{ \lambda/\mu_0 -M, M/2\}} \dd s
+ \\
& 
~~~~~(1+(M\beta t)^{\epsilon+\gamma}) t^{1-\alpha'-\epsilon - \gamma}
e^{t\beta M}e^{ - t \beta \min\{ \lambda/\mu_0 -M, M/2\}}
 \Big)
 \\
= &
 C'_{\alpha',\epsilon, \gamma}\norm{(-A')^{-\epsilon}h}\norm{(-A')^{-\gamma}k}
\beta^{1-\alpha'}  e^{t \beta M} \times \\
& ~~~
\left(
\int_0^t 
M \beta
(1+(M \beta s)^{\epsilon + \gamma}) s^{1-\alpha'-\epsilon - \gamma} e^{- s \beta \min\{ \lambda/\mu_0 -M, M/2\}} \dd s 
+ 
(1+(M\beta t)^{\epsilon + \gamma}) t^{1-\alpha'-\epsilon - \gamma}e^{- t \beta \min\{ \lambda/\mu_0 -M, M/2\}}
\right)
\\
\leq
& 
 C''_{\alpha',\epsilon, \gamma}\norm{(-A')^{-\epsilon}h}\norm{(-A')^{-\gamma}k}
\beta^{1-\alpha'}  e^{t \beta M}
 [\beta^{-(1-\alpha' - \epsilon - \gamma)} +(1+(M\beta t)^{\epsilon + \gamma})  t^{1-\alpha'-\epsilon - \gamma} e^{- t \beta \min\{ \lambda/\mu_0 -M, M/2\}}]
%
\end{align*}    
By multiplying both terms by $e^{-t\beta\lambda/\mu_0}$, we obtain 
\begin{align*}
  \norm{\zeta^{h,k,x}(t)}
& \leq     
 C'_{\alpha',\epsilon, \gamma}\norm{(-A')^{-\epsilon}h}\norm{(-A')^{-\gamma}k}
\beta^{\epsilon + \gamma}
 [1 +  (\beta t)^{1-\alpha'-\epsilon - \gamma}]
e^{ -t\beta (\lambda/\mu_0 - M)},
\end{align*}
where we used that 
$\sup_{t > 0} (1+(M\beta t)^{\epsilon + \gamma})  e^{- t \beta \min\{ \lambda/\mu_0 -M, M/2\}} < C$ (bounded by a constant independent of $\beta$).
Since $1 - \epsilon -\gamma > 0$ and $1 - \alpha' > 0$, 
it holds that 
$(\beta t)^{1-\alpha'-\epsilon - \gamma} \leq (\beta t)^{-\alpha'} + (\beta t)^{-\epsilon - \gamma}$.
Then, we finally obtain
\begin{align*}
\norm{\zeta^{h,k,x}(t)}
\leq     
 C'_{\alpha',\epsilon, \gamma}\norm{(-A')^{-\epsilon}h}\norm{(-A')^{-\gamma}k}
\beta^{\epsilon + \gamma} 
\left[ 1 + (\beta t)^{-\alpha'} +  + (\beta t)^{-\epsilon - \gamma} \right]
e^{ -t\beta \lambda^*},
\end{align*}
where we used $\lambda^* = \lambda/\mu_0 - M~(>0)$.
Applying this inequality and \Cref{eq:etahkStDisBound} to \Cref{zeta}, we obtain the second inequality.
\end{proof}

\begin{remark}
Note that \Cref{lemm:StrongDissUconverge} for the strict dissipativity condition does not require the restriction $t \leq 1/\beta$ while 
\Cref{finite-time} is for the bounded gradient condition.
This is advantageous to show better dependency on $\beta$ under the strict dissipativity condition than the bounded gradient condition.
\end{remark}

We can finally prove \Cref{u-bound}. The proof is again in line with \citet{Kopec}.
\begin{proof}[Proof of \Cref{u-bound}.]
First, we show the assertion for the bounded gradient condition.
    By the Markov property and \Cref{long-time}, for any $t\geq 1/\beta$, we have
    \begin{align*}
        \norm{Du(t,x).h} & \leq C\sqrt{\beta}\norm{\phi}_\infty e^{-\beta \lambda^*(t-1/\beta)}\bbE\left[(1+\norm{\Xhat^{N'}(1/\beta,x)}^2)\norm{\eta^{h,x}(1/\beta)}\right] \\
                         & \leq C\sqrt{\beta}\norm{\phi}_\infty e^{-\beta \lambda^*(t-1/\beta)}(1+\norm{x}^2)\beta^\epsilon \norm{(-A')^{-\epsilon}h},
    \end{align*}
    where the last estimate comes from \Cref{moments1} and \Cref{finite-time}. Combining this estimate and \Cref{finite-time}, we obtain \Cref{bound1}.
    We can easily see \Cref{bound2} follows from the similar argument.

As for the strict dissipativity condition, \Cref{lemm:StrongDissUconverge} directly gives the assertion.
\end{proof}


\section{Proof of SGLD convergence rate (\Cref{prop:SGLDdiscrepancy})}
\label{sec:ProofOfSGLD}
In this chapter, we prove \Cref{prop:SGLDdiscrepancy}.
Before that, we need to prepare the following lemmas to bound $\bbE [L(Y_k^N)-L(X_k^N)]$.
For lighter notation, our constants may differ from line to line.

\begin{lemma}
    \label{sg-error-bound}
    For any $x\in\mathcal{H}_N$, it holds that
    \begin{align*}
        \bbE \norm{\nabla L(x) - g_k(x)}^2 \leq \frac{C(\ntr-\nbch)}{\nbch(\ntr-1)},
    \end{align*}
    where $\nbch$ is the mini-batch size and $C>0$ is some constant.
\end{lemma}

We can prove the following bound similarly to \Cref{moments1} and \ref{moments2} thanks to \Cref{ass:SGLD_cond}. 
\begin{lemma}
    \label{moments3}
    For any $p\geq 1$, there exists a constant $C_p$ such that for every $N\in\mathbb{N}$, $\beta\geq \eta_0$, and $x\in\mathcal{H}_N$,
    \begin{align*}
        E\norm{Y_k^N}^p\leq C_p(1+\norm{x_0}^p).
    \end{align*}
\end{lemma}

\begin{lemma}
    \label{log-exp-bound}
    It holds that:
    \begin{align*}
        \begin{split}
            \exists\:& C_1,\:C_2 > 0,\: \forall\: \beta > \frac{2\mu'_0}{2+\eta/\mu'_0},\\
            & \log\bbE \left[\exp(\norm{X_k^N}^2)\right]\leq \norm{x_0}^2+C_1/\beta+C_2, 
        \end{split}
    \end{align*}
    where $C_1,C_2>0$ is an constant.
\end{lemma}
\begin{remark}
    Note that our estimate is not subject to ``the curse of dimensionality'' which explicitly appears in Lemma C.7 in \citet{Xu_Chen_Zou_Gu18}.
\end{remark}
\begin{proof}
    The proof is similar to that of Lemma C.7 in \citet{Xu_Chen_Zou_Gu18}. The main difference lies in the existence of regularizer in our scheme and the absence of dissipativity assumption of $L_N$. 
Instead, we assume \Cref{ass:SGLD_cond}. 

    Let $Q=\frac{2 \eta}{\beta}$ and $p_j = \frac{1}{(1+\eta/\mu'_j)^2}$.
    Let $S' := \diag((q_j)_{j=0}^N)$ for $q_j > 0~(j=0,\dots,N)$ and $1 > q_0 \geq q_1 \geq \dots \geq q_N$, then 
    we have
    \begin{align*}
    & \bbE\left[\exp \|X_{k+1}^{N}  \|_{S'}^2\right] \\
    = & \bbE\left[\exp \left\| S_\eta (X_{k}^{N} - \eta \nabla L_N(X_k^N) + \sqrt{\frac{2\eta}{\beta}} \epsilon_k^N  \right\|_{S'}^2\right],
    \end{align*}
    where $\epsilon_k^N \sim \mathcal{N}(0, I_N)$. Let $x_i, \epsilon_i$ denote the $i$-th component of $X_k^N-\eta\nabla L_N(X_k^N),$ and $\epsilon_k^N$ respectively, which corresponds to the coefficient of $f_i$, the $i$-th eigenfunction of $T_K$ defined in \Cref{eq:integral_operator}. Under this notation, we have the following estimate:

    \begin{align*}
        & \bbE \left[\exp\norm{S_\eta(X_k^N-\eta\nabla L_N(X_k^N)+\sqrt{\frac{2 \eta}{\beta}}\epsilon_k^N)}_{S'}^2\right]                                                                                                                                                          \\
        & =\bbE \left[\bbE \left[\exp\norm{S_\eta(X_k^N-\eta\nabla L_N(X_k^N)+\sqrt{\frac{2 \eta}{\beta}}\epsilon_k^N)}_{S'}^2 \biggl{|} X_k^N\right]\right]                                                                                                                  \\
        & =\bbE \left[\prod_{i=0}^N \int\exp\left( p_i q_i\left(x_i^2 + 2\sqrt{\frac{2 \eta}{\beta}}x_i\epsilon_i + \frac{2 \eta}{\beta}\epsilon_i^2\right) \right)\frac{1}{\sqrt{2\pi}}\exp\left(-\frac{\epsilon_i^2}{2}\right)\mathrm{d}\epsilon_i \right]          \\
        & =\bbE \prod_{i=0}^N\frac{1}{\sqrt{1-p_i q_i Q}}\exp\left(\frac{x_i^2}{\frac{1}{p_i q_i}-2Q}\right)                                                                                                                                  \\
        & \leq \exp\left( \sum_{j=0}^N  \frac{Q p_j q_j}{1 - 2 Q p_0 q_0 }   \right)\bbE \left[\exp\left(\sum_{i=0}^N\frac{x_i^2}{\frac{1}{p_i q_i}-2Q}\right)\right],
   \end{align*}
   thanks to the formula of Gaussian integral, $\mu'_0\geq \mu'_i$ and $\log(1-x)\geq -x/(1-x)$.

   Then, we have
    \begin{align*}
    & \exp\left( \sum_{j=0}^N  \frac{Q p_j q_j}{1 - 2 Q p_0 q_0 }   \right)\bbE \left[\exp\left(\sum_{i=0}^N\frac{x_i^2}{\frac{1}{p_i q_i}-2Q}\right)\right]\\
    &\leq \exp\left( \sum_{j=0}^N  \frac{Q p_j q_j}{1 - 2 Q p_0 q_0 }   \right)\bbE \left[\exp\left(\sum_{i=0}^N\frac{1}{\frac{1}{p_i q_i}-2Q}\left({X_{k,i}^N}^2 
    -2\eta X_{k,i}^N\nabla L_{N,i}(X_k^N) + \eta^2\nabla L_{N,i}(X_k^N)^2\right)\right)\right],
    \end{align*}
    where $X_{k,i}^N,\: L_{N,i}(X_k^N)$ denotes the $i$-th component of $X_k^N,\: L_N(X_k^N)$ respectively.

    Then \Cref{ass:SGLD_cond} implies
    \begin{align*}
    & \exp\left( \sum_{j=0}^N  \frac{Q p_j q_j}{1 - 2 Q p_0 q_0 }   \right)\bbE \left[\exp\left(\sum_{i=0}^N\frac{1}{\frac{1}{p_i q_i}-2Q}\left({X_{k,i}^N}^2 
-2\eta X_{k,i}^N\nabla L_{N,i}(X_k^N) + \eta^2\nabla L_{N,i}(X_k^N)^2\right)\right)\right],\\
    &\leq \exp\left( \sum_{j=0}^N  \frac{Q p_j q_j}{1 - 2 Q p_0 q_0 }   \right)\bbE \left[\exp\left(\sum_{i=0}^N\frac{1}{\frac{1}{p_i q_i}-2Q}\left({X_{k,i}^N}^2 
+2\eta B |X_{k,i}^N| + B^2 \eta^2\right)\right)\right],\\
    &\leq \exp\left( \sum_{j=0}^N  \frac{Q p_j q_j}{1 - 2 Q p_0 q_0 }   \right)\bbE \left[\exp\left(\sum_{i=0}^N\frac{1}{\frac{1}{p_i q_i}-2Q}\left((1+\kappa){X_{k,i}^N}^2+C\left(1+\frac{1}{\kappa}\right)\eta^2\right)
    \right)\right]                                                                                                                                                           \\
    &\leq \exp\left( \sum_{j=0}^N  \frac{Q p_j q_j}{1 - 2 Q p_0 q_0 }   \right)\bbE\left\{\exp\left[\sum_{i=0}^N\left( \frac{1+\kappa}{\frac{1}{p_iq_i} - 2Q}(X_{k,i}^N)^2 + 
    \frac{C p_i q_i (1 + \frac{1}{\kappa})}{1 -2 Q p_0 q_0} \eta^2 \right)  \right]\right\} \\
    & = \bbE\left\{\exp\left[ \|X_{k}^N\|_{S^{(k)}}^2  
    + \sum_{j=0}^N \left(  \frac{Q p_j q_j}{1 - 2 Q p_0 q_0 }
    + \frac{C p_j q_j (1 + \frac{1}{\kappa})}{1 -2 Q p_0 q_0} \eta^2 \right)  \right]\right\},
    \end{align*}
    where $S^{(k)} := \diag\left(\left(\frac{1+\kappa}{\frac{1}{p_jq_j} - 2Q}\right)_{j=0}^N\right)$.

    Since $q_0 \leq 1$, it holds that 
    $$
    \frac{Q p_j q_j}{1 - 2 Q p_0 q_0 } \leq \frac{Q p_j q_j}{1 - 2 Q p_0}.
    $$
    If we have chosen $\kappa$ so that $\frac{1 + \kappa}{\frac{1}{p_0 q_0} - 2Q} < 1$, then
    $$
    q_j^{(k)} := \frac{1+\kappa}{\frac{1}{p_jq_j} - 2Q} \leq 
    \frac{1+\kappa}{\frac{1}{p_0 q_0} - 2Q} < 1,
    $$
    and we also have $q_0^{(k)} \geq q_1^{(k)} \geq \dots \geq q_N^{(k)}$.
    Here again, since $q_j\leq 1$, it holds that
    $$
    q_j^{(k)} = \frac{1+\kappa}{\frac{1}{p_jq_j} - 2Q} \leq 
    \frac{1+\kappa}{ (p_j^{-1}-2Q)q_j^{-1} }.
    $$
    Let $\kappa = \frac{1}{2} (2 \eta/\mu_0' + (\eta/\mu_0')^2 - 2\eta/\beta)$, then
    it holds that 
    $$
    \frac{1+\kappa}{ p_j^{-1}-2Q}
    = \frac{1 + \frac{1}{2}[2 \eta/\mu_0' + (\eta/\mu_0')^2 - 2\eta/\beta]}{1 + 2 \eta/\mu_j' + (\eta/\mu_j')^2 - 2\eta/\beta}
    \leq \frac{1}{1+ \frac{1}{4}[2 \eta/\mu_j' + (\eta/\mu_j')^2 - 2\eta/\beta]} =: \frac{1}{1 + \alpha_j} < 1.
    $$
    Therefore, we obtain the following evaluation for $q_j^{(k)}$:
    $$
    q_j^{(k)} \leq \frac{q_j}{1 + \alpha_j},
    $$
    which implies $\norm{\cdot}_{S^{(k)}} \leq \norm{\cdot}_{S'}$.
    Hence, by noticing $p_j \leq (1+\alpha_j)^{-1}$, a recursive argument yields
    \begin{align*}
    \bbE\left[\exp\left(\|X_{k+1}^N\|_{S'}^2 \right)\right]
    \leq 
    \exp\left[\|x_0\|^2 +\frac{Q+C\left(1+1/\kappa\right)\eta^2}{1-2Q p_0q_0} \sum_{j=0}^N \sum_{i=0}^k \frac{q_j}{(1 + \alpha_j)^{k+1-i}} \right].
    \end{align*}
    The second term in the right hand side can be evaluated as 
    $$
    \sum_{i=0}^k\frac{1}{(1 + \alpha_j)^{k+1-i}} =  \frac{(1+\alpha_j)^{-1} - (1+\alpha_j)^{-(k+2)}}{1-(1+\alpha_j)^{-1}}
    \leq  \frac{1}{\alpha_j}. 
    $$
    Finally, if we set $q_j = 1$, then by observing that 
    $$
    \sum_{j=0}^N \frac{1}{\alpha_j} \lesssim 
    \int_1 ^\infty \frac{1}{\eta x^2}\dd x \lesssim \frac{1}{\eta},
    $$
    we have 
    $$
    \frac{Q+C\left(1+1/\kappa\right)\eta^2}{1-2Q p_0q_0} \sum_{j=0}^N \sum_{i=0}^k \frac{1}{(1 + \alpha_j)^{k+1-i}} q_j 
    \lesssim 
    \frac{Q+\left(1+1/\kappa\right)\eta^2}{\eta} = \frac{1}{\beta}+\left(1+\frac{1}{\alpha_0}\right)\eta.
    $$
    Since $\alpha_0 = O(\eta)$, the second term in the right hand side can be evaluated 
    $$
    \left(1+\frac{1}{\alpha_0}\right)\eta \lesssim 1.
    $$
    Combining all arguments, we obtain that 
    $$
    \bbE\left[\exp\left(\|X_{k+1}^N\|^2 \right)\right]
    \leq 
    \exp\left(\|x_0\|^2 + \frac{C_1}{\beta}+C_2\right).
    $$
\end{proof}

The following two lemmas are used to prove Theorem 3.6 in \cite{Xu_Chen_Zou_Gu18}. These results can only be applied to finite dimensional spaces. However, our schemes $Y_k^N, X_k^N$ are no longer infinite dimensional, which means we can follow the same argument in \cite{Xu_Chen_Zou_Gu18}.
\begin{lemma}[\citet{Polyanskiy,Raginsky_Rakhlin_Telgarsky2017,Xu_Chen_Zou_Gu18}]
    \label{wasserstein-bound}
    For any two probability density functions $\mu, \nu$ with bounded second moments, let $g:\mathbb{R}^d\rightarrow\mathbb{R}$ be a $C^1$ function such that
    \begin{align*}
        \norm{\nabla g(x)}_2\leq C_1\norm{x}_2+C_2,\: \forall\: x\in\mathbb{R}^d,
    \end{align*}
    for some constants $C_1,C_2\geq 0$. Then
    \begin{align*}
        \left|\int_{\mathbb{R}^d}g\mathrm{d}\mu-\int_{\mathbb{R}^d}g\mathrm{d}\nu\right|\leq (C_1\sigma+C_2)\mathcal{W}_2(\mu,\nu),
    \end{align*}
    where $\mathcal{W}_2$ is the 2-Wasserstein distance and $\sigma^2=\max\left\{\int_{\mathbb{R}^d}\norm{x}_2^2\mu(\mathrm{d}x),\int_{\mathbb{R}^d}\norm{x}_2^2\nu(\mathrm{d}x)\right\}$.
\end{lemma}

\begin{lemma}{(Corollary 2.3 in \citet{Bolley})}
    \label{kl-divergence-bound}
    Let $\nu$ be a probability measure on $\mathbb{R}^d$. Assume that there exist $x_0$ and a constant $\alpha>0$ such that $\int\exp(\alpha\norm{x-x_0}_2)\nu(\mathrm{d}x)< \infty$. Then for any probability measure $\mu$ on $\mathbb{R}^d$, it satisfies
    \begin{align*}
        \mathcal{W}_2(\mu,\nu)\leq C_\nu(\mathrm{D}(\mu||\nu)^{1/2}+\mathrm{D}(\mu||\nu)^{1/4}),
    \end{align*}
    where $C_\nu$ is defined as
    \begin{align*}
        C_\nu = \inf_{x_0\in\mathbb{R}^d,\alpha>0}\sqrt{\frac{1}{\alpha}\left(\frac{3}{2}+\log\int\exp(\alpha\norm{x-x_0}_2^2)\nu(\mathrm{d}x)\right)}.
    \end{align*}
\end{lemma}

\begin{proof}[Proof of \Cref{prop:SGLDdiscrepancy}]
    Let $P_k, Q_k$ denote the probability measures for GLD scheme $X_k^N$ and SGLD scheme $Y_k^N$ respectively. Applying \Cref{wasserstein-bound}, \Cref{moments3} and \Cref{moments2} yields
    \begin{align}
        |\bbE \left[L(Y_k^N)\right] - \bbE \left[L(X_k^N)\right]| \leq C(1+\norm{x_0})\mathcal{W}_2(Q_k, P_k),
    \end{align}
    where $C>0$ are absolute constants. We further apply \Cref{kl-divergence-bound} to bound Wasserstein distance and get the following bound:
    \begin{align}
        |\bbE \left[L(Y_k^N)\right] - \bbE \left[L(X_k^N)\right]| \leq C(1+\norm{x_0})\Lambda(\mathrm{D}(Q_k || P_k)^{1/2}+\mathrm{D}(Q_k || P_k)^{1/4}),
    \end{align}
    where $\Lambda=\sqrt{3/2+\log\bbE \left[\exp\norm{X_k^N}^2\right]}$.
    Moreover, \Cref{log-exp-bound} yields
    \begin{align}
        \Lambda \leq \sqrt{\frac{3}{2}+\norm{x_0}^2+\frac{C_1}{\beta}+C_2},
    \end{align}
    where $C_1,C_2>0$ is some constants. To bound KL-divergence between $P_k$ and $Q_k$, we use the following decomposition:
    \begin{align*}
        \mathrm{D}(Q_k || P_k) & \leq \mathrm{D}(Q_k || P_k) + \mathrm{D}(Q_{1:k-1} | Q_k || P_{1:k-1} | P_k)= \mathrm{D}(Q_{1:k} || P_{1:k}) \\
                               & = \mathrm{D}(Q_1 || P_1) + \sum_{i=2}^{k}\mathrm{D}(Q_i | Q_{1:i-1} || P_i | P_{1:i-1})                      \\
                               & = \sum_{i=1}^{k}\mathrm{D}(Q_i | Q_{i-1} || P_i | P_{i-1}),
    \end{align*}
    where $P_{1:k}, Q_{1:k}$ denotes joint distribution of $(X_1^N,\cdots, X_k^N)$ and $(Y_1^N,\cdots, Y_k^N)$ respectively and $Q_i | Q_{i-1}$ denotes the conditional distribution of $X_i^N$ given $X_{i-1}^N$.
    The first inequality is based on non-negativity of KL-divergence and the final equality comes from the fact that $Q_0, P_0$ are deterministic and that $X_i^N$ and $(X_1^N,\cdots, X_{i-2}^N)$ are conditionally independent given $X_{i-1}^N$. For clarity, we write down the definition of conditional KL-divergence in the following line:
    \begin{align*}
        \mathrm{D}(F_2 | F_1 || G_2 | G_1) = \int f(x_1, x_2)\log\frac{f(x_2 | x_1)}{g(x_2 | x_1)}\mathrm{d}x_1\mathrm{d}x_2.
    \end{align*}
    Now that $Q_i | Q_{i-1}$ and $P_i | P_{i-1}$ are both gaussian, that is,
    \begin{align*}
        X_i^N | X_{i-1}^N=x & \sim \mathcal{N}(S_\eta(x-\eta\nabla L_N(x)),\frac{\eta}{\beta}S_\eta^{T}S_\eta), \\
        Y_i^N | Y_{i-1}^N=x & \sim \mathcal{N}(S_\eta(x-\eta g_{i-1}(x)),\frac{\eta}{\beta}S_\eta^{T}S_\eta),
    \end{align*}
    we can calculate each conditional KL-divergence as below:
    \begin{align*}
         & \mathrm{D}(Q_i | Q_{i-1} || P_i | P_{i-1}) = \bbE _Q\left[\log\dv{Q_i|Q_{i-1}}{P_i|P_{i-1}}\right] \\
         & = \frac{\beta}{2\eta}\bbE _{(x, y)\sim Q_{i-1:i}}\left[
            \norm{S_\eta^{-1}y - (x-\eta\nabla L_N(x))}^2 - \norm{S_\eta^{-1}y - (x-\eta g_{i-1}(x))}^2
            \right]                                                                                                \\
         & = \frac{\beta}{2\eta}\bbE _{(x, y)\sim Q_{i-1:i}}\left[
        2\eta\langle S_\eta^{-1}y-x, \nabla L_N(x)-g_{i-1}(x)\rangle
        + \eta^2(\norm{\nabla L_N(x)}^2 - \norm{g_{i-1}(x)}^2)
        \right]                                                                                                    \\
         & = \frac{\beta}{2\eta}\bbE _{x\sim Q_{i-1}}\left[
        2\eta\langle -\eta g_{i-1}(x), \nabla L_N(x)-g_{i-1}(x)\rangle
        + \eta^2(\norm{\nabla L_N(x)}^2 - \norm{g_{i-1}(x)}^2)
        \right]                                                                                                    \\
         & = \frac{\beta\eta}{2}\bbE _{x\sim Q_{i-1}}\left[
            \norm{\nabla L_N(x) - g_{i-1}(x)}^2
            \right]                                                                                                \\
         & \leq \frac{\beta\eta}{2}\bbE _{x\sim Q_{i-1}}\left[\frac{C(\ntr-\nbch)}{\nbch(\ntr-1)}\right]                    \\
         & \leq C\frac{\beta\eta(\ntr-\nbch)}{\nbch(\ntr-1)},
    \end{align*}
    thanks to \Cref{sg-error-bound} and \ref{moments3}. Therefore, we finally get the following bound:
    \begin{align}
        \mathrm{D}(Q_k || P_k) & \leq C\frac{\beta\eta k(\ntr-\nbch)}{\nbch(\ntr-1)}.
    \end{align}
    Combining all of the above yields the claim.
\end{proof}

\end{document}